\documentclass[a4paper,11pt]{amsart} 
\usepackage{amscd,amssymb}
\usepackage{amsfonts}
\usepackage{graphicx,epsf}
\usepackage{enumerate,color}
\usepackage[mathcal]{euscript}
\usepackage[dvipsnames]{xcolor}
\usepackage{hyperref}

\numberwithin{equation}{section}
\usepackage{epstopdf}
\usepackage{epsfig}

\newtheorem{theorem}{Theorem}[section]

\newtheorem{proposition}[theorem]{Proposition}
\newtheorem{corollary}[theorem]{Corollary}
\newtheorem{definition}[theorem]{Definition}

\newtheorem {problem}[theorem]{Inverse Problem}

\newcommand{\R}{\mathbb{R}}
\newcommand{\C}{\mathbb{C}}
\newcommand{\Z}{\mathbb{Z}}

\newcommand{\real}{\mathrm{Re}}
\newcommand{\Ei}{\mathrm{Ei}}

\newcommand{\supp}{\mathrm{supp}}
\newcommand{\ND}{\mathcal{R}}
\newcommand{\expikz}{e^{ikz}}

\newcommand{\expiconjkz}{e^{i\bar{k}\bar{z}}}

\DeclareMathOperator{\parphi}{\widetilde{\varphi}}
\newcommand{\Hmhalf}{\widetilde{H}^{-1/2}(\bndry)}

\newcommand{\parspace}{\widetilde{H}^{-1/2}_\Gamma(\partial\Omega)}
\DeclareMathOperator{\parI}{\mathcal{I}}

\DeclareMathOperator{\dbar}{\overline{\partial}}

\DeclareMathOperator{\dnu}{\partial_\nu}

\DeclareMathOperator{\bndry}{ {\partial\Omega} }
\DeclareMathOperator{\T}{\mathbf{t}}

\begin{document}
\title[Direct Inversion from partial-boundary data]{Direct inversion from \\ partial-boundary data in \\ electrical impedance tomography}
\author[A. Hauptmann, M. Santacesaria, and S. Siltanen]{}
\date{\today}

\maketitle

\centerline{\scshape Andreas Hauptmann$^{\rm a}$, Matteo Santacesaria$^{\rm b}$, and Samuli Siltanen$^{\rm a}$}
\medskip
{\footnotesize \noindent$^{\rm a}$Department of Mathematics and Statistics, University of Helsinki, Helsinki, Finland \\
\noindent$^{\rm b}$Department of Mathematics "Francesco Brioschi",Politecnico di Milano, Milano, Italy}
\begin{abstract}
\noindent  
In Electrical Impedance Tomography (EIT) one wants to image the conductivity distribution of a body from current and voltage measurements carried out on its boundary. In this paper we consider the underlying mathematical model, the inverse conductivity problem, in two dimensions and under the realistic assumption that only a part of the boundary is accessible to measurements. In this framework our data are modeled as a partial Neumann-to-Dirichlet map (ND map). We compare this data to the full-boundary ND map and prove that the error depends linearly on the size of the missing part of the boundary. The same linear dependence is further proved for the difference of the reconstructed conductivities -- from partial and full boundary data. The reconstruction is based on a truncated and linearized D-bar method.
Auxiliary results include an extrapolation method to estimate the full-boundary data from the measured one, an approximation of the complex geometrical optics solutions computed directly from the ND map as well as an approximate scattering transform for reconstructing the conductivity. Numerical verification of the convergence results and reconstructions are presented for simulated test cases.
\end{abstract}

\section{Introduction}
In electrical impedance tomography (EIT) a body is probed with an electrical current to obtain information about the inner conductivity distribution. In this application full-boundary measurements are not always possible. This is especially true in three-dimensional medical imaging; but even in two dimensions we may be unable to access parts of the boundary. For example, when monitoring an unconscious patient in an intensive care unit, we usually have only access to the front part of the patient's chest. These limitations motivate the study of the inverse conductivity problem with partial-boundary data.

In this {work} we extend the theory of direct reconstructions by the D-bar method from full-boundary data to partial-boundary measurements. {Furthermore, we are interested in the error that is introduced to the data and to the reconstructed conductivity by restricting the measurement to a part of the boundary.}\smallskip	

We consider a two-dimensional bounded domain $\Omega\subset\R^2$ to which a current $f$ is injected from part of the boundary $\Gamma\subset\bndry$. The problem of EIT can then be modelled by the conductivity equation with Neumann boundary conditions
\begin{equation}\label{eq:cond_intro}
\begin{array}{rl}
\nabla\cdot\sigma\nabla u \ = &0, \quad \mbox{ in }\Omega,\\
 \sigma \frac{\partial u}{\partial\nu}\  =&f, \quad  \mbox{ on }\Gamma\subset\bndry, \\
 \sigma \frac{\partial u}{\partial\nu}\  =&0, \quad  \mbox{ on }\Gamma^c=\bndry\backslash\Gamma.
\end{array}
\end{equation} 
For uniqueness we assume that the solutions $u$ satisfy $\int_{\bndry} u ds=0$ and due to conservation of charge $\int_{\Gamma} f ds=0$. We are interested in recovering the conductivity $\sigma$ from boundary measurements, i.e. the trace $\left. u\right|_{\bndry}$, under given current patterns $f$. This measurement is modelled by the Neumann-to-Dirichlet, or current-to-voltage map, that {associates} every possible current pattern with the corresponding voltage on the boundary. Given a current on the full boundary $\varphi\in \widetilde{H}^{-1/2}(\bndry)$ (the space of $H^{-1/2}$ functions with zero mean on $\partial\Omega$) then the Neumann-to-Dirichlet map (ND map) is given by the operator
\[
\ND_\sigma:\widetilde{H}^{-1/2}(\bndry) \to H^{1/2}(\bndry),  \hspace{0.25 cm} \left.\ND_\sigma\varphi=u\right|_{\bndry}.
\]
For the computational reconstruction of $\sigma$ one ideally wants to represent the ND map with respect to an orthonormal basis on the full boundary. But if the currents are only supported on a part of the boundary, this is not directly possible. For this reason we introduce a {partial ND map}. Let us first consider a linear and bounded operator $\parI$ from $\widetilde{H}^{-1/2}(\partial\Omega)$  to a subspace of functions supported only on $\Gamma$ 
\begin{equation}\label{eqn:Hgamma}
{H}^{-1/2}_\Gamma(\partial\Omega):=\{\varphi\in \widetilde{H}^{-1/2}(\bndry) : \supp( \varphi)=\Gamma \text{ and } \int_\Gamma \varphi = 0\}.
\end{equation}
Then the {partial ND map} is defined as the composition
\[
\widetilde{\ND}_\sigma:=\ND_\sigma\parI,
\]
with the mapping properties $\widetilde{\ND}_{\sigma} :\widetilde{H}^{-1/2}(\partial\Omega)\xrightarrow{\parI}\widetilde{H}^{-1/2}_\Gamma(\partial\Omega) \xrightarrow{\ND_\sigma} H^{1/2}(\bndry)$. By this formulation we are able to interpret the measurements and we can represent the ND map with respect to an orthonormal basis, as discussed in Section \ref{sec:partialData}. 

In our main result, Proposition \ref{prop:convergence_single}, we analyse the error of measured traces from the partial ND map compared to the full ND map. 
In particular we choose the basis functions to be $\varphi_n(\theta)=\frac{1}{\sqrt{2\pi}} e^{in\theta}$, and $h=|\Gamma^c|$ be sufficiently small. Then for some constant $C = C (n)>0$, with possible dependence on $n$, the following error estimate holds
\begin{equation}\label{eqn:dataError_Intro}
\|(\widetilde{\ND}_\sigma-\ND_\sigma)\varphi_n\|_{L^2(\partial\Omega)}\leq C h.
\end{equation}
In the case that $C$ is independent of $n$ we prove, in Theorem \ref{theo:ReconError}, linear dependence of the reconstruction error when using the {partial ND map}. That is, given a truncation radius $R>0$ in the scattering transform, the {(truncated) reconstructions} $\sigma_R$ from full-boundary data and $\widetilde{\sigma}_R$ from partial-boundary data, then we have for some $C = C(R)>0$ that
\begin{equation}\label{eqn:reconError_Intro}
\|\widetilde{\sigma}_R-\sigma_R\|_{L^2(\Omega)} \leq Ch.
\end{equation}

We discuss shortly the difference of full-boundary and partial-boundary data as well as the underlying physical behaviour. In the full-boundary case, i.e. $\Gamma=\bndry$, the Dirichlet-to-Neumann data are equivalent to the Neumann-to-Dirichlet data. This is not true any more for partial-boundary data, where the graphs of the two operators represent different subsets of the Cauchy data. This can be further emphasized by the fact that the partial-boundary Dirichlet problem is nonphysical in many applications. That is, given a noninsulating body (e.g. a human), applied voltages on a part of the boundary will immediately distribute to the full boundary. Thus, partially supported Dirichlet data do not represent a common physical problem. On the other hand, if one injects current only on a subset $\Gamma\subset\bndry$, the current will stay zero on $\bndry\backslash\Gamma$. We stress that even in this setting the resulting voltage distribution will be supported on the full-boundary.
This is an essential problem for the measurements: we need in our representation the measurement information on $\bndry$. This limitation is overcome by an extrapolation {procedure} of the measured data, as we will discuss in Section \ref{sec:extrapolation}.

{Theoretical results for the inverse conductivity problem with partial boundary data mainly concentrate} on the uniqueness question. That means, does it follow from infinite-precision data that the conductivities are equal? Uniqueness has been proved in several cases for the Dirichlet-to-Neumann problem, including the important works \cite{Bukhgeim2002,DosSantosFerreira2007,Imanuvilov2010,Isakov2007,Kenig2007}. A thorough survey of these results can be found in \cite{Kenig2014}. For the more physical Neumann-to-Dirichlet problem there are just a few uniqueness results published. In particular for $C^2$ conductivities and coinciding measurement and input domains in $\R^2$ by \cite{Imanuvilov2012}, in higher dimensions in \cite{Harrach2016}, and for different input and measurement domains in \cite{Chung2015}. A more pratical case with bisweep data has been addressed in \cite{Hyvoenen2012}.
We would like to note that these results are of great importance for the theoretical understanding, but are so far not readily applicable for the computational reconstruction task. Furthermore, given only (very limited) finite data uniqueness can not be guaranteed any more, as demonstrated for the point electrode model in \cite{Chesnel2015}. A constructive uniqueness proof has been published by Nachman and Street \cite{Nachman2010} for dimension $n\geq3$. The proof is based on the Dirichlet-to-Neumann problem as well and hence does not apply to the problem at hand. A stability estimate was established in \cite{Heck2006}, showing log-log-type stability of the partial-data problem, in contrast to log-type stability for the full-boundary case \cite{Alessandrini1988}.

Reconstruction algorithms can be roughly divided into two classes: direct and indirect methods. Algorithms based on direct inversion are closely related to theoretical studies and demand a deep {understanding of} the mathematical structure of the problem. An investigation on direct inversion from partial-boundary data has been done in \cite{Hamilton2014a} based on the D-bar method by utilizing localized basis functions (Haar wavelets) to recover the complex geometric optics (CGO) solutions. Another direct approach is complex spherical probing with localized boundary measurements \cite{Ide2010,Ide2007}.

On the other side, indirect approaches for the partial-boundary problem are more common and perform very well in reconstruction quality, but tend to be slow. Typically those approaches consist in minimizing a carefully chosen penalty functional, which is based on a thorough understanding of physical aspects of the imaged target. In this category there are many algorithms available. We mention a few that are of importance in our perception. Those include reconstruction algorithms based on sparsity priors for simulated continuum data \cite{Garde2015} and planar real measurements \cite{Gehre2014}. Algorithms based on the complete electrode model \cite{Cheng1989,Somersalo1992}, which takes contact impedances at the electrodes into account, include domain truncation approaches \cite{Calvetti2015,Calvetti2015a,Liu2015a}, difference imaging \cite{Liu2015}, and electrode configurations that cover only a certain part of the boundary \cite{Mueller1999,Vauhkonen1999a}. In particular the complete electrode model is a partial-boundary problem, for which Hyv\"onen  \cite{Hyvoenen2009} proved linear dependence of the data error on the maximal electrode distance, similar to our estimate \eqref{eqn:dataError_Intro}.

\begin{figure}[t!]
\centering
\begin{picture}(300,105)
\put(-100,-25){\includegraphics[width=180 pt]{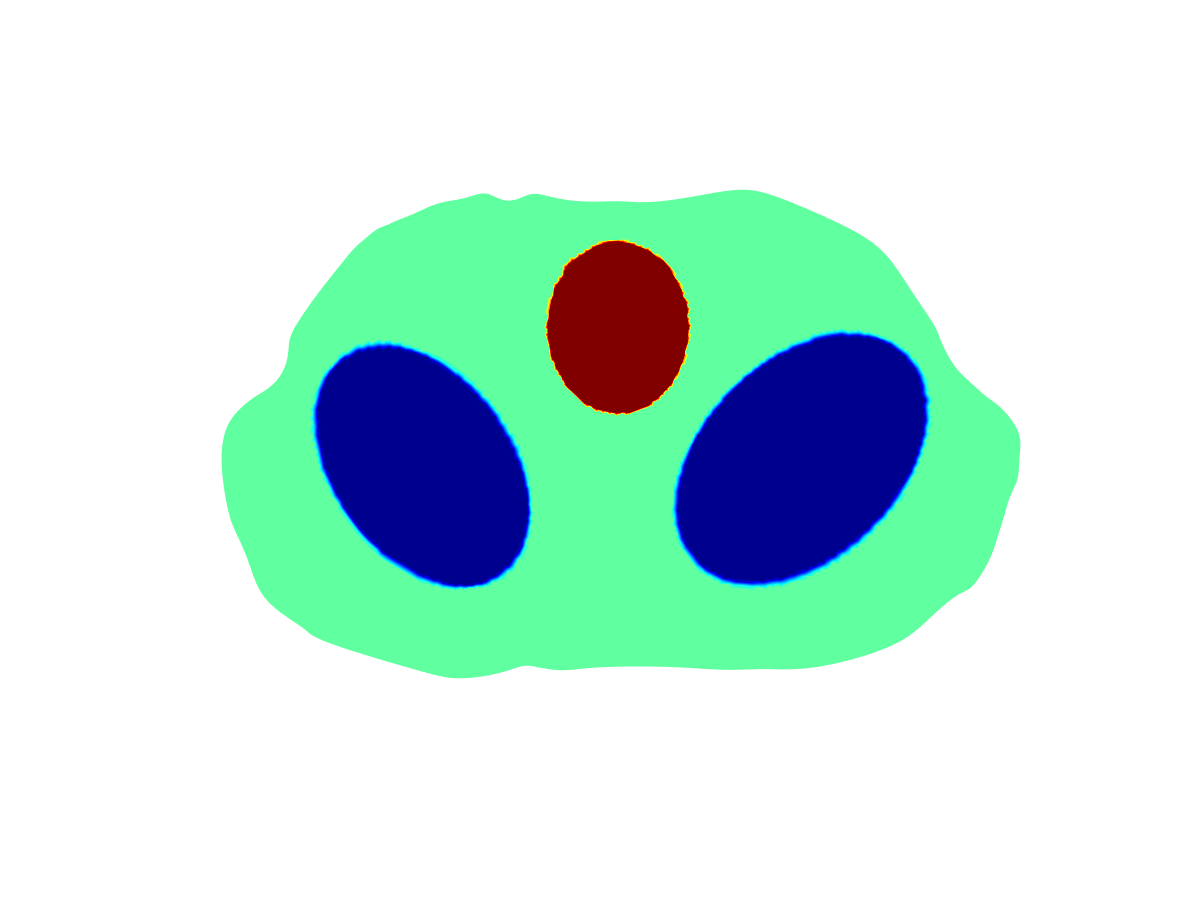}}
\put(60,-25){\includegraphics[width=180 pt]{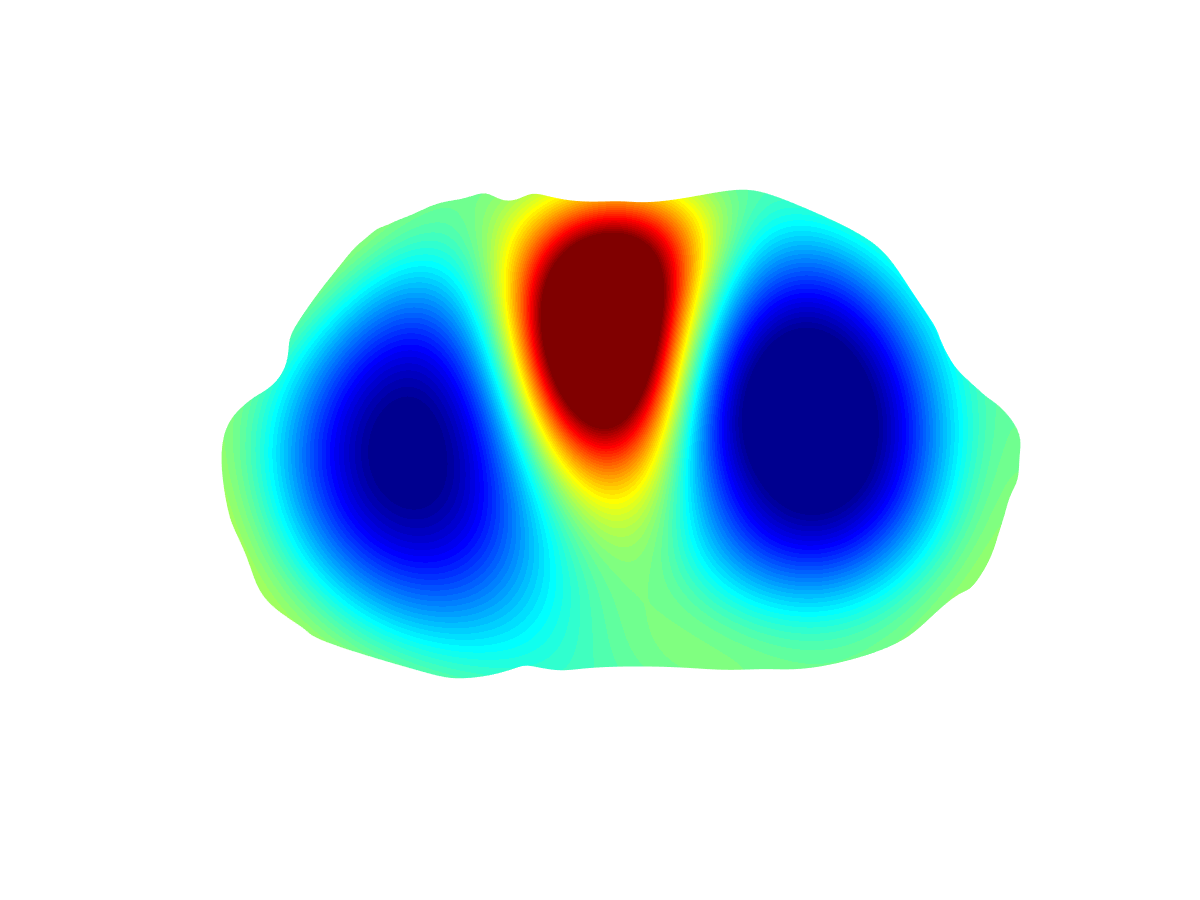}}
\put(220,-25){\includegraphics[width=180 pt]{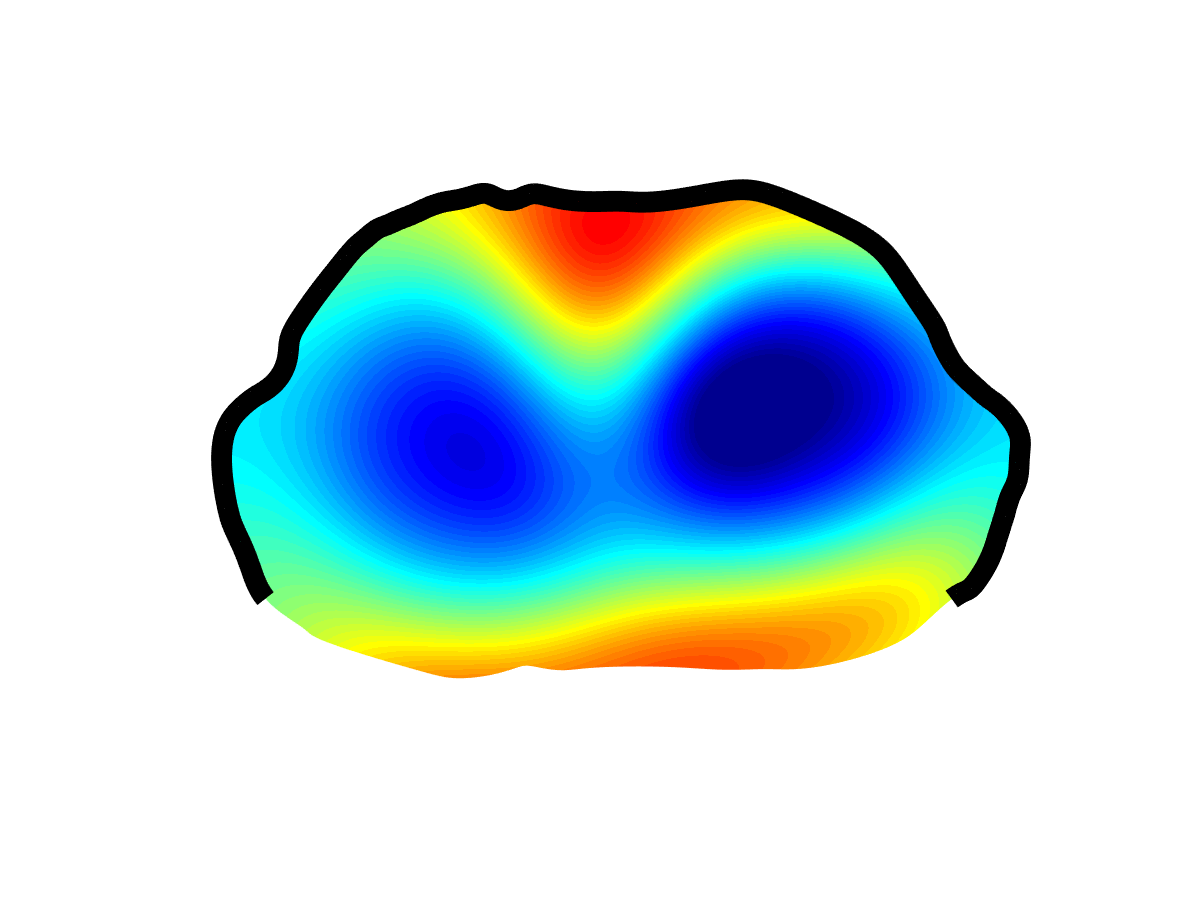}}
\put(-35,95){\Large Phantom}
\put(110,101){\Large Full boundary}
\put(110,87){\Large reconstruction}
\put(260,101){\Large 75\% of boundary}
\put(267,87){\Large reconstruction}
\end{picture}
\caption{\label{fig:HammerHead} Reconstructions of a human chest phantom from simulated data. In the middle reconstruction from full-boundary data and on the right the reconstruction from measurements on 75\% of the boundary. The measurement domain is indicated by the black line. Both reconstructions are plotted with the same colorscale. The quality of the reconstruction from partial-boundary data seems to be sufficient for detecting the collapse of a lung, for example.}
\end{figure}
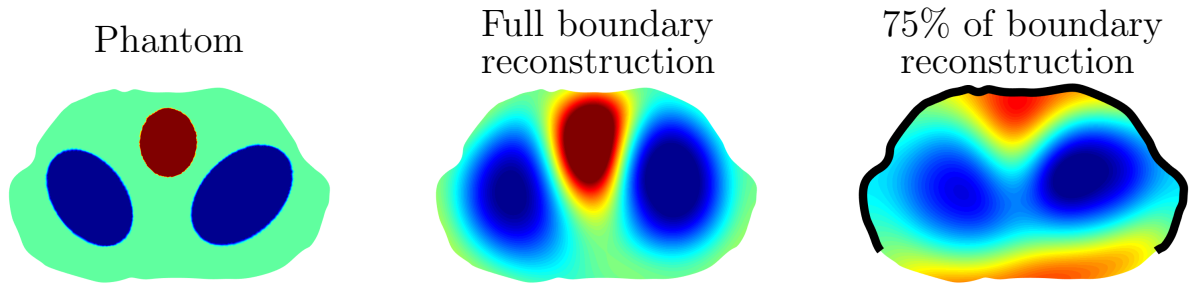

In this paper we use as reference the D-bar algorithm by Knudsen et al. \cite{Knudsen2009} based on results of Novikov \cite{Novikov1988} and Nachman \cite{Nachman1996}, see also \cite{Knudsen2007,Siltanen2000}. This approach is heavily dependent on the Dirichlet-to-Neumann map (DN map). As we discussed this is not a problem for full-boundary measurements (essentially equivalent to the ND map), but it is for the partial-boundary case. Therefore, we need some adjustments when having only the {partial ND map} available. In Section \ref{sec:partialData} we carefully define the {partial ND map} and derive a representation by boundary layer potentials. In Section \ref{sec:projectionOp} we define the operator $\parI{: \widetilde{H}^{-1/2}(\partial\Omega) \to \widetilde{H}_{\Gamma}^{-1/2}(\partial\Omega)}$ and then we establish the error estimate \eqref{eqn:dataError_Intro} of the measured traces from the {partial ND map} to the full ND map. 

In Section \ref{sec:DbarND} we derive equations for the CGO solutions and the scattering transform in case we have only the ND map available. 
In perspective of the error estimate, {we treat the partial ND map as a noisy perturbation of the full ND map}. 
The resulting integrals are evaluated by applying a Born approximation as in \cite{Siltanen2000}. For practical truncation radii this approach is known to differ only minimally from the full nonlinear one. As a theoretical conclusion we prove the linear dependence of the reconstruction error \eqref{eqn:reconError_Intro} in Theorem \ref{theo:ReconError}.

In the computational Section \ref{sec:computationalResults} we demonstrate that the error estimates hold numerically and present reconstructions for a simple circular inclusion and a Heart-and-Lungs phantom on the unit disk. We also present a more realistic chest phantom on a non-circular domain, see Figure \ref{fig:HammerHead}. For better readability, a short discussion is directly presented with the computations. The results are then followed by our conclusions in Section \ref{sec:conclusion}.

\textit{Notation.} Throughout the paper, $C(\alpha,\beta,\ldots)$ is a positive constant depending on parameters $\alpha, \beta, \ldots$.

\section{Partial-boundary measurements and data extrapolation} \label{sec:partialData}
In this chapter we introduce our setting of {the inverse problem for electrical impedance tomography with continuum data supported on a part of the boundary}. We will derive a formulation of the {partial ND map} that allows us to represent the measured data and analyse the error we are doing in comparison to full-boundary data. 

\begin{figure}[ht!]
\centering
\begin{picture}(200,125)
\put(0,-20){\includegraphics[width=200 pt]{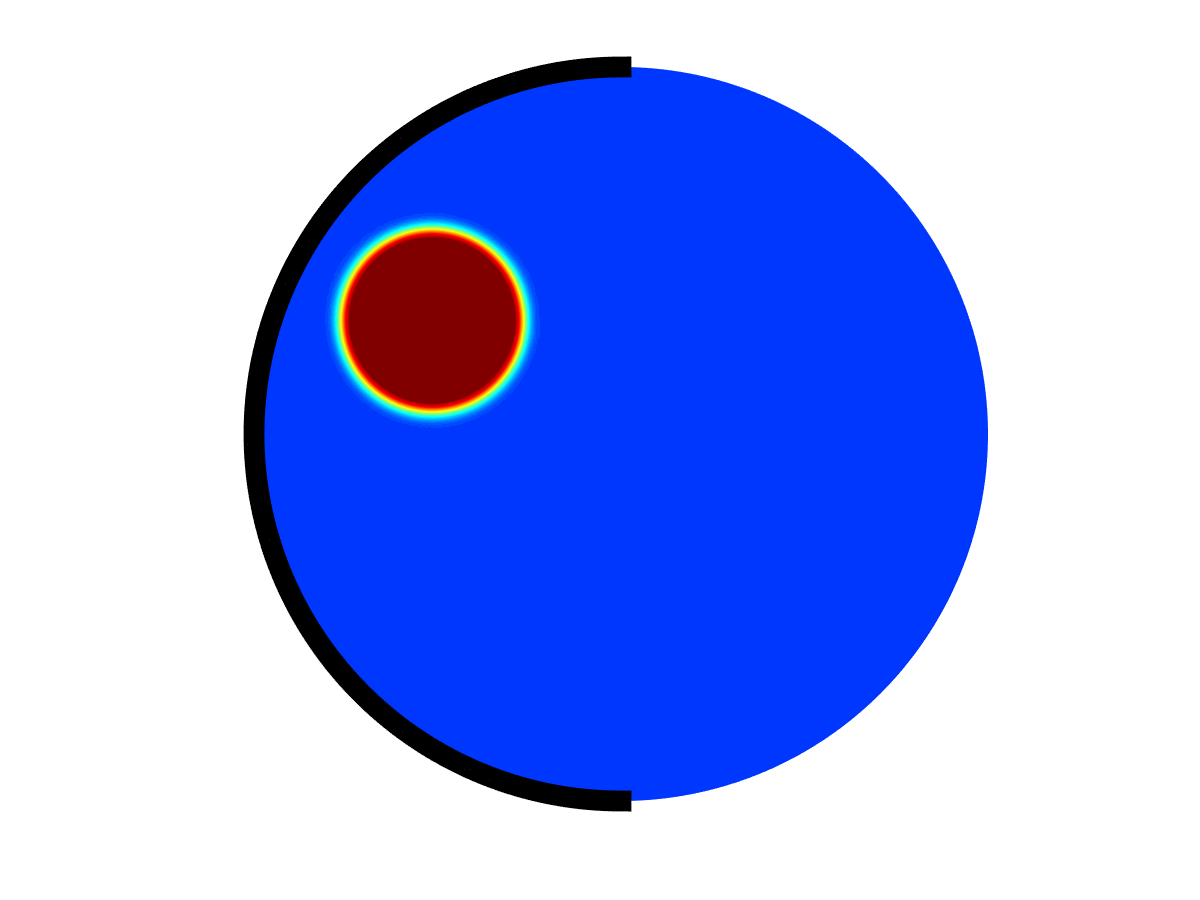}}
\put(30,75){$\Gamma$}
\end{picture}
\caption{\label{fig:GammaCircularCond}Illustration of the subset $\Gamma\subset\bndry$ in black with $\Omega$ the unit disk and a simple conductivity with circular inclusion.}
\end{figure}

Given our domain of interest $\Omega\subset\R^2$, we inject a current $f$ with zero mean on part of the boundary $\Gamma\subset\bndry$. The setting of EIT can be modelled by the conductivity equation with Neumann boundary condition \eqref{eq:cond_intro}.
For the Neumann data we introduced in \eqref{eqn:Hgamma} the space $\widetilde{H}^{-1/2}_\Gamma(\partial\Omega)\subset \widetilde{H}^{-1/2}(\bndry)$ of functions only supported on $\Gamma$. 
Let $\parphi\in\parspace$ with $\parphi|_\Gamma=f$. Then we can rewrite the boundary condition in \eqref{eq:cond_intro} as
\begin{equation}
\label{eq:PartCalderon}
\sigma \frac{\partial u}{\partial\nu}  = \widetilde{\varphi}, \quad  \mbox{ on } \bndry.
\end{equation}
From this boundary condition we can model the measurement process, given a partial current pattern $\parphi$, the ND map is defined as
\[
\ND_\sigma \parphi = u|_{\bndry}.
\]
The resulting voltages are supported on the whole boundary $\bndry$ and represent the actual measurement. For now we assume that we can measure on the full-boundary $\bndry$. This assumption is obviously not practical and hence we discuss in Section \ref{sec:extrapolation} how to estimate the full-boundary measurements from restricted data. 

To get further understanding of the boundary value $\parphi$, we introduced a linear and bounded operator $\parI:\Hmhalf\to\parspace$ and defined the {partial ND map} by $
\widetilde{\ND}_\sigma:=\ND_\sigma\parI. 
$
The operator $\parI$ will be specified in Section \ref{sec:projectionOp}. Let the partial-boundary function $\parphi$ be produced by $\parphi=\parI\varphi$ for some $\varphi\in\Hmhalf$. Then we immediately obtain the identity
\begin{equation}
\label{eq:centralIdent}
\ND_\sigma\parphi={\ND}_\sigma\parI\varphi=\widetilde{\ND}_\sigma \varphi.
\end{equation}
Now we can properly define the main question of this study.
\begin{problem}
By injecting partial current patterns $\parI\varphi=\parphi\in \widetilde{H}^{-1/2}_\Gamma(\partial\Omega)$ as Neumann boundary data for 
\begin{equation*}
\begin{array}{rl}
\nabla\cdot\sigma\nabla u = & 0, \quad \mbox{ in }\Omega,\\
\sigma \frac{\partial u}{\dnu}  =& \widetilde{\varphi}, \quad  \mbox{ on } \bndry, \\
\end{array}
\end{equation*}
we can model the measurements as 
\[
\ND_\sigma\parphi=\widetilde{\ND}_\sigma \varphi = u|_{\bndry}.
\]
We want to know from the knowledge of the partial Neumann-to-Dirichlet map
\[
\widetilde{\ND}_\sigma:\widetilde{H}^{-1/2}(\partial\Omega)\to\widetilde{H}^{1/2}(\partial\Omega),
\]
how well we can recover $\sigma$.
\end{problem}

The key quality of this formulation is that we can represent our measurements {via} an orthonormal basis of $L^2(\bndry)$. For instance, let $\Omega$ be the unit disk. We choose the orthonormal basis given by the Fourier basis functions $\varphi_n(\theta)=\frac{1}{\sqrt{2\pi}}e^{in\theta}$ for $n\in\Z\backslash \{0\}$. Now we can obtain a matrix approximation $\widetilde{\mathbf{R}}_\sigma$ of the {partial ND map} from the measurements $\ND_\sigma\parphi_n=u_n|_{\bndry}$ with respect to the orthonormal basis as
\[
(\widetilde{\mathbf{R}}_\sigma)_{n,\ell}=(\widetilde{\ND}_\sigma \varphi_n,\varphi_\ell)=(\ND_\sigma\parphi_n,\varphi_\ell)=\frac{1}{\sqrt{2\pi}}\int_{\bndry} u_n|_{\bndry}(\theta)e^{-i\ell\theta}d\theta.
\]

\subsection{Representing the ND map by boundary layer potentials}\label{sec:repNDmaps}
In this section we {derive} a representation of the {partial ND map} based on boundary layer potentials. This way we are able to analyse the error we are doing compared to full-boundary data in a general setting. 

Let $G_\sigma(x,y)$ be the Green's functions of the conductivity equation with Neumann boundary conditions, that is
\begin{align} \label{def:GN1}
-\nabla \cdot \sigma \nabla G_\sigma(x,y)&=\delta(x-y), \qquad \text{for } x,y\in \Omega,
\\ \label{def:GN2}
\sigma \dnu G(x,y)&=1/|\partial \Omega|, \qquad \text{ for } y \in \partial \Omega, x \in \Omega.
\end{align}
We have the following integral representation formula for a solution $u$ of the conductivity equation (see \cite[Theorem 7.7]{mclean2000}):
\begin{equation*}\label{eq:repForm}
u(x)= \int_{\bndry} \sigma(y) \dnu u(y) G_\sigma(x,y)ds_y, \hspace{0.25 cm} \forall x\in\Omega,
\end{equation*}
with the condition $\int_{\partial \Omega} u = 0$.
Taking the limit $x\to\bndry$, we obtain the identity
\begin{equation*}\label{eq:BIE_cond}
u(x) = (S_\sigma\dnu u)(x), \hspace{0.25 cm} \forall x\in\bndry,
\end{equation*}
where $S_\sigma:H^{-1/2}(\bndry)\to H^{1/2}(\bndry)$ is the single layer operator  given by 
\[
S_\sigma\varphi(x) = \int_{\bndry} \sigma(y)G_\sigma(x,y)\varphi(y) ds_y.
\]
Thus, in this representation, the Neumann-to-Dirichlet map coincides with the single layer operator restricted to the space $\widetilde{H}^{-1/2}(\partial \Omega)$: 
\begin{equation*}\label{eq:ND_map}
\ND_\sigma=S_\sigma:\widetilde{H}^{-1/2}(\bndry)\to \widetilde{H}^{1/2}(\bndry).
\end{equation*}
Furthermore, this representation is used to define the {partial ND map}, by using the identity \eqref{eq:centralIdent}
\begin{equation}\label{eq:partialND_map}
\widetilde{\ND}_\sigma\varphi={\ND_\sigma}\parI\varphi=S_\sigma(\parI\varphi) = u|_{\partial\Omega},
\end{equation}
and the difference of ND maps can then be simply expressed by
\begin{equation}\label{eq:NDerrorBasis}
(\widetilde{\ND}_\sigma-\ND_\sigma)\varphi=\ND_\sigma(\parphi-\varphi)=S_\sigma(\parphi-\varphi).
\end{equation}

\subsection{Measurement extrapolation}\label{sec:extrapolation}
We have so far assumed that we can measure the data on the full boundary, which is of course not a reasonable assumption. In a realistic setting we can measure only on the same area {where} we inject the currents, due to restrictions in accessibility. {More precisely,} given a partial current pattern $\parphi\in\parspace$, the measurement $u|_{\bndry} = {\ND}_\sigma\parphi$ is only known on $\Gamma$. We denote the actual measurement by
$$\widetilde{u}:=u|_\Gamma.$$
We {propose an extrapolation procedure} to estimate $u|_{\bndry}$ from $\widetilde{u}$. In principle it is possible to do the extrapolation on the single measured boundary trace $\widetilde{u}$, but the task is a lot simpler if one uses difference data, see Figure \ref{fig:measTraces} for an illustration.  
The D-bar method we base this study on needs difference data to a constant conductivity, i.e. the difference of ND maps $\ND_{\sigma,1}:=\ND_\sigma-\ND_1$.
 
The difference map is a smoothing operator. This can be seen by representing the ND maps by layer potentials as in Section \ref{sec:repNDmaps}. The principle part of the difference is just the difference of single layer operators. Assuming that the conductivity coincides with the constant background close to the boundary, this difference is indeed a smoothing operator.\footnote{Thanks to Petri Ola for his insight!} We note that this can be generalized to any smooth background and is not limited to constants as reference data. Thus, for the extrapolation task any difference data is sufficient. 

Now we are left with the extrapolation on the difference measurement
\[
\widetilde{g}:=\left.(\ND_{\sigma,1}\parphi)\right|_\Gamma.
\]
The extrapolation is best adjusted to the problem at hand, but for simple conductivities, such as a circular inclusion close to the boundary as used in Figure \ref{fig:measTraces}, we propose to use cubic extension of the traces. For this we parametrize the complement of the measurement area $\Gamma^c=\bndry\backslash\Gamma$ by an open interval $\Gamma^c=(-a,a)$, with $a=h/2$. Then we know the boundary values at $\widetilde{g}(\pm a)$, and we can numerically calculate the derivatives $\widetilde{g}'(\pm a)$. This way we can define a unique cubic polynomial that extends the measurement $\widetilde{g}$ to $\Gamma^c$ {and} denote the estimated data by $g$. The question left is, how well does the extended data $g$ approximate the difference data $\ND_{\sigma,1}\parphi$? 

Note that we are essentially using spline interpolation on two interpolation points, for which the error bound is known to be $O(h^4)$, see for instance \cite{Deuflhard2012}. We will see in the next section that the convergence of partial ND maps is of lower order and hence the convergence rate is conserved under this choice.

\begin{figure}[t!]
\centering
\begin{picture}(400,250)
\put(-30,130){\includegraphics[width=150 pt]{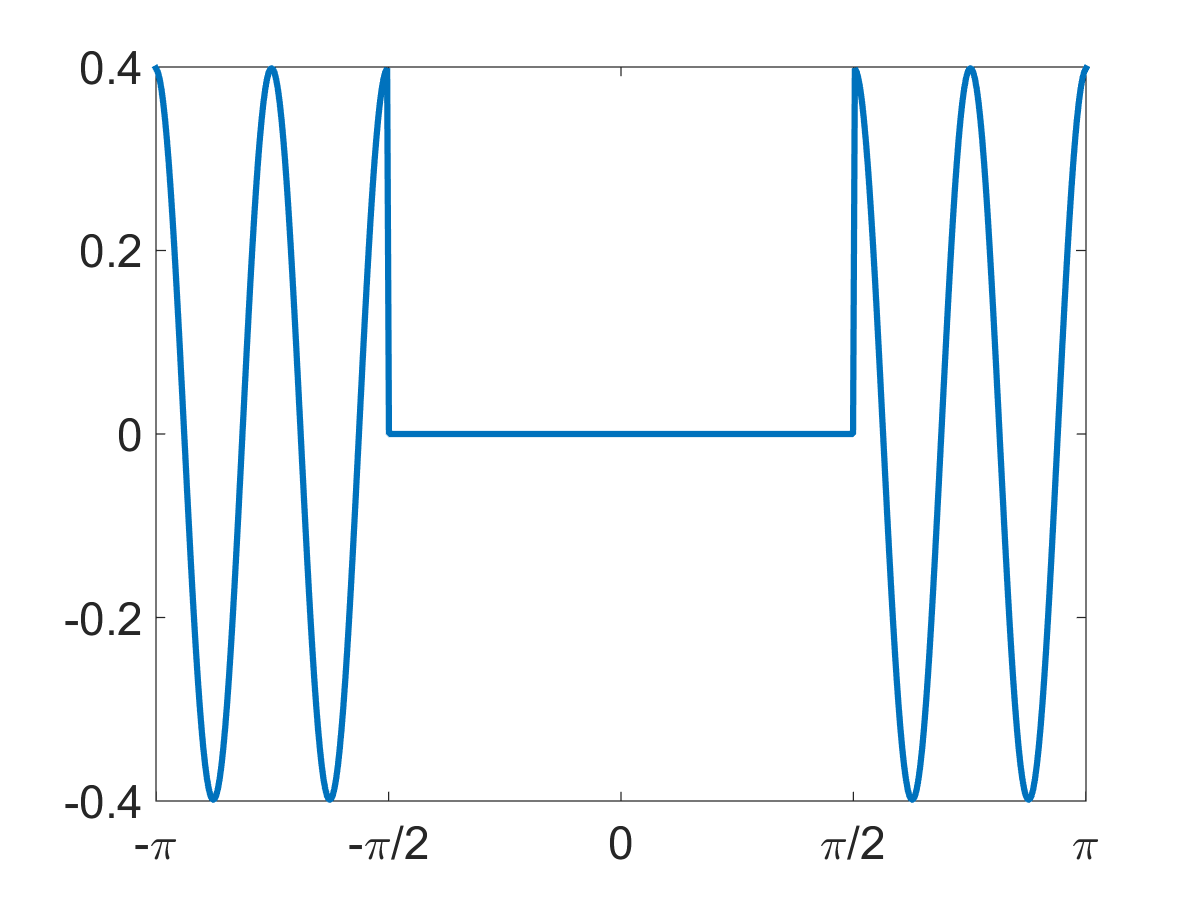}}
\put(110,130){\includegraphics[width=150 pt]{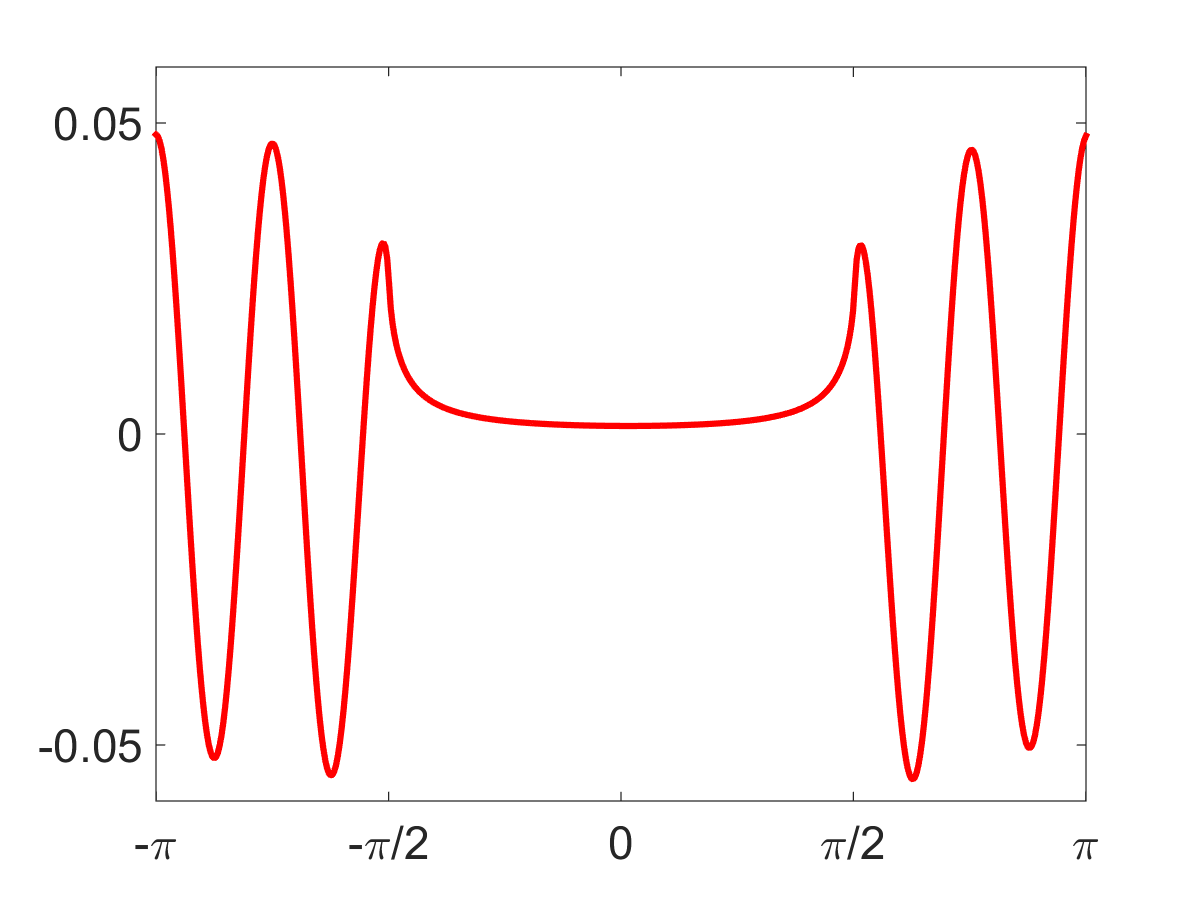}}
\put(250,130){\includegraphics[width=150 pt]{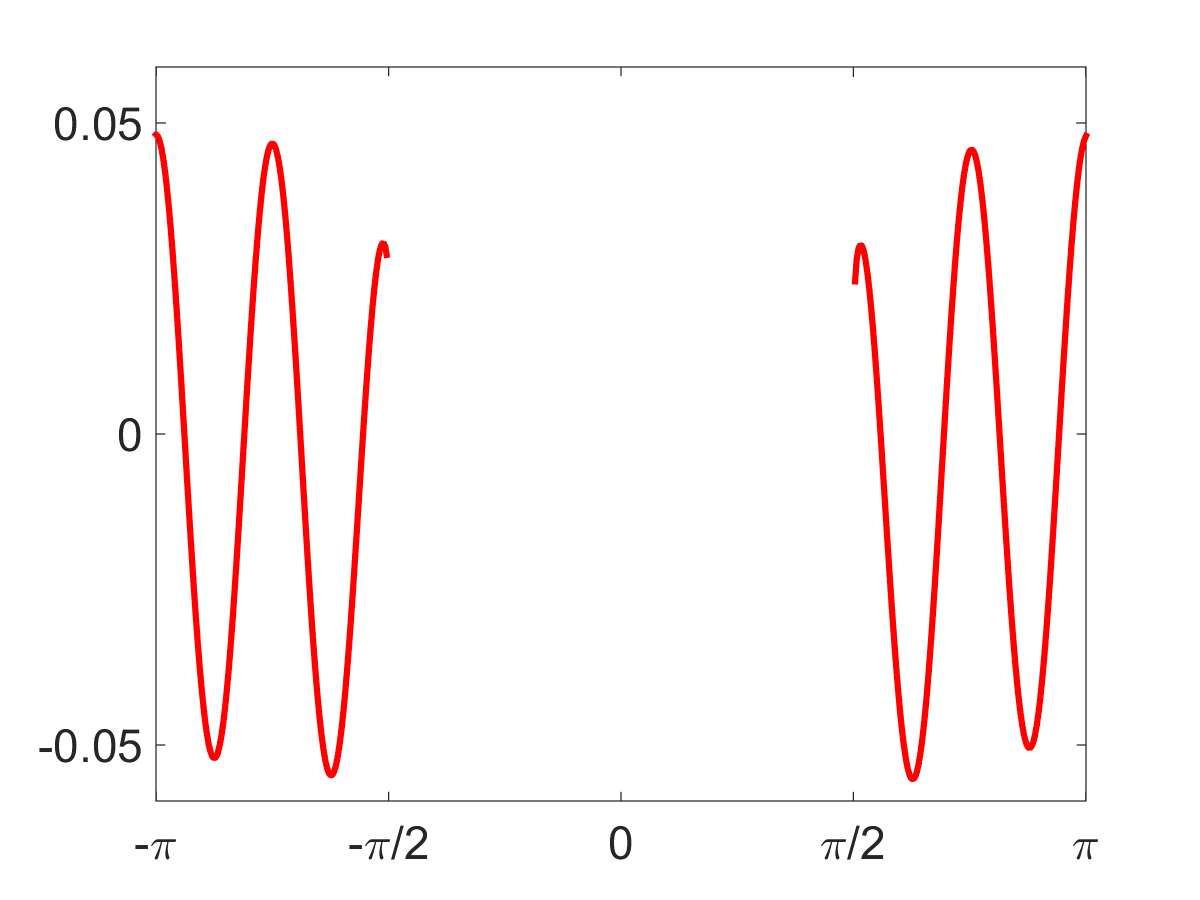}}

\put(-30,0){\includegraphics[width=150 pt]{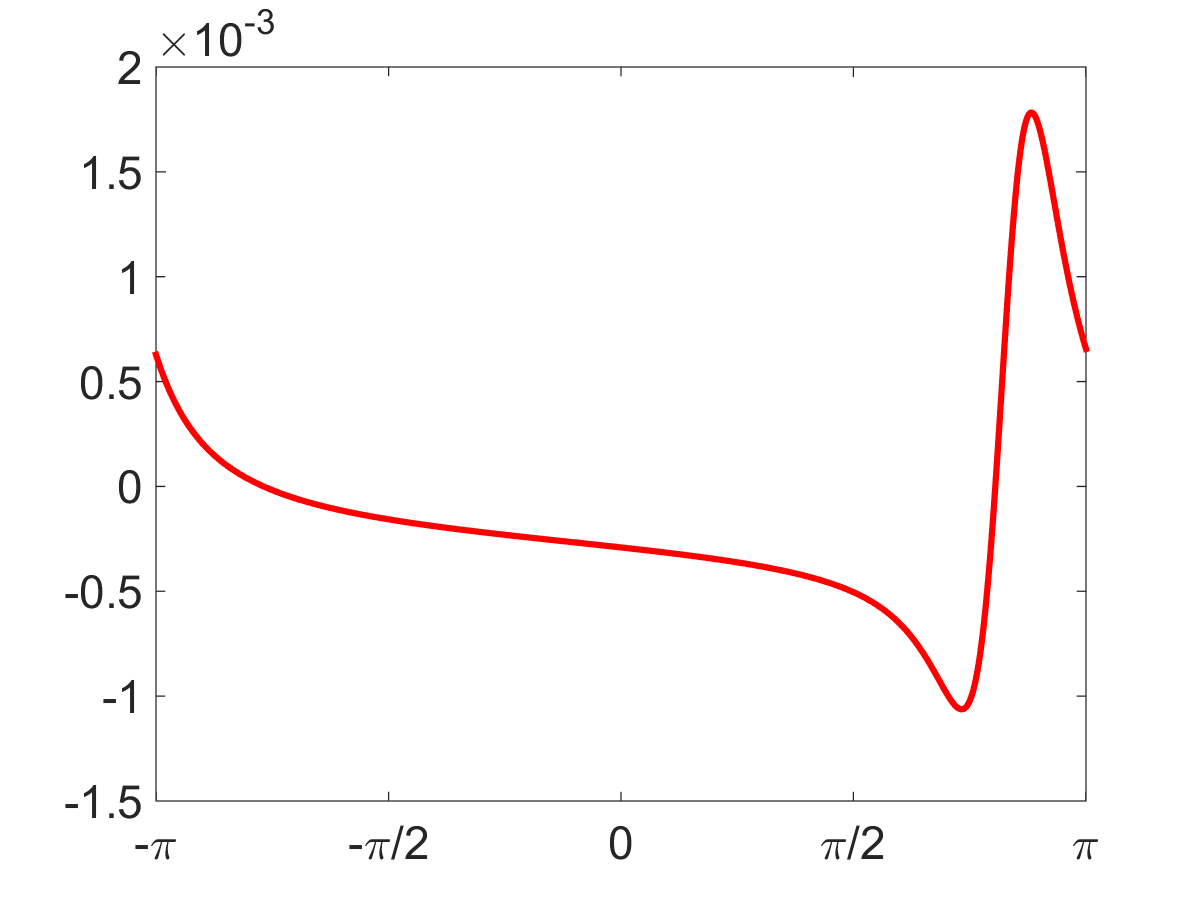}}
\put(110,0){\includegraphics[width=150 pt]{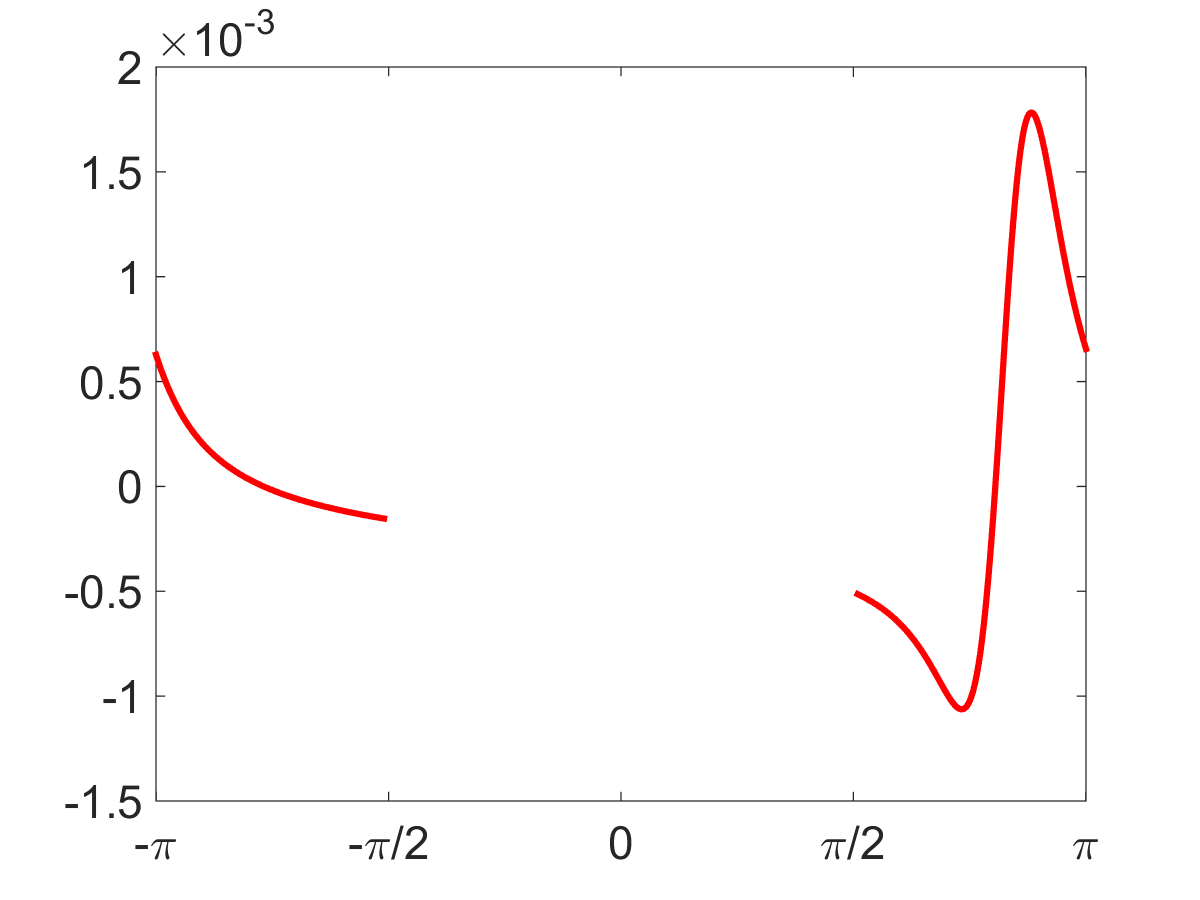}}
\put(250,0){\includegraphics[width=150 pt]{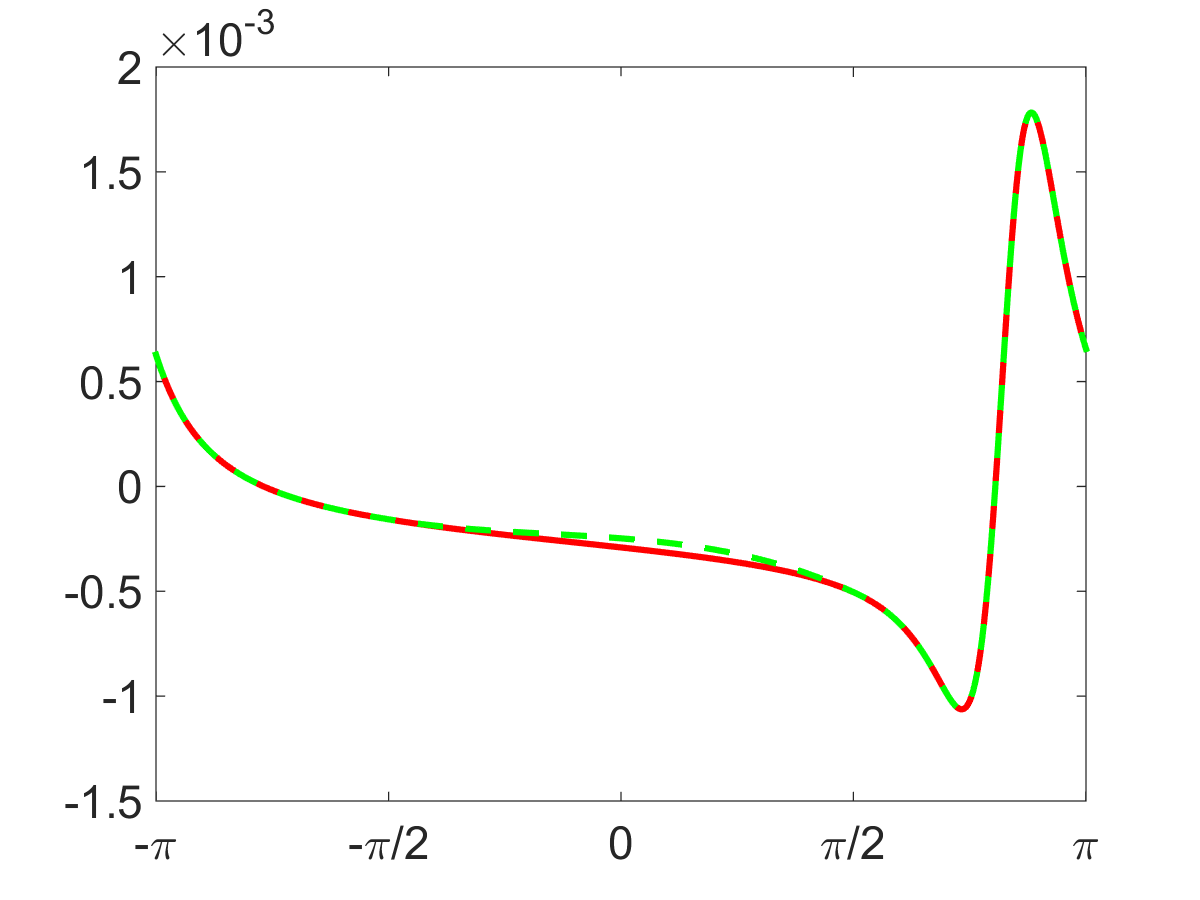}}
\put(30,245){Input $\parphi$}
\put(155,245){Full data $u|_{\bndry}$}
\put(290,245){Partial data $u|_\Gamma$}

\put(-5,115){Difference data $\ND_{\sigma,1}\parphi$}
\put(140,115){Measured difference $\widetilde{g}$}
\put(300,115){Estimated $g$}

\end{picture}
\caption{\label{fig:measTraces}Illustration of input and measurement data from the setting in Figure \ref{fig:GammaCircularCond}. The first row shows the input data (Left) and the full-boundary trace (Middle) next to what we can actually measure on $\Gamma$ (Right). The lower row shows that extrapolation on difference data is indeed an easier task. The ideal data (Left) next to the measurement restricted to $\Gamma$ (Middle). (Right) The estimated trace in green compared to the ideal data in red.}
\end{figure}
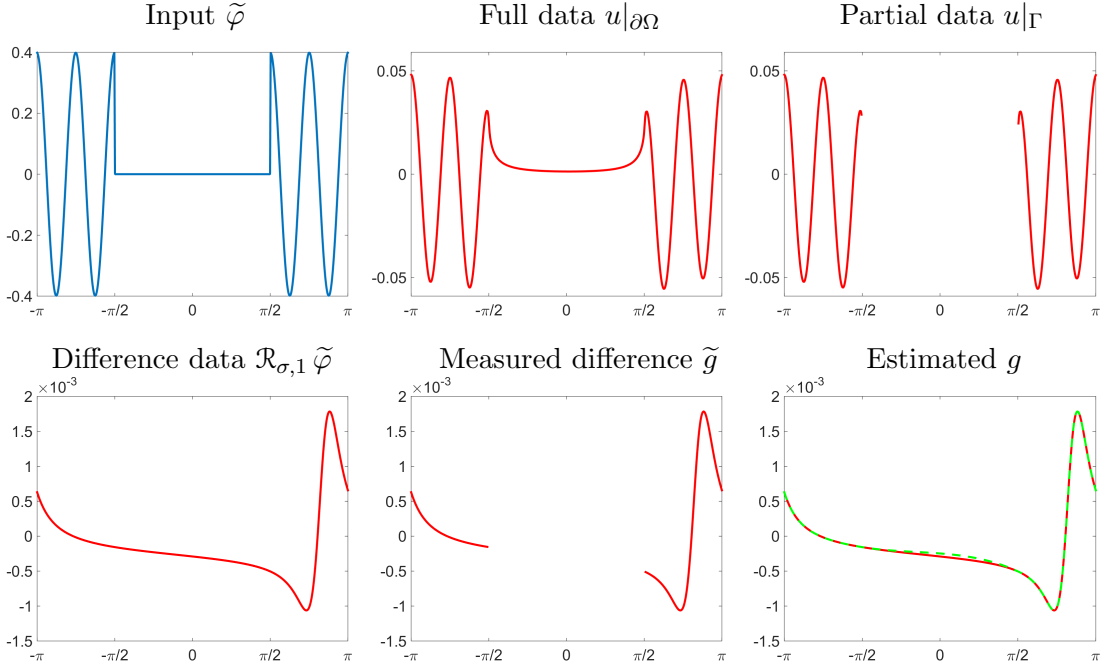 

\section{The partial ND map} \label{sec:projectionOp}
As we have seen, we can write the ND map in terms of boundary layer potentials. {In this section we will use this approach to carry out an error analysis}. 
We start by establishing the notation of the operator $\parI$ that maps orthonormal current patterns from $\Hmhalf$ to the subspace $\parspace$. Further, we will give some basic properties of the {partial ND map}. Let us start with a definition.
\begin{definition}
Let $\parI$ be a linear and bounded operator from $\Hmhalf$ to the subspace $\parspace$, we will call $\parI$ the partial-boundary map. The {partial ND map} is then given as the composition $\widetilde{\ND}_\sigma:=\ND_\sigma\parI:\widetilde{H}^{-1/2}(\partial\Omega)\xrightarrow{\parI}\widetilde{H}^{-1/2}_\Gamma(\partial\Omega) \xrightarrow{\ND_\sigma} H^{1/2}(\bndry).$
\end{definition}
The ND map $\ND_\sigma$ is a linear, bounded, and self-adjoint operator, by standard theory of elliptic PDEs \cite{Gilbarg1977}. The {partial ND map} is easily seen to be linear and bounded as composition of two linear and bounded operator. From the dual pairing we have 
\[
\langle\widetilde{\ND}_\sigma\varphi,\phi \rangle=\left\langle\ND_\sigma\parphi,\phi \right\rangle=\left\langle\parphi,\ND_\sigma\phi \right\rangle
\]
by self-adjointness of $\ND_\sigma$. Since $\ND_\sigma\parI\neq\parI\ND_\sigma$, we get that $\widetilde{\ND}_\sigma$ is not self-adjoint on $\Hmhalf$. Assuming $\parI$ is a projection operator, we have $\parI\varphi=\varphi$ for $\varphi\in\parspace$, from which we can deduce that the {partial ND map} $\widetilde{\ND}_\sigma$ is self-adjoint on $\parspace$. Anyhow, we will not need this result in our further analysis.

Let us now introduce two simple choices for the partial-boundary map.
\begin{itemize}
\item[{\textbf{Scaling}}] The first map shifts and scales the functions $\varphi$ to the partial-boundary $\Gamma$. We give here the definition for the unit disk, where the boundary $\bndry$ can be parametrized by an angle $\theta\in[0,2\pi]$. Then we denote $\Gamma=[\theta_1,\theta_2]\subset[0,2\pi]$ and the partial-boundary map is given by
\[
\begin{split}
\parI^s\varphi(\theta)=
\left\{
\begin{array}{cl}
\varphi\left(\frac{\theta-\theta_1}{r}\right) &\text{ if } \theta\in\Gamma ,\\
0 &\mbox{ else },
\end{array} \right.
\hspace{0.25cm} \text{with} \hspace{0.25cm} r=\frac{|\Gamma|}{|\bndry|}.
\end{split}
\]
\item[{\textbf{Cut-off}}] The second option is a cut-off with mean correction, that is
\[
\parI^c\varphi(\theta)=\left\{
\begin{array}{cl}
\varphi(\theta)-\frac{1}{|\Gamma|}\int_\Gamma\varphi(\tau)d\tau &\text{ if } \theta\in\Gamma,\\
0  &\mbox{ else }.
\end{array}\right.
\]
The mean is subtracted to make sure that $\parI^c\varphi\in\parspace$. Note that this choice is clearly a projection.
\end{itemize}

\subsection{Error analysis for the chosen maps}\label{sec:err_analy}
With the previous definitions we can prove the main result of this study, the convergence for the {partial ND map}. The goal is to prove an error estimate of the {partial ND map} to the full-boundary ND map. The estimate depends on the length of the missing boundary $h=|\Gamma^c|$ and hence we can establish a convergence result to full-boundary data.
In the proof we restrict ourselves to the unit disk and the exponential Fourier basis functions. The methodology is straightforward and can be readily generalized to other basis functions and more {general} domains. For the two choices $\parI^s$ and $\parI^c$ of the partial-boundary map we obtain a linear convergence rate. 

\begin{proposition}[Error of the {partial ND map}]\label{prop:convergence_single} Let $\Omega \subset \R^2$ be the unit disk,\ $\sigma \in L^{\infty}(\Omega)$ be a conductivity with $0 < \sigma_0 \leq \sigma(x)$ and $\sigma \equiv 1$ close to $\partial \Omega$.
Denote the {partial ND maps} as $\widetilde{\ND}_{\sigma}^c=\ND_\sigma\parI^c$ and $\widetilde{\ND}_{\sigma}^s=\ND_\sigma\parI^s$. Let the basis functions be $\varphi_n(\theta)=\frac{1}{\sqrt{2\pi}} e^{in\theta}$ for $n \neq 0$, and $\Gamma = \{ e^{i \theta} \in \partial \Omega : \theta \in [h/2,2\pi-h/2] \}$. Then there is a constant $C>0$ independent on $n$ such that:
\begin{align}\label{eq:convRes_cut}
&\|(\widetilde{\ND}_\sigma^c-\ND_\sigma)\varphi_n\|_{L^2(\partial\Omega)}\leq C h, \qquad \text{for } 0 < h \leq \pi,\\
&\|(\widetilde{\ND}_\sigma^s-\ND_\sigma)\varphi_n\|_{L^2(\partial\Omega)}\leq Cn^2 h, \qquad \text{for } 0 < h  < \frac{2\pi}{n+1}.
\end{align}
For the cut-off case, this is equivalent to the operator estimate
\begin{equation} \label{eq:convRes_cutOP}
\|\widetilde{\ND}_\sigma^c-\ND_\sigma\|_{L^2(\partial\Omega)\to L^2(\partial\Omega)}\leq C h.
\end{equation} 
\end{proposition}

\begin{proof} 
Starting from \eqref{eq:NDerrorBasis}, we have
\[
\|(\widetilde{\ND}_\sigma-\ND_\sigma)\varphi_n\|_{L^2(\partial\Omega)} = \|S_\sigma(\parI-1)\varphi_n\|_{L^2(\partial\Omega)}.
\]
We recall that the single layer operator $S_\sigma$ involves the Green's function $G_\sigma(x,y)$ defined by properties \eqref{def:GN1}, \eqref{def:GN2}. Using the classical result \cite[Theorem 7.1]{Littman1963}, as well as the a priori bounds on $\sigma$, there exists a constant $K>0$ such that 
\[
K^{-1}G_1(x-y)\leq G_\sigma(x,y)\leq K G_1(x-y), \quad \text{for } x,y \in \Omega,
\]
where $G_1$ is the following fundamental solution of the Laplace equation:
\[
G_1(x)=-\frac{1}{2\pi}\log |x|.
\]
Let $S_1$ be the classical single layer operator
\[
S_1 \varphi(x) = \int_{\bndry} G_1(x-y)\varphi(y) ds_y.
\]
Then we have 
\[
\|S_\sigma(\parI-1)\varphi_n\|_{L^2(\partial\Omega)} \leq C \|S_1 (\parI-1)\varphi_n\|_{L^2(\partial\Omega)}.
\]
We can now use \cite[Lemma 7.1]{Nachman1996} that states boundedness of $S_1:H^s(\bndry)\to H^{s+1}(\bndry)$ for $-1\leq s\leq 0$, which yields for $s=-1$ that
\[
\|S_1 (\parI-1)\varphi_n\|_{L^2(\partial\Omega)} \leq C \| (\parI-1)\varphi_n\|_{H^{-1}(\partial\Omega)},
\]
where $\| \varphi \|^2_{H^{-1}(\partial \Omega)} = \sum_{k \in \Z} (1+k^2)^{-1} |\widehat{\varphi}(k)|^2$ and 
\[
\widehat{\varphi}(k) = \frac{1}{2 \pi}\int_0^{2 \pi}\varphi(\theta)e^{-ik\theta}d\theta.
\]

We are now left with explicitly calculating the Fourier coefficients for the two choices of the partial-boundary map. For the cut-off case, i.e. $\parI = \parI^c$ we have
\begin{equation}\label{def:cutoffdiff}
(\parI^c -1)\varphi_n(\theta)= -C_{\Gamma}\chi_{\Gamma} - \chi_{\Gamma^c}\varphi_n(\theta),
\end{equation}
where $\chi_{\Gamma}$ (resp. $\chi_{\Gamma^c}$) is the characteristic function of $\Gamma$ (resp. $\Gamma^c$) and
$$C_{\Gamma} = \frac{1}{|\Gamma|}\int_{\Gamma}\frac{1}{\sqrt{2\pi}} e^{in\theta}d\theta = \frac{-2\sin(nh/2))}{(2\pi-h)n}.$$

Straightforward computations of the Fourier coefficients of \eqref{def:cutoffdiff} and the basic estimate $|\sin(x)| \leq |x|$ show that $|\widehat{(\parI^c -1)\varphi_n}(k)| \leq C h$, for $h \leq \pi$ and the constant is independent of $n$. This finishes the proof of the error estimate for the cut-off case.

%

For the scaling case, i.e. $\parI = \parI^s$, we have, by definition
\begin{equation}\label{def:scalingdiff}
(\parI^s -1)\varphi_n(\theta)= \chi_{\Gamma}\varphi_n\left(\frac{\theta-h/2}{2\pi-h}2\pi\right) -\varphi_n(\theta).
\end{equation}

The Fourier coefficients are
\begin{align*}
&\widehat{(\parI^s -1)\varphi_n}(k) = \frac{ie^{-\frac i 2 h k}(e^{ihk}-1)(2\pi-h)}{(\sqrt{2\pi})^{3}(2(k-n)\pi -hk)}, \quad \text{for } k \neq n, \\
&\widehat{(\parI^s -1)\varphi_n}(n) = \frac{2(2\pi-h)\sin(\frac{hn}{2})}{(\sqrt{2\pi})^{3}hn}-\frac{1}{\sqrt{2\pi}},
\end{align*}
which satisfy, for $h < \frac{2\pi}{n+1}$,
\begin{align*}
&|\widehat{(\parI^s -1)\varphi_n}(k)| \leq C h, \qquad \text{for } k \neq n,\\
&|\widehat{(\parI^s -1)\varphi_n}(n)| \leq C h n^2.
\end{align*}
This gives the error estimates for the scaling case.
\end{proof}

These estimates immediately extend to difference data.
\begin{corollary}\label{cor:DiffND}
Under the assumptions of Proposition \ref{prop:convergence_single} we have
\begin{align}\label{eq:convResDiff}
&\|(\widetilde{\ND}_{\sigma,1}^c-\ND_{\sigma,1})\varphi_n\|_{L^2(\partial\Omega)}\leq C h,\\
&\|(\widetilde{\ND}_{\sigma,1}^s-\ND_{\sigma,1})\varphi_n\|_{L^2(\partial\Omega)}\leq Cn^2 h,
\end{align}
where $\ND_{\sigma,1}:=\ND_\sigma-\ND_1$, $\widetilde \ND_{\sigma,1}:=\widetilde \ND_\sigma-\widetilde \ND_1$. 

{For the cut-off case this is equivalent to the operator estimate
\begin{equation} \label{eq:convResDiffOP}
\|\widetilde{\ND}_{\sigma,1}^c-\ND_{\sigma,1}\|_{L^2(\partial\Omega) \to L^2(\partial\Omega)}\leq C h.
\end{equation} }
\end{corollary}

We note that the above estimates hold for more general partial-boundary maps $\parI$ that have at least linear convergence on $\bndry$ with respect to $h$: that is, for some $C>0$,
\[
\|(\parI-1)\varphi_n\|_{H^{-1}(\partial \Omega)}\leq C h.
\]
To conclude this chapter we recall that the extrapolation error of partial-boundary measurements is of higher order, as described in Section \ref{sec:extrapolation}. Thus, the data error is governed by restricting the input currents and not by the limitation of the measurement domain. 

\section{CGO solutions and reconstructing the conductivity}\label{sec:DbarND}
In this section we present a fast way to compute the CGO solutions or directly the scattering transform for the reconstruction. For this we use a Born approximation as introduced and studied in \cite{Knudsen2007,Siltanen2000}. It has been recently shown that using this approximation to obtain the scattering transform can lead with modern computing power to a real time D-bar {reconstruction} algorithm \cite{Dodd2014}. This is achieved by an initialization step and parallelization. In particular the scattering transform can be parallelized in the parameter $k\in\C\backslash\{0\}$ and solving the D-bar equation can be done independently for each point $z\in\Omega$. We will discuss in the following how to obtain the CGO solutions and scattering transform from the ND map. But first we need a short overview of the classical D-bar method for $C^2$-conductivities to establish the terminology.
\subsection{The classical D-bar method for EIT}
We give a very brief review of the standard D-bar method to compute direct reconstructions for EIT with full-boundary data. As the D-bar method we refer to \cite{Knudsen2009}{, which is based on Novikov's \cite{Novikov1988} and Nachman's \cite{Nachman1996} results}. We consider the conductivity equation with Dirichlet condition
\begin{equation}\label{eq:cond_eqD}
\begin{array}{rl}
\nabla\cdot\sigma\nabla u\ = & 0, \quad \mbox{ in }\Omega,\\
\left. u \right|_{\bndry}\ =& f, \quad  \mbox{ on }\bndry.
\end{array}
\end{equation} 
Let the conductivity $\sigma\in C^2(\Omega)$ be bounded by $0<\sigma_0\leq \sigma(x) \leq C$ for all $x\in\Omega$ and $\sigma\equiv 1$ close to $\bndry$. The measurement is modelled by the Dirichlet-to-Neumann map 
\begin{equation*}
\Lambda_\sigma:\;f\mapsto \left.\sigma\frac{\partial u}{\partial\nu}\right|_{\bndry}.
\end{equation*}
The basis of the algorithm is to transform the conductivity equation to a Schr\"odinger type equation by change of variables $v=\sqrt{\sigma}u$ to
\begin{equation*}
(-\Delta + q)v=0.
\end{equation*}
The potential $q=\frac{\Delta\sqrt{\sigma}}{\sqrt{\sigma}}$ is extended from $\Omega$ to $\C$ by setting $\sigma(z)\equiv 1$ for $z\in \C\backslash\Omega$. The main idea of the D-bar method is to look for special Complex Geometric Optics (CGO) solutions $\psi(z,k)$, as introduced in \cite{Faddeev1966,Sylvester1987}, that satisfy 
\begin{equation}
(-\Delta + q(\cdot))\psi(\cdot,k)=0,
\end{equation}
for an auxiliary variable $k\in\C$ with the asymptotic condition $e^{-ikz}\psi(z,k)-1 \in W^{1,p}(\R^2)$, with $2<p<\infty$.

These CGO solutions can be recovered for $k\in\C\backslash \{0\}$ by solving a Fredholm boundary integral equation of the second kind
\[
\psi(z,k)=e^{ikz}-\int_{\bndry} G_k(z-\zeta)\left[\Lambda_\sigma - \Lambda_1\right] \psi(\zeta,k)\; dS(\zeta),
\]
where $G_k$ is Faddeev's Green's function for the Laplacian \cite{Faddeev1966}, with asymptotics matching $\psi$.

Having the CGO solutions we can compute the scattering transform for $k\neq0$ by
\begin{equation}
\mathbf{t}(k)=\int_{\partial\Omega} \expiconjkz (\Lambda_\sigma - \Lambda_1)\psi(z,k) ds.
\end{equation}

We set $\mu(z,k)=\psi(z,k)\expikz$. With the scattering transform $\mathbf{t}(k)$ obtained, the reconstruction of $\sigma$ can then be recovered by solving the D-bar equation for $z\in\Omega$
\begin{equation}\label{eq:dbark_eq}
\dbar_k \mu(z,k) = \frac{1}{4\pi \bar{k}} \mathbf{t}(k)e_{-k}(z) \overline{\mu(z,k)},
\end{equation} 
with the modified exponential $e_k(z)=e^{i(kz+\bar{k}\bar{z})}$. {Here we used the complex derivative $\dbar_k = \frac{1}{2}\left(\frac{\partial}{\partial k_1}+i \frac{\partial}{\partial k_2}\right)$, for $k = k_1 + ik_2$. }The reconstructed conductivity is then given by 
\[
\lim_{k\to0}\mu(z,0)=\sqrt{\sigma(z)},
\]
where one can substitute $k=0$ for the limit as shown by the analysis in \cite{Barcel'o2001,Knudsen2004a}.

\subsection{Recovering the CGO solutions and scattering transform from the ND map}
The classical D-bar method uses the DN map, but in our setting we have only the {partial ND map} available. For full-boundary data we could invert the measurement matrix to obtain the DN matrix, but the partial-boundary measurement matrix is not invertible. That means we need to compute the CGO solutions and the scattering transform directly from the ND map. We restrict the analysis to the full-boundary ND map and treat the {partial ND map} in the following as noisy input data. 

In the case of DN maps the derivation is done with Alessandrini's identity \cite{Alessandrini1988}, which states a relation of two DN maps for different potentials.
It is straightforward to adjust this to a modified Alessandrini's identity for the Neumann problem: given two solutions $v_i\in H^1(\Omega)$ of 
\begin{equation}
\label{eq:Schroed}
\begin{array}{rl}
(-\Delta+q_i)v_i =& 0 \quad \text{ in }  \Omega, \\
\dnu v_i  =& \varphi_i \quad \text{on }\bndry,
\end{array}
\end{equation}
then the following identity holds,
\begin{equation}
\label{eq:Alesan}
\int_\Omega (q_1-q_2)v_1v_2 dz = -\int_{\partial\Omega} \varphi_1({\ND}_{q_1}-{\ND}_{q_2})\varphi_2 ds.
\end{equation}

With the identity \eqref{eq:Alesan} we can continue to derive the equation used to obtain the CGO solutions, summarized in the following result.
\begin{proposition}\label{prop:CGO}
Let $\sigma\in C^2(\Omega)$ be bounded by $0<\sigma_0\leq \sigma(x) \leq C$ for all $x\in\Omega$ and $\sigma\equiv 1$ close to $\bndry$. 
Given the Complex Geometric Optics solutions $\psi(z,k)$ for $k\in\C\backslash \{0\} $ and $z\in\bndry$, with $\Omega\subset\R^2$ the unit disk. The following equation holds 
\begin{equation}\label{eq:BndryInt_psi}
\psi(z,k)=e^{ikz} - (B_k-\frac{1}{2})(\ND_1 - \ND_\sigma)\dnu \psi(\cdot,k),
\end{equation}
{where $B_k$ is} the double layer operator
\[
B_k\phi(z)= p.v. \int_{\bndry}\left(\dnu G_k(z-\zeta)\right)\phi(\zeta) ds(\zeta)
\]
and $G_k$ is Faddeev's Green's function.
The normal derivative of $G_k(z-\zeta)$ for $z,\zeta \in \bndry$ with respect to $\zeta$ is given by
\begin{equation}\label{eq:normDeriv_G}
\dnu G_k(z-\zeta) = \frac{1}{2\pi}\text{Re}\left(\zeta \frac{e^{ik(z-\zeta)}}{z-\zeta}\right).
\end{equation}

\end{proposition}
\begin{proof}
The boundary integral equation $\psi$ can be obtained similar as in the DN case. Consider the modified Alessandrini's identity \eqref{eq:Alesan} and choose $q_2=q=\sqrt{\Delta \sigma}/\sqrt{\sigma}$, $v_2=\psi(\cdot,k)$, $q_1=0$, and $v_1=G_k(z-\zeta)$ with $z\notin\bar{\Omega}$,
we have
\begin{equation}
\label{eq:CGO_Omega}
-\int_\Omega G_k(z-\zeta)q(\zeta)\psi(\zeta,k)ds(\zeta) = \int_{\bndry} \left(\dnu G_k(z-\zeta)\right)(\ND_q-\ND_0)\dnu\psi(\zeta,k)ds(\zeta).
\end{equation}
From the Lippmann-Schwinger-type equation for the CGO solutions we have the identity
\begin{equation}\label{eq:LippSchw}
\psi = \expikz-G_k\ast(q\psi),
\end{equation}
which can be used to rewrite the left hand side of \eqref{eq:CGO_Omega}. Then taking the limit $z\to\bndry$ in the double layer potential (right hand of \eqref{eq:CGO_Omega}) introduces the usual jump term \cite{Nachman1988a,Steinbach2008} and we obtain
\begin{equation}
\label{eq:CGO_bndry}
\psi(z,k) = \expikz - (B_k-\frac{1}{2})(\ND_0 - \ND_q)\dnu \psi(\cdot,k).
\end{equation}
Next we want to express the above boundary integral equation by the ND map $\ND_\sigma$. For that let $v$ be the solution of \eqref{eq:Schroed} and consider the conductivity equation
\[
\nabla\cdot\sigma\nabla u =0 \ \text{ in } \ \Omega, \hspace{0.5cm} \dnu u = \phi.
\]
We have by assumption that $\left. \dnu \sigma\right|_{\bndry}=0$, then we obtain
\[
\ND_q \phi = v|_{\bndry}=\sigma^{1/2}u|_{\bndry}=\sigma^{1/2}\ND_\sigma(\sigma^{1/2}\phi).
\]
Further, the assumption that $\sigma\equiv 1$ near the boundary yields the identity
\[
\ND_q=\ND_\sigma
\]
and \eqref{eq:BndryInt_psi} follows.

We are left to show \eqref{eq:normDeriv_G} the derivation of the normal derivative of $G_k$, for this we use the explicit form of the integral defining $G_k$. We will concentrate on the case $k=1$, the general case follows by the identity
$G_k(z)=G_1(kz) $ and $G_1=e^{iz}g_1(z)$.
Further, we can represent $g_1(z)$ by the exponential-integral function $\Ei(z)$ as done in \cite{Boiti1987} by
\begin{equation}
g_1(z)=-\frac{1}{4\pi}e^{-iz}(\Ei(iz)+\Ei(-i\bar{z}))=-\frac{1}{2\pi}e^{-iz}\real(\Ei(iz)).
\end{equation}
Thus, 
\[
G_1(z)=e^{iz}g_1(z)=-\frac{1}{2\pi}\real(\Ei(iz)).
\]
The exponential-integral function $\Ei(z)$ can be easily evaluated by modern software such as MATLAB. However, the most important property for us is that the derivative is given by
\[
\partial_z \Ei(z)=\frac{e^z}{z}.
\]
For the argument $z-\zeta$, with $z,\zeta\in \bndry$ and $z-\zeta \neq 0$ the normal derivatives in direction of $\zeta$ are then given by
\[
\begin{split}
\dnu \Ei(i(z-\zeta))=-\zeta\frac{e^{i(z-\zeta)}}{z-\zeta}, \text{ and }
\dnu \Ei(-i(\bar{z}-\bar{\zeta}))=-\bar{\zeta}\frac{e^{-i(\bar{z}-\bar{\zeta})}}{\bar{z}-\bar{\zeta}}.
\end{split}
\]
Together we obtain
\begin{equation}\nonumber
\dnu G_1(z-\zeta)=-\frac{1}{4\pi}\left(-\zeta \frac{e^{i(z-\zeta)}}{z-\zeta} - \bar{\zeta}\frac{e^{-i(\bar{z}-\bar{\zeta})}}{\bar{z}-\bar{\zeta}} \right) =\frac{1}{2\pi}\real\left(\zeta \frac{e^{i(z-\zeta)}}{z-\zeta} \right).
\end{equation}
\end{proof}

The equation \eqref{eq:BndryInt_psi} obtained cannot be solved without further knowledge about the normal derivative $\dnu\psi$. We would need to state the boundary integral equation with $\dnu\psi$ on the left side, which has been done in \cite{Isaev2012}. For our purpose the formulation \eqref{eq:BndryInt_psi} is sufficient, since we compute it with the help of a Born approximation to compute the CGO solutions. But before we discuss this in the next section, let us obtain an equation for the scattering transform in a similar fashion. The scattering transform is defined as
\[
\T(k) = \int_{\Omega} q(\zeta) e^{i\bar{k}\bar{\zeta}}\psi(\zeta,k)ds(\zeta).
\]
By the modified Alessandrini's identity \eqref{eq:Alesan} and the choices $q_1=0$, $v_1=e^{i\bar{k}\bar{z}}$, $q_2=q=\sqrt{\Delta \sigma}/\sqrt{\sigma}$ and $v_2=\psi(\cdot,k)$ we obtain
\[
\T(k) = \int_{\bndry} \left( \dnu e^{i\bar{k}\bar{\zeta}} \right)(\ND_1-\ND_\sigma)\dnu\psi(\zeta,k)ds(\zeta). \\
\]
\subsection{Born approximation and reconstruction error}\label{sec:reconError}
We present the {integral formulas} to obtain the CGO solutions and the scattering transform separately by using a Born approximation as utilized in \cite{Knudsen2007,Siltanen2000}. We approximate the CGO solutions {with} their asymptotic behaviour, i.e.
\[
\psi\approx e^{ikz}.
\]
On the unit disk the normal derivative is simply given by
\[
\dnu e^{ikz}=ikz e^{ikz}.
\]
Now we can compute approximate CGO solutions from equation \eqref{eq:BndryInt_psi} by evaluating the following equation
\begin{equation}\label{eq:psiExp}
\psi^{\mathrm{ND}}(z,k):=e^{ikz} - (B_k-\frac{1}{2})(\ND_1 - \ND_\sigma)\dnu e^{ikz}.
\end{equation}
The CGO solutions can be of particular interest, since they contain non-linear information of the inverse conductivity problem and can be used as a data fidelity term as shown in \cite{Hamilton2014}.

\noindent
If one is instead just interested in computing the reconstruction, the approximate scattering transform can be directly obtained from
\begin{equation}\label{eq:scatExp}
\begin{split}
\T^{\mathrm{ND}}(k) &:= \int_{\bndry} \left(\dnu e^{i\bar{k}\bar{\zeta}}\right)(\ND_1-\ND_\sigma)\dnu e^{ik\zeta}ds(\zeta).
\end{split}
\end{equation}	
{This} can then be used to solve the D-bar equation \eqref{eq:dbark_eq} to get the reconstruction of $\sigma$. Computationally it is not feasible to compute the scattering transform for all $k\in\C\backslash\{0\}$. Therefore we define a cut-off radius $R>0$ and compute 
\[
\T^{\mathrm{ND}}_{R}(k)=
\left\lbrace\begin{array}{cl}
\T^{\mathrm{ND}}(k) &\text{ if } |k|< R ,\\
0 &\mbox{ else. }
\end{array} \right.
\]
We use $\T^{\mathrm{ND}}_{R}(k)$ to solve the D-bar equation \eqref{eq:dbark_eq} and denote the reconstructed conductivity by $\sigma_R$.

So far we used the full-boundary ND map $\ND_\sigma$. Interpreting the {partial ND map} $\widetilde{\ND}_\sigma$ as a noisy approximation to $\ND_\sigma$, we can use the results of Section \ref{sec:err_analy} to establish an error estimate for the reconstruction from partial-boundary data.
The essential step is to verify that the linear estimate holds for the scattering transform. 

{From now on we will restrict our analysis to the partial ND map $\widetilde \ND_{\sigma}^c$ constructed with the cut-off operator $\parI^c$, since we need the operator estimate \eqref{eq:convResDiffOP} in the proof.} Anyhow, for the computations with finite matrix approximations of the ND maps, the results can be extended for both partial-boundary maps.

\begin{proposition}\label{prop:scattering}
Let $\Omega\subset\R^2$ be the unit disk  and $\Gamma = \{ e^{i \theta} \in \partial \Omega : \theta \in [h/2,2\pi-h/2] \}$. Let $\sigma\in C^2(\Omega)$ be bounded by $0<c\leq \sigma(x) \leq C$ for all $x\in\Omega$ and $\sigma\equiv 1$ close to $\bndry$. For a fixed cut-off radius $0<R<\infty$, let $\T^{\exp}_R$ be computed from ${\ND}_{1,\sigma}$ and $\widetilde{\T}^{\exp}_R$ from {$\widetilde{\ND}_{1,\sigma}^c$ (cut-off case)}. Then for $0 < h \leq \pi$, there are constants $C, L>0$ such that
\begin{equation}
\left| \frac{\T^{\exp}_R(k)-\widetilde{\T}^{\exp}_R(k)}{\bar k}\right| \leq C h |k| e^{2L |k|}, \quad \text{for } |k| \leq R.
\end{equation}
\end{proposition}
\begin{proof}
Estimating for fixed $k$, we have by
\[
\begin{split}
\left|\frac{\T^{\exp}_R(k)-\widetilde{\T}^{\exp}_R(k)}{\bar k}\right| &= \left| \frac{1}{\bar{k}} \int_{\bndry} \left( \dnu \expiconjkz\right) (\ND_{\sigma,1}-\widetilde{\ND}_{1,\sigma}^c)\dnu\expikz dz \right| \\
&\leq \left| \frac{1}{\bar{k}} \right| \left\| \left( \dnu \expiconjkz\right) (\ND_{\sigma,1}-\widetilde{\ND}_{1,\sigma}^c)\dnu\expikz\right\|_{L^1(\bndry)} \\
&\leq \left| \frac{1}{\bar{k}} \right| \left\|\dnu \expiconjkz\right\|_{L^2(\bndry)}\left\|(\ND_{\sigma,1}-\widetilde{\ND}_{1,\sigma}^c)\dnu\expikz\right\|_{L^2(\bndry)} \\
&\leq \frac{C h}{| k|} \left\|\dnu \expiconjkz\right\|_{L^2(\bndry)} \left\|\dnu \expikz \right\|_{L^2(\bndry)},
\end{split}
\]
where we used the operator inequality \eqref{eq:convResDiffOP}. The last norms can be bounded by
\[
 \|\dnu \expiconjkz\|_{L^2(\bndry)} = \|i\bar{k}\bar{z} \expiconjkz\|_{L^2(\bndry)} \leq |{k}| \|{z}\expiconjkz\|_{L^2(\bndry)} \leq C |{k}| e^{L|k|},
\]
which yields the proof.
\end{proof}

The error estimate of the scattering transform is the crucial result for the reconstruction error, since the reconstructed conductivity depends continuously on the scattering transform as shown and utilized in several papers \cite{Knudsen2007,Knudsen2009}. Following this we can establish the linear dependence of the reconstructions.

\begin{theorem}[Reconstruction error]\label{theo:ReconError}
Let $\Omega, \Gamma, \sigma, R, \T^{\exp}_R, \widetilde{\T}^{\exp}_R$ satisfy the assumptions of Proposition \ref{prop:scattering}. Let $\sigma_R$ and $\widetilde{\sigma}_R$ be conductivities reconstructed from $\T^{\exp}_R$, $\widetilde{\T}^{\exp}_R$ respectively. Then, for $0<h\leq \pi$, there exists a constant $C = C(R)>0$ such that
\begin{equation} \label{eq:boundScatOverK}
\|\sigma_R-\widetilde{\sigma}_R\|_{L^2(\Omega)} \leq C h.
\end{equation}
\end{theorem}
\begin{proof}
We want to apply Lemma 2.1 from \cite{Knudsen2009}, that states continuous dependence of the solution to the D-bar equation on the scattering transform, for which we first need to show boundedness of $\T^{\exp}_R(k)/ \bar{k}$ and $\widetilde{\T}^{\exp}_R(k) / \bar{k}$. By the same arguments of the proof of Proposition \ref{prop:scattering} we have
\begin{equation}
\left| \frac{\T^{\exp}_R(k)}{\bar{k}}\right| \leq C |k| e^{L|k|}, \qquad \left| \frac{\widetilde \T^{\exp}_R(k)}{\bar{k}}\right| \leq C |k| e^{L|k|},
\end{equation}
for some constants $C,L>0$ and for $|k| \leq R$.
Now \cite[Lemma 2.1]{Knudsen2009} for the solutions of the D-bar equation gives
\begin{equation}\label{eq:Lemma21Use}
\|\mu_R(x,\cdot)-\tilde{\mu}_R(x,\cdot)\|_{C^\alpha(B_0(R))}\leq C\left\|\frac{\T^{\exp}_R-\widetilde{\T}^{\exp}_R}{\bar{k}}\right\|_{L^p\cap L^{p'}(B_0(R))},
\end{equation}
where the constant depends on the support of the scattering transform with $3/4<p<2$, $\alpha < 2/p -1$, and $1/p'=1-1/p$. The right hand of \eqref{eq:Lemma21Use} can be estimated by using Proposition \ref{prop:scattering}:
\begin{equation}\label{eq:scatOverK_lin}
\begin{split}
\left\|\frac{\T^{\exp}_R-\widetilde{\T}^{\exp}_R}{\bar{k}}\right\|_{L^p\cap L^{p'}(B_0(R))} \leq C(R) h.
\end{split}
\end{equation}
The reconstruction at a point $x\in\Omega$ is given by $\sigma_R(x)=\mu_R(x,0)^2$, $\widetilde{\sigma}_R(x)=\widetilde{\mu}_R(x,0)^2$ resp., and we obtain
\[
\begin{split}
\|\sigma_R-\widetilde{\sigma}_R\|_{L^2(\Omega)} &= \| \mu_R(\cdot,0)^2- \widetilde{\mu}_R(\cdot,0)^2 \|_{L^2(\Omega)}  \\
&= \| (\mu_R(\cdot,0)- \widetilde{\mu}_R(\cdot,0))(\mu_R(\cdot,0)+ \widetilde{\mu}_R(\cdot,0)) \|_{L^2(\Omega)}  \\
&\leq C \|\mu_R(\cdot,0)- \widetilde{\mu}_R(\cdot,0) \|_{L^2(\Omega)} \\
&\leq C \|\mu_R(\cdot,0)- \widetilde{\mu}_R(\cdot,0) \|_{L^\infty(\Omega)}.
\end{split}
\]
Further, for fixed $x\in\Omega$
\[\begin{split}
|\mu_R(x,0)- \widetilde{\mu}_R(x,0)| &\leq C \|\mu_R(x,\cdot)- \widetilde{\mu}_R(x,\cdot) \|_{L^\infty(B_0(R))} \\
 &\leq C \|\mu_R(x,\cdot)- \widetilde{\mu}_R(x,\cdot) \|_{C^\alpha(\R^2)}.
\end{split}
\]
The linear dependence \eqref{eq:boundScatOverK} in the claim follows from \eqref{eq:Lemma21Use} and \eqref{eq:scatOverK_lin}.
\end{proof}

\section{Computational results}\label{sec:computationalResults}
The first results presented in this section are for phantoms on the unit disk, for which we have proven linear convergence of the data error. We will verify that the estimates hold. In the following we will investigate how well the CGO solutions are preserved visually. Reconstructions are then presented and we verify the reconstruction error proved in Theorem \ref{theo:ReconError}. In the last part of our computations we will present a more realistic medical motivated Heart-and-Lungs phantom on a non-circular chest-shaped mesh. For better readability, we have included a short discussion at the end of each subsection. But first we need to address a few computational aspects.

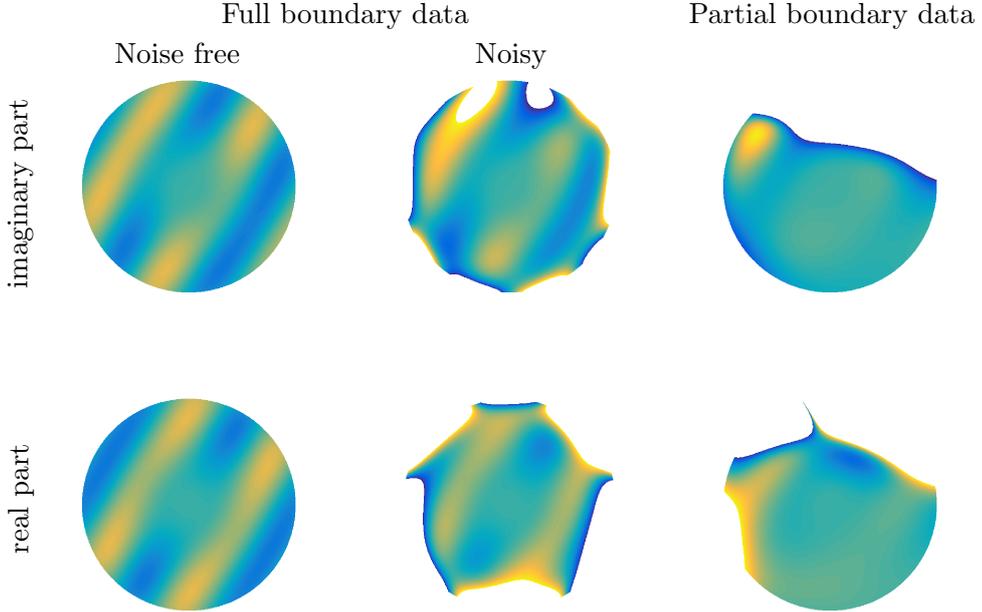
\begin{figure}[t!]
\centering
\begin{picture}(380,250)

\put(0,0){\includegraphics[width=150 pt]{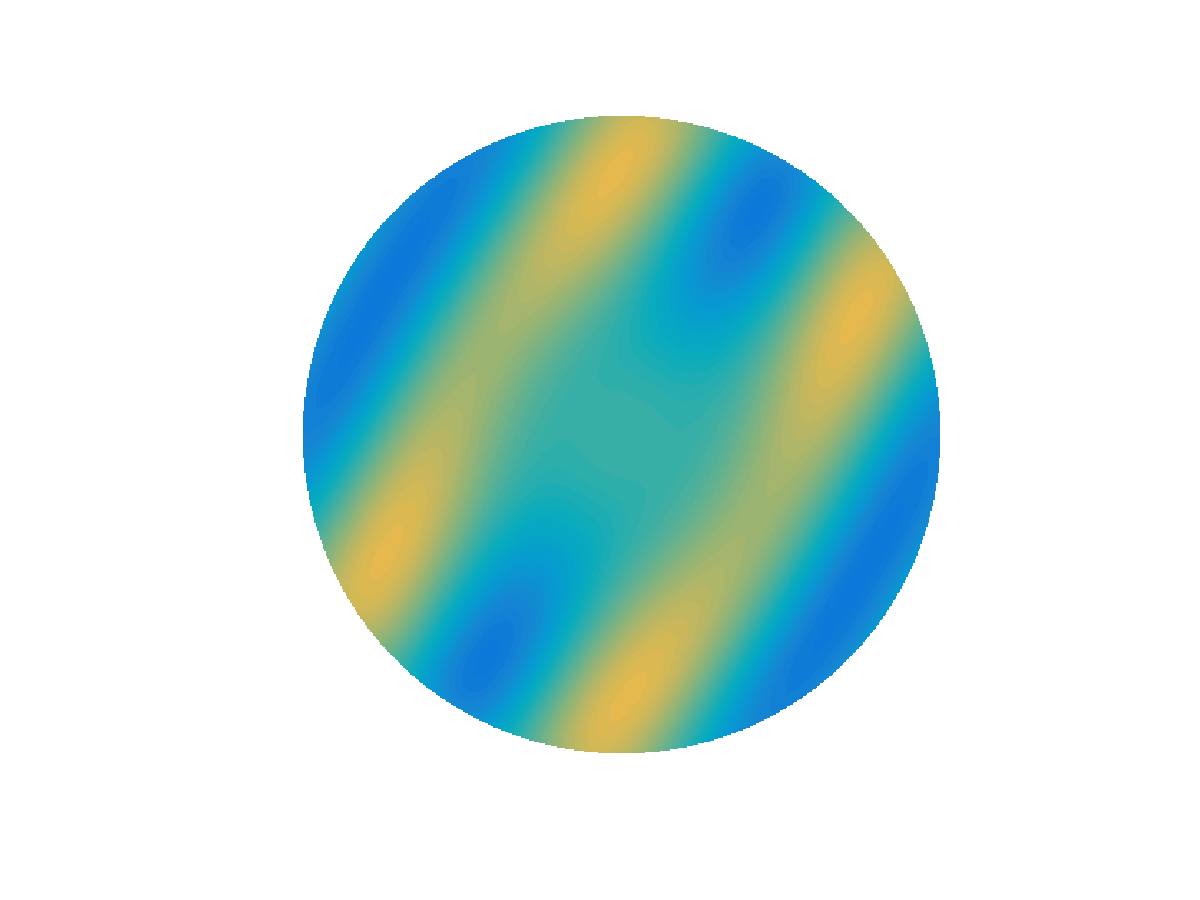}}
\put(0,120){\includegraphics[width=150 pt]{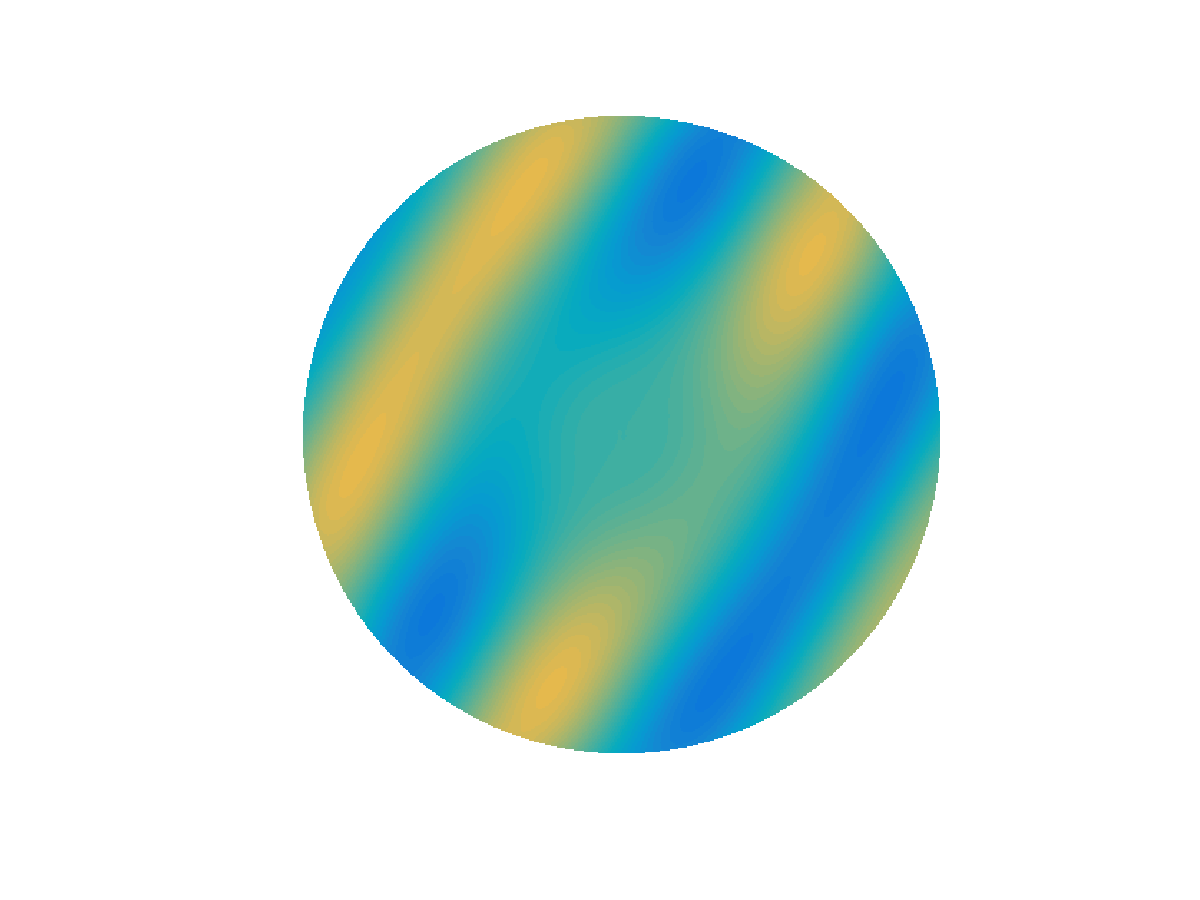}}

\put(120,0){\includegraphics[width=150 pt]{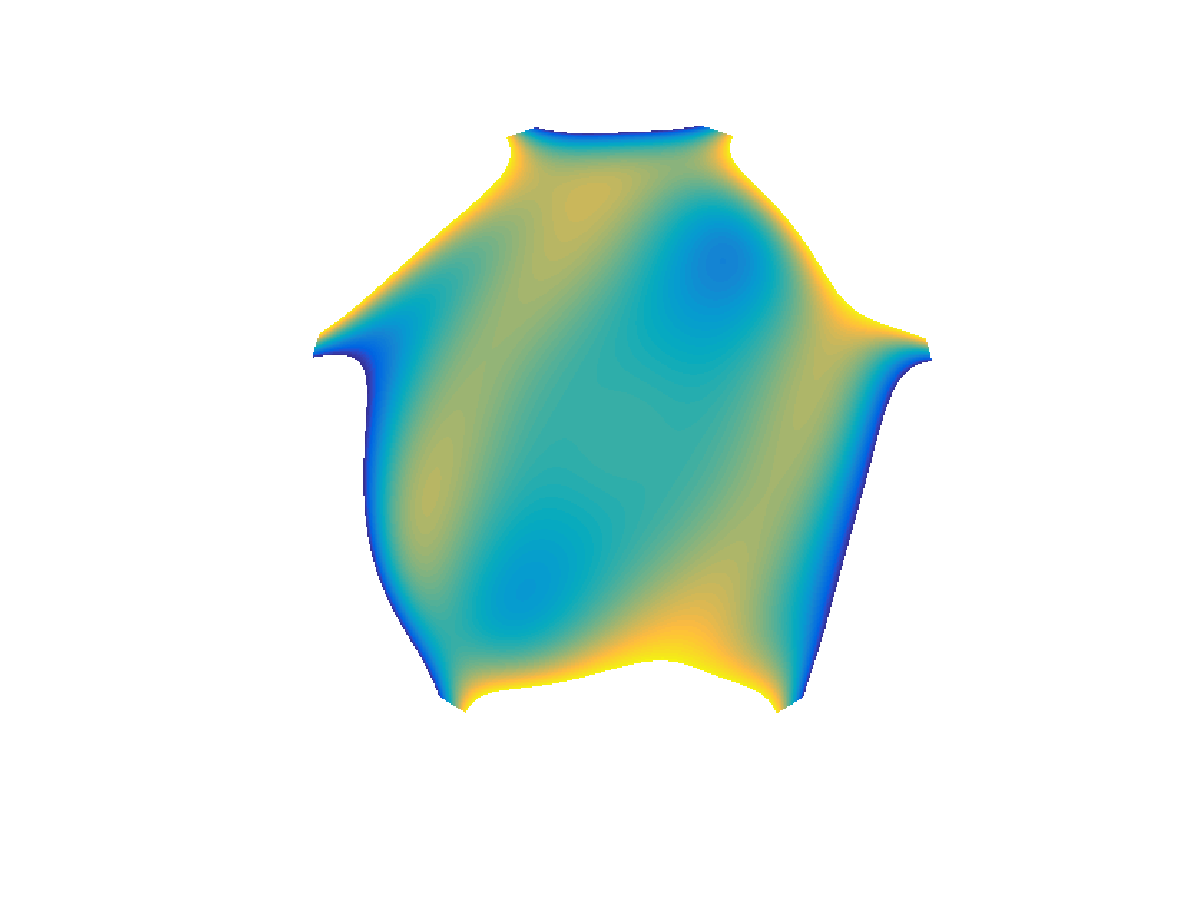}}
\put(120,120){\includegraphics[width=150 pt]{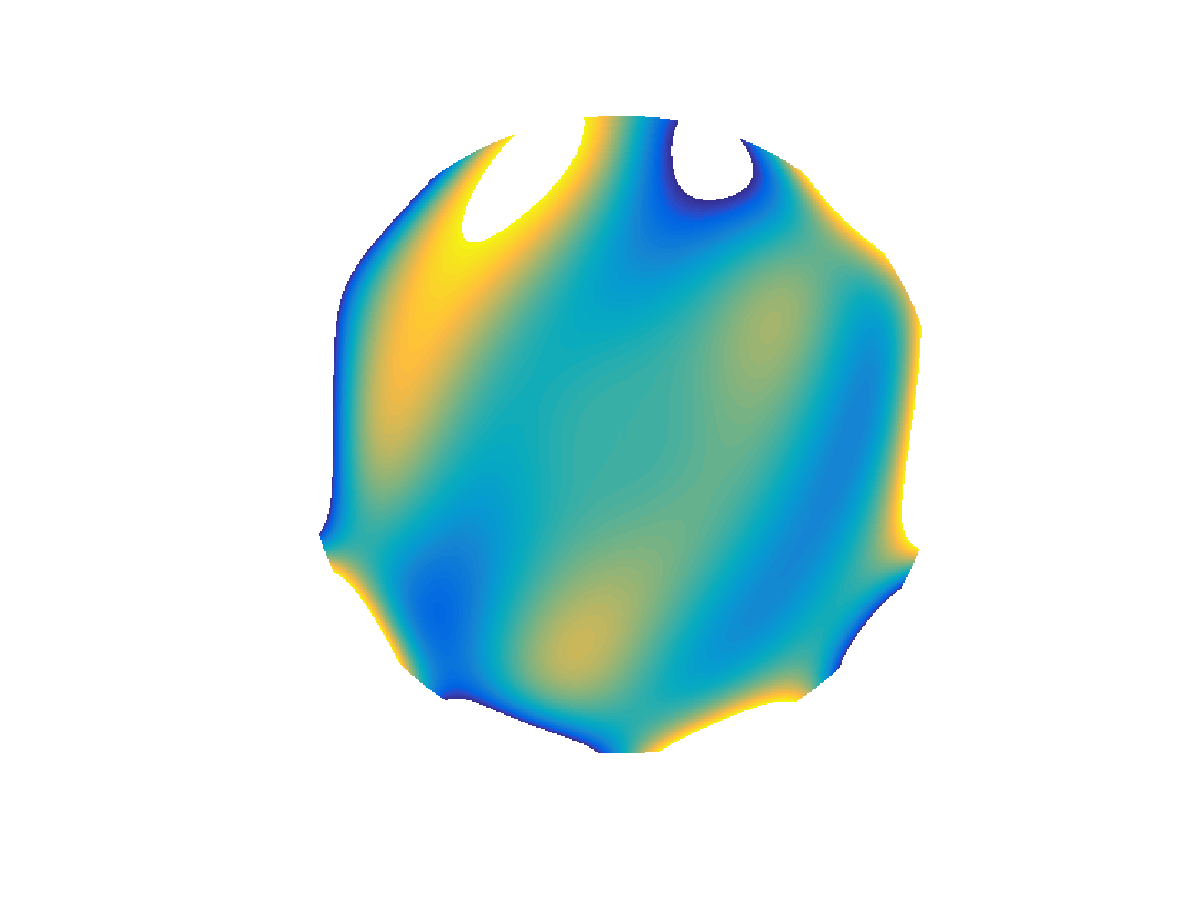}}

\put(240,0){\includegraphics[width=150 pt]{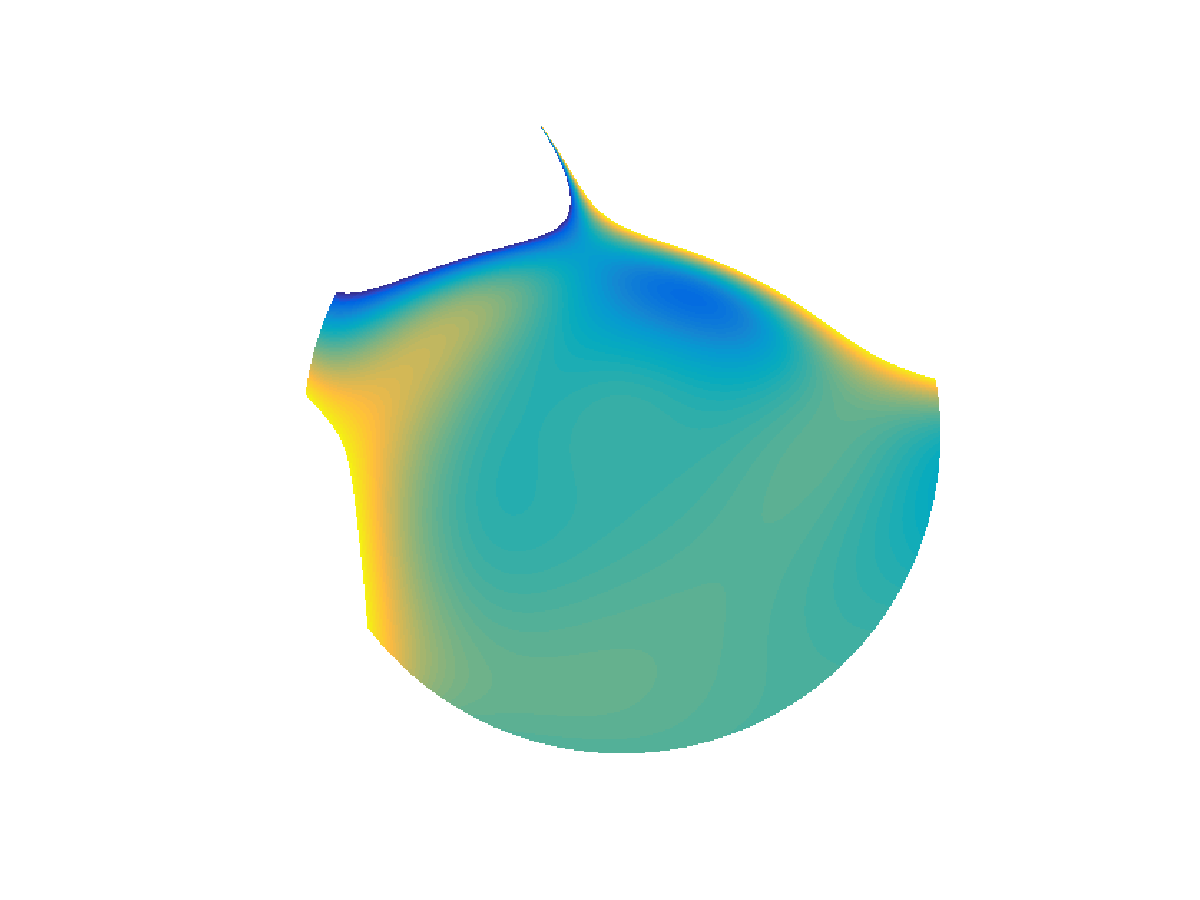}}
\put(240,120){\includegraphics[width=150 pt]{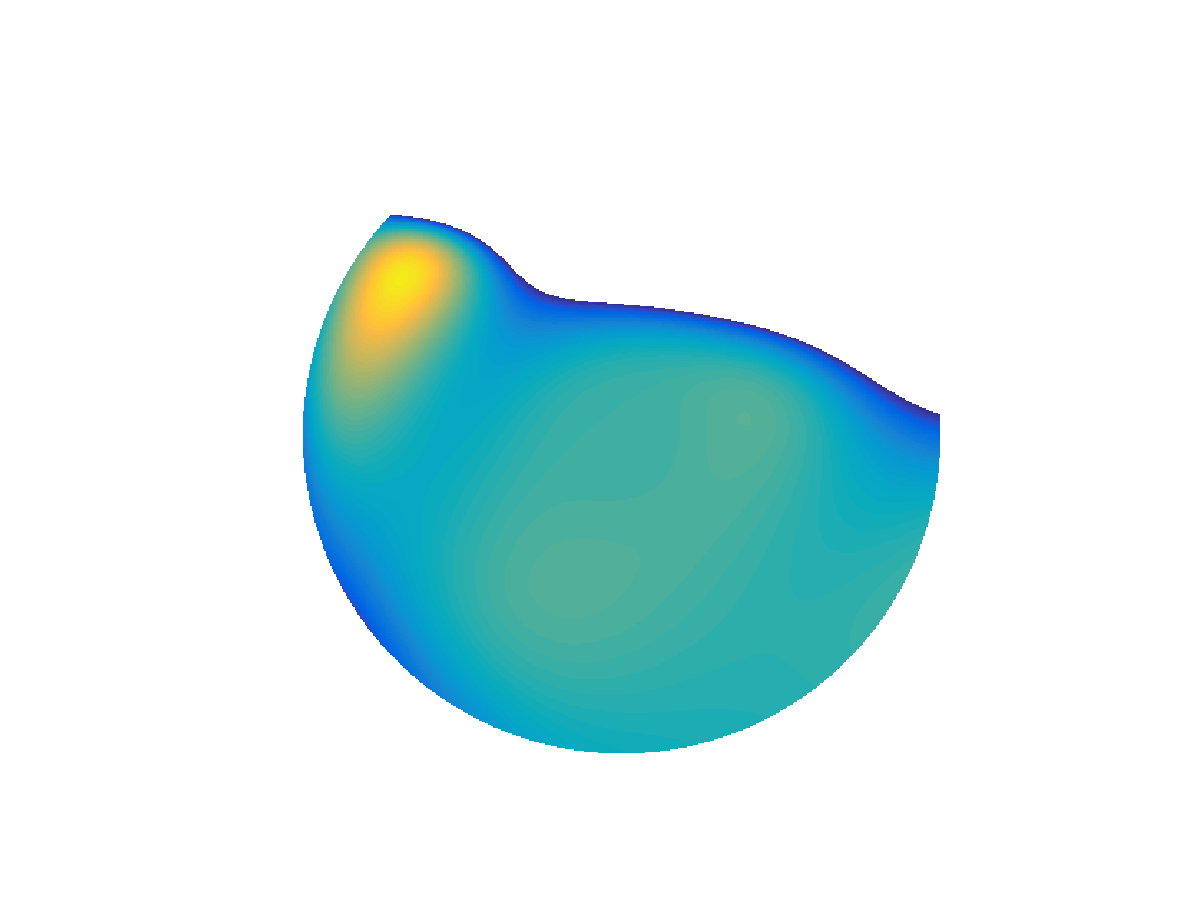}}

\put(90,240){Full boundary data}
\put(50,225){Noise free}
\put(185,225){Noisy}
\put(265,240){Partial boundary data}

\put(10,140){\rotatebox{90}{imaginary part}}
\put(10,40){\rotatebox{90}{real part}}

\end{picture}
\caption{\label{fig:scatTrafos} Computed scattering transform $\T^{\exp}(k)$ for $|k|<6$ of the circle phantom. The scattering transform for full-boundary data (Left) is stable in the disk of radius 6. Full boundary data with noise (Middle) causes instability. 
In the partial data case for 50\% of the boundary (Right) we see that the instability does not form in the typical circular manner. This suggests to use a cut-off radius combined with a threshold parameter.}
\end{figure} 

The equations \eqref{eq:psiExp} and \eqref{eq:scatExp} are defined for all $k\in\C\backslash\{0\}$. Obviously this is computationally not feasible and we need to restrict ourselves to solving on a disk of a radius $R$: this approach is a proven regularization technique \cite{Knudsen2009} in the DN case. The instability under noise can be easily observed in the scattering transform, see Figure \ref{fig:scatTrafos} for an illustration. For full-boundary data with noise, the instability forms outside a stable circle with radius $R>0$ such that $|k|<R$, hence for stabilizing the reconstruction procedure a cut-off radius is a well working approach. With partial-boundary data this is not the case any more as seen in the right column of Figure \ref{fig:scatTrafos}. Thus, we propose to use an additional threshold parameter $c>0$ and set
\[
\T^{\mathrm{ND}}_{R,c}(k)=
\left\lbrace\begin{array}{cl}
\T_R^{\mathrm{ND}}(k) &\text{ if } |\mathrm{Re}(\T_R^{\mathrm{ND}}(k))|< c\ \text{or}\ |\mathrm{Im}(\T_R^{\mathrm{ND}}(k))| < c.\\
0 &\mbox{ else. }
\end{array} \right.
\]
This approximate scattering transform will then be used to solve the D-bar equation \eqref{eq:dbark_eq} to obtain a reconstruction 
\[
\sigma^{R,c}(z)=(\mu^{R,c}(z,0))^2.
\]

\subsection{Verifying the error estimates}
In this section we verify that the error estimates from Section \ref{sec:err_analy} hold also numerically. From Proposition \ref{prop:convergence_single} we have the linear error estimates for single currents $\varphi_n(\theta)=\frac{1}{\sqrt{2\pi}} e^{in\theta}$. 

At first we are interested if the asymptotic behaviour holds for the constant conductivity $\sigma\equiv 1$, i.e. the classical Laplace equation. The error of the scaling basis depends also on the order $n\in\Z\backslash \{0\}$. Since we compute only for small $n$, this has no notable effect on the computational convergence. 
The computed relative errors can be seen in Figure \ref{fig:conv_laplace_single}, on the left for the scaling partial-boundary map and on the right for the cut-off case. We emphasize that for this error estimate we assume to know the measurement on the whole boundary $\bndry$, whereas the Neumann data is only supported on $\Gamma$. The error is computed with respect to the known analytical solution
\[
\ND_\sigma\varphi_n(\theta)=\frac{1}{|n|\sqrt{2\pi}}e^{in\theta}.
\]
\begin{figure}[ht!]
\centering
\begin{picture}(350,180)
\put(181,-10){\includegraphics[width=250 pt]{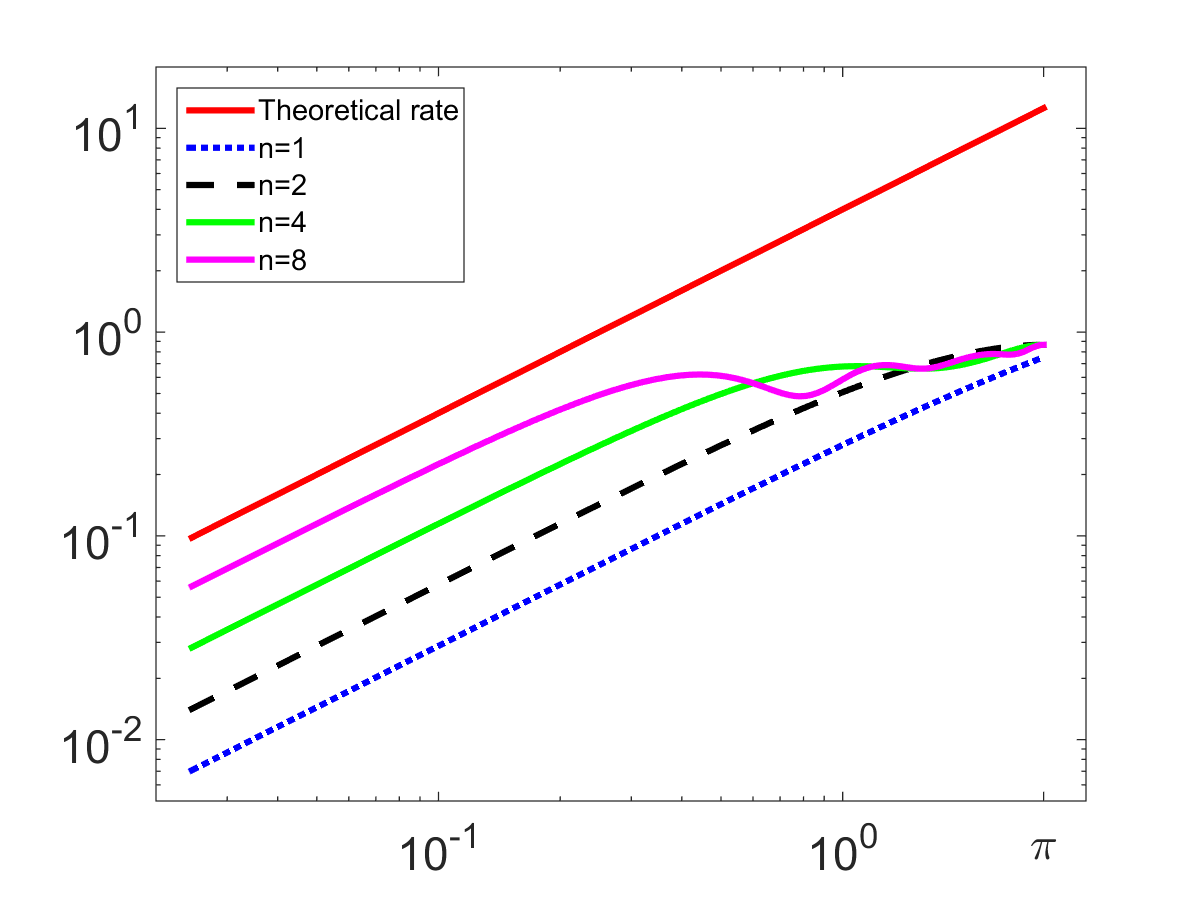}}
\put(-56,-10){\includegraphics[width=250 pt]{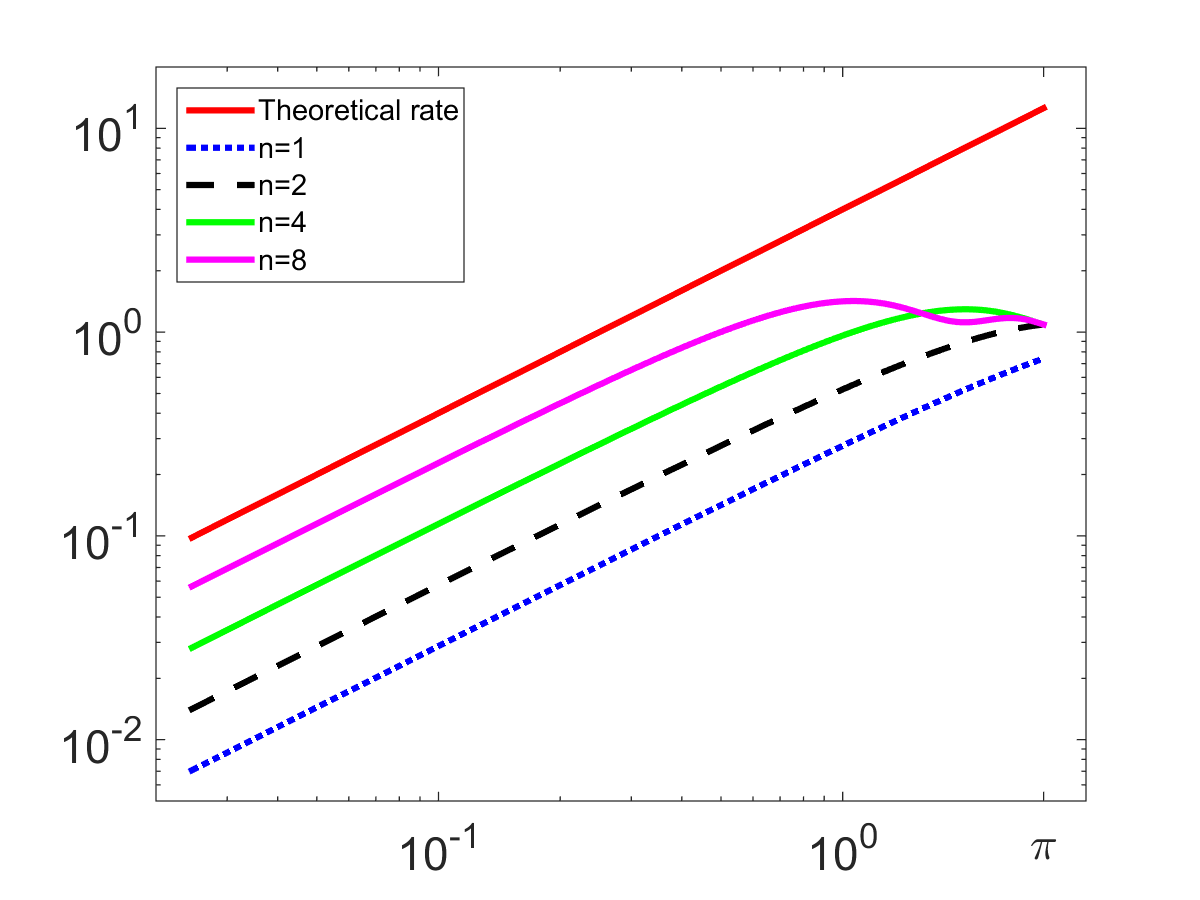}}
\put(75,-5){\large h}
\put(310,-5){\large h}
\put(-55,35){\rotatebox{90}{\footnotesize $\|(\widetilde{\ND}^s_1-\ND_1)\varphi_n\|_2/\|\ND_1\varphi_n\|_2$ }}
\put(180,35){\rotatebox{90}{\footnotesize $\|(\widetilde{\ND}_1^c-\ND_1)\varphi_n\|_2/\|\ND_1\varphi_n\|_2$ }}
\put(28,175){Data error: Scaling}
\put(265,175){Data error: Cut-off}
\end{picture}
\caption{\label{fig:conv_laplace_single} Relative error of measurements from the {partial ND map} to the full-boundary ND map. Computations are for the Laplace equation and different order $n$ of the current patterns.  On the left for the scaling and on the right for the cut-off partial-boundary map.}
\end{figure}
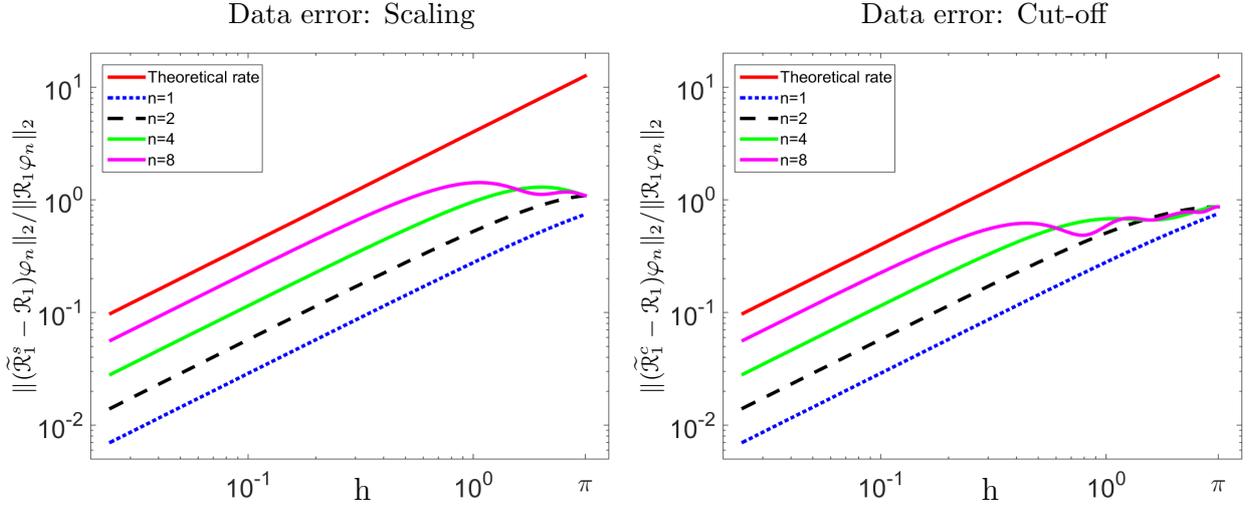 

The next test is to verify that the convergence rates are preserved under conductivities different to 1. We use a circular inclusion and the Heart-and-Lungs phantom for comparison. The computed relative errors can be seen in Figure \ref{fig:conv_sigma_single}. These computations are still for ideal full-boundary measurements. Here we computed the reference full-boundary measurements with high accuracy for the computation of the error.

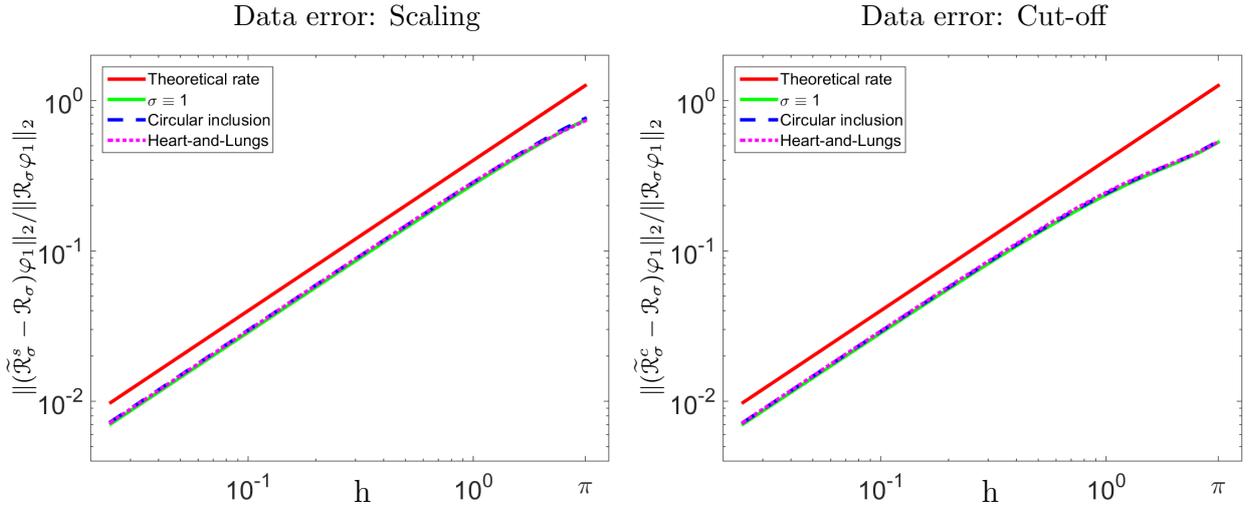
\begin{figure}[ht!]
\centering
\begin{picture}(350,180)
\put(181,-10){\includegraphics[width=250 pt]{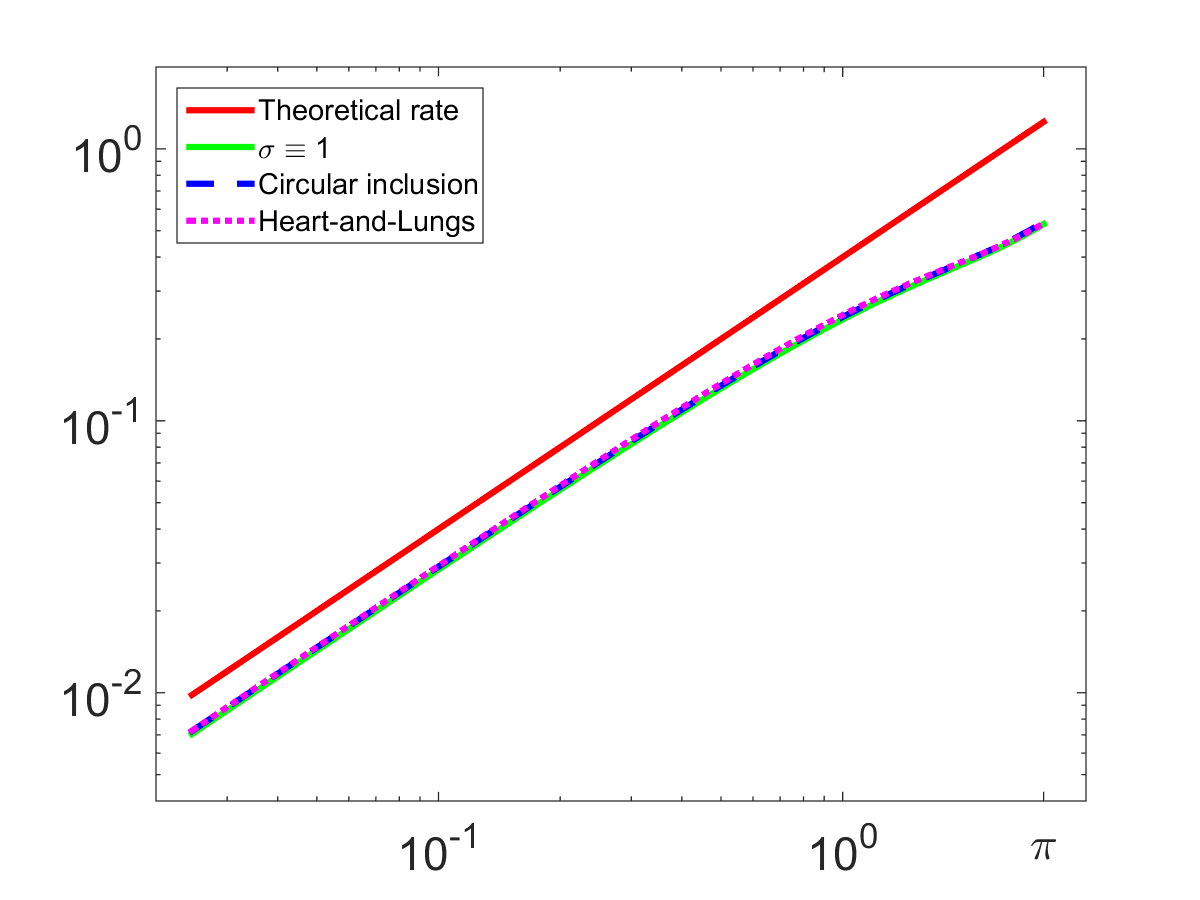}}
\put(-56,-10){\includegraphics[width=250 pt]{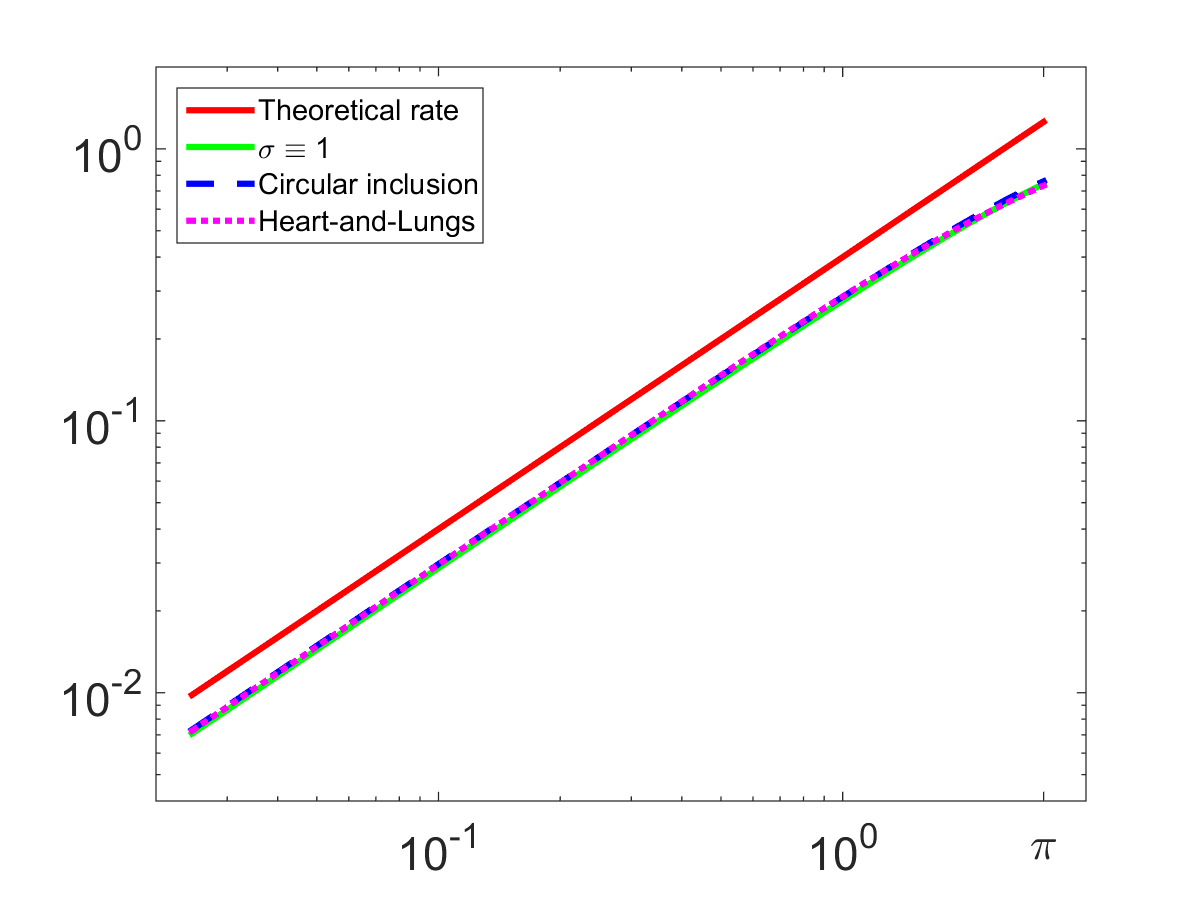}}
\put(75,-5){\large h}
\put(310,-5){\large h}
\put(-55,35){\rotatebox{90}{\footnotesize $\|(\widetilde{\ND}^s_\sigma-\ND_\sigma)\varphi_1\|_2/\|\ND_\sigma\varphi_1\|_2$ }}
\put(180,35){\rotatebox{90}{\footnotesize $\|(\widetilde{\ND}^c_\sigma-\ND_\sigma)\varphi_1\|_2/\|\ND_\sigma\varphi_1\|_2$ }}
\put(30,175){Data error: Scaling}
\put(265,175){Data error: Cut-off}
\end{picture}
\caption{\label{fig:conv_sigma_single} Relative error of measurements from the {partial ND map} to the full-boundary ND map for different conductivities $\sigma$ and fixed order $n=1$. We compare the constant conductivity, a circular inclusion close to the boundary $\Gamma$ and the Heart-and-Lungs phantom. On the left for the scaling and on the right for the cut-off partial-boundary map.}
\end{figure}

\begin{figure}[h!]
\centering
\begin{picture}(350,180)
\put(181,-10){\includegraphics[width=250 pt]{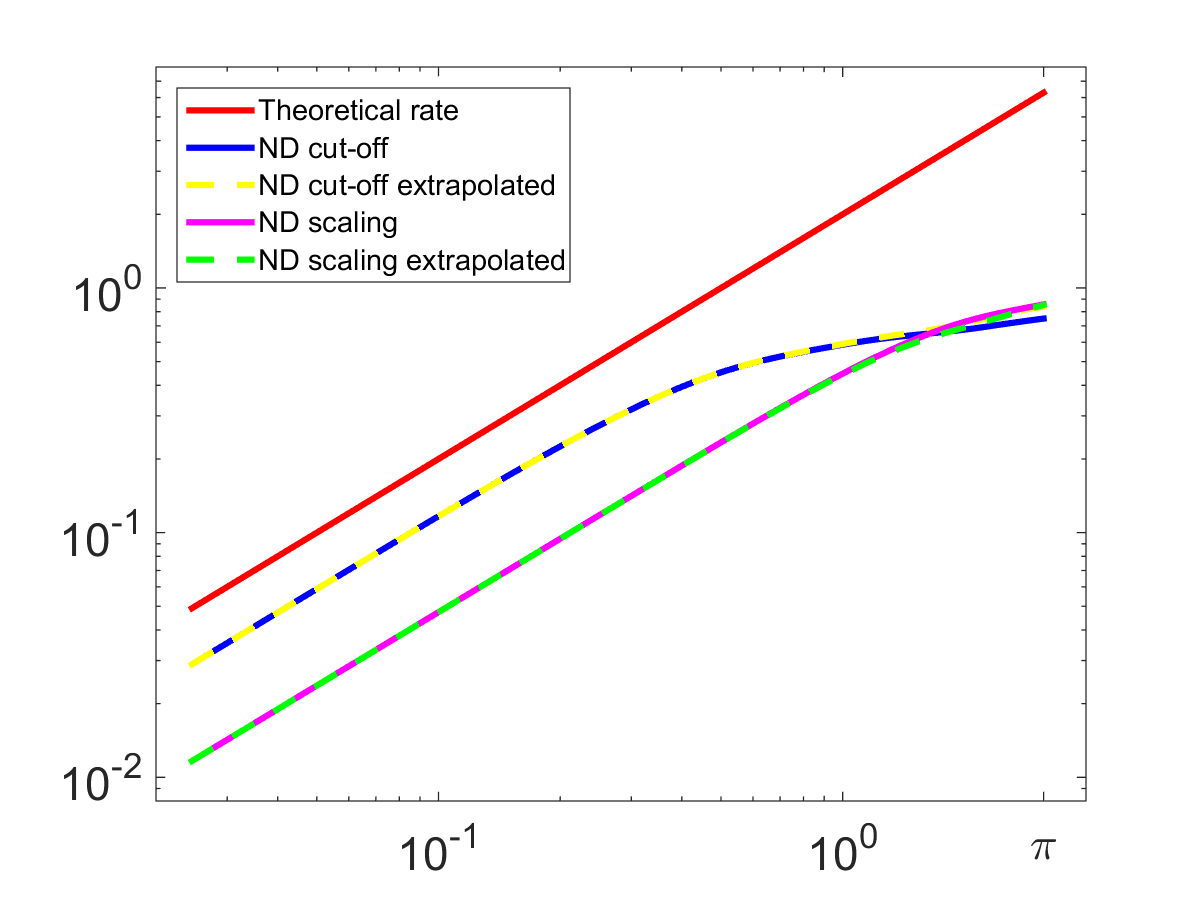}}
\put(-56,-10){\includegraphics[width=250 pt]{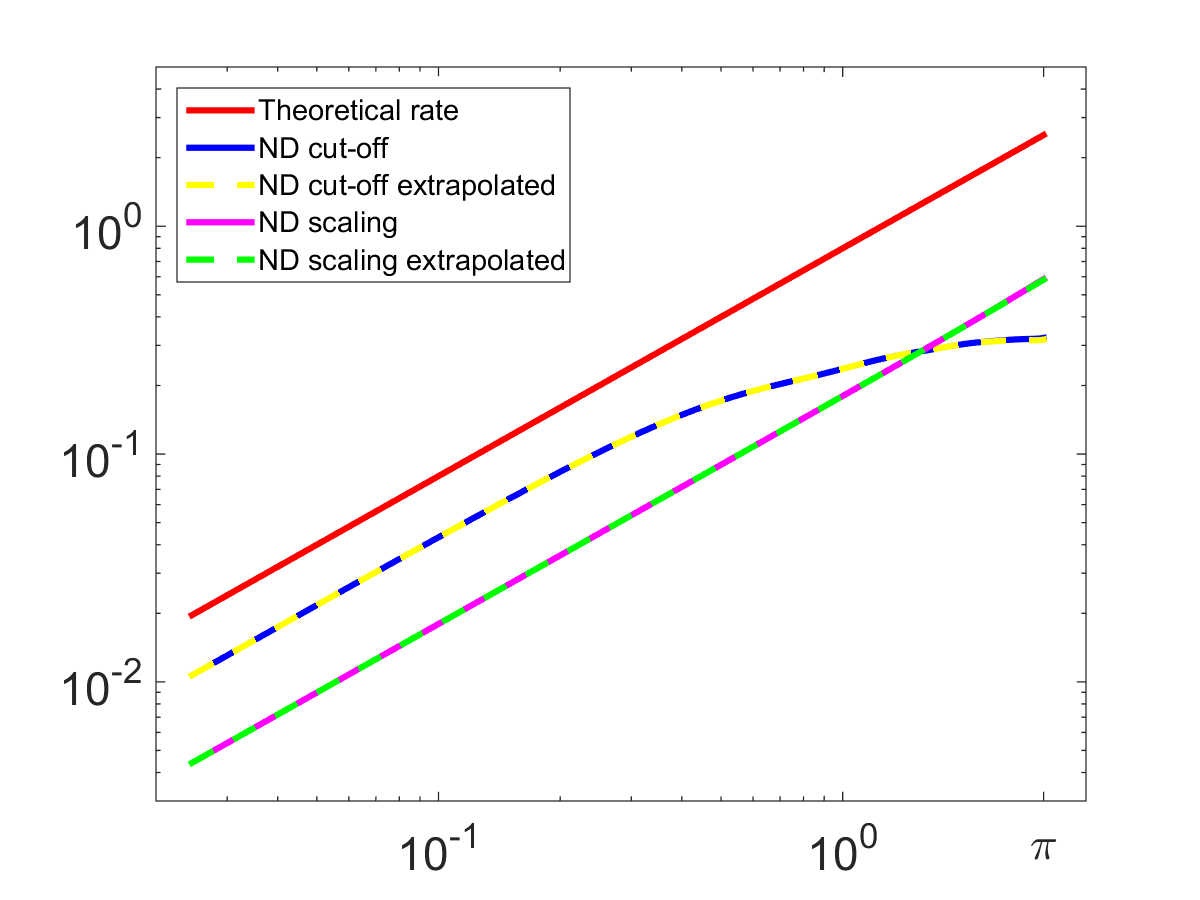}}
\put(75,-5){\large h}
\put(310,-5){\large h}
\put(-55,40){\rotatebox{90}{\footnotesize $\|\widetilde{\mathbf{R}}_{\sigma,1}-\mathbf{R}_{\sigma,1}\|_2/\|\mathbf{R}_{\sigma,1}\|_2$ }}
\put(180,40){\rotatebox{90}{\footnotesize $\|\widetilde{\mathbf{R}}_{\sigma,1}-\mathbf{R}_{\sigma,1}\|_2/\|\mathbf{R}_{\sigma,1}\|_2$ }}
\put(7,175){Data error: Circular inclusion}
\put(245,175){Data error: Heart-and-Lungs}
\end{picture}
\caption{\label{fig:conv_NDmatrices} Relative error of the difference partial ND matrices to the full-boundary ND matrix for both partial-boundary maps. On the left for the circular inclusion and on the right for the Heart-and-Lungs phantom, with comparison of ideal full-boundary information and extrapolated data.}
\end{figure} 

At last we compute the partial ND matrices $(\widetilde{\mathbf{R}}_\sigma)_{n,\ell}=(\widetilde{\ND}_\sigma \varphi_n,\varphi_\ell)$, 
with 16 basis functions, such that $n,\ell \in \{\pm 1,\pm 2,\dots,\pm 8\}$. The error estimate for the difference {partial ND map} has been stated in Corollary \ref{cor:DiffND} and can be directly used to estimate the error of computed ND matrices. The difference partial ND matrices $\widetilde{\mathbf{R}}_{\sigma,1}$ are computed from ideal measurements and from extrapolated boundary traces. The resulting relative errors are plotted in Figure \ref{fig:conv_NDmatrices}.

In all three examples the convergence for small $h$ has linear behaviour and follows the predicted rate. In Figure \ref{fig:conv_laplace_single} one can see that for higher order $n$ the convergence breaks down for large $h$ due to numerical instabilities. For smaller order the convergence rate is preserved up to 50\% of missing boundary, that is $|\Gamma^c|=\pi$. For the three different conductivities, the convergence graphs have very similar shape as seen in Figure \ref{fig:conv_sigma_single}, only the constants of the error term change slightly. Introducing the extrapolation for computing the ND matrices, we can see in Figure \ref{fig:conv_NDmatrices} that the error is preserved and only a slight difference for large $h$ occurs due to inaccuracies in the estimation process.

\subsection{Computing the CGO solutions}
We compute the CGO solutions directly from the ND map, as discussed in Section \ref{sec:reconError}. The approximate CGO solutions $\psi^{\mathrm{ND}}(z,k)$ can then be computed for fixed $k\in\C\backslash\{0\}$ from \eqref{eq:psiExp}.
For illustrating the results we use the concept of the CGO sinogram in the unit disk as introduced in \cite{Hamilton2014}. That is for $\varphi, \theta\in[0,2\pi]$ and $r>0$ the CGO sinogram is given by
\begin{equation}
\begin{split}
S_\sigma(\theta,\varphi,r)&=\mu(e^{i\theta},e^{i\varphi})-1 \\
&=\exp(-ire^{i(\varphi+\theta)})\psi(e^{i\theta},e^{i\varphi})-1.
\end{split}
\end{equation}
For a fixed radius $r>0$ the CGO sinogram consists of all CGO solutions with $|k|=r$ and hence is especially suitable for illustration. In Figure \ref{fig:CGOsol} one can see the CGO sinogram for the circular inclusion phantom and the radius $r=2$. The reference solution is computed by solving the Lippmann-Schwinger type equation \eqref{eq:LippSchw} for $\mu$ directly from the knowledge of the conductivity. The approximate CGO solutions are computed from the ND map by \eqref{eq:psiExp}. We compare the results for full-boundary data to measurements from 75\% of the boundary for both partial-boundary maps.

\begin{figure}[ht!]
\centering
{
\begin{picture}(400,200)


\put(280,100){\includegraphics[width=125 pt]{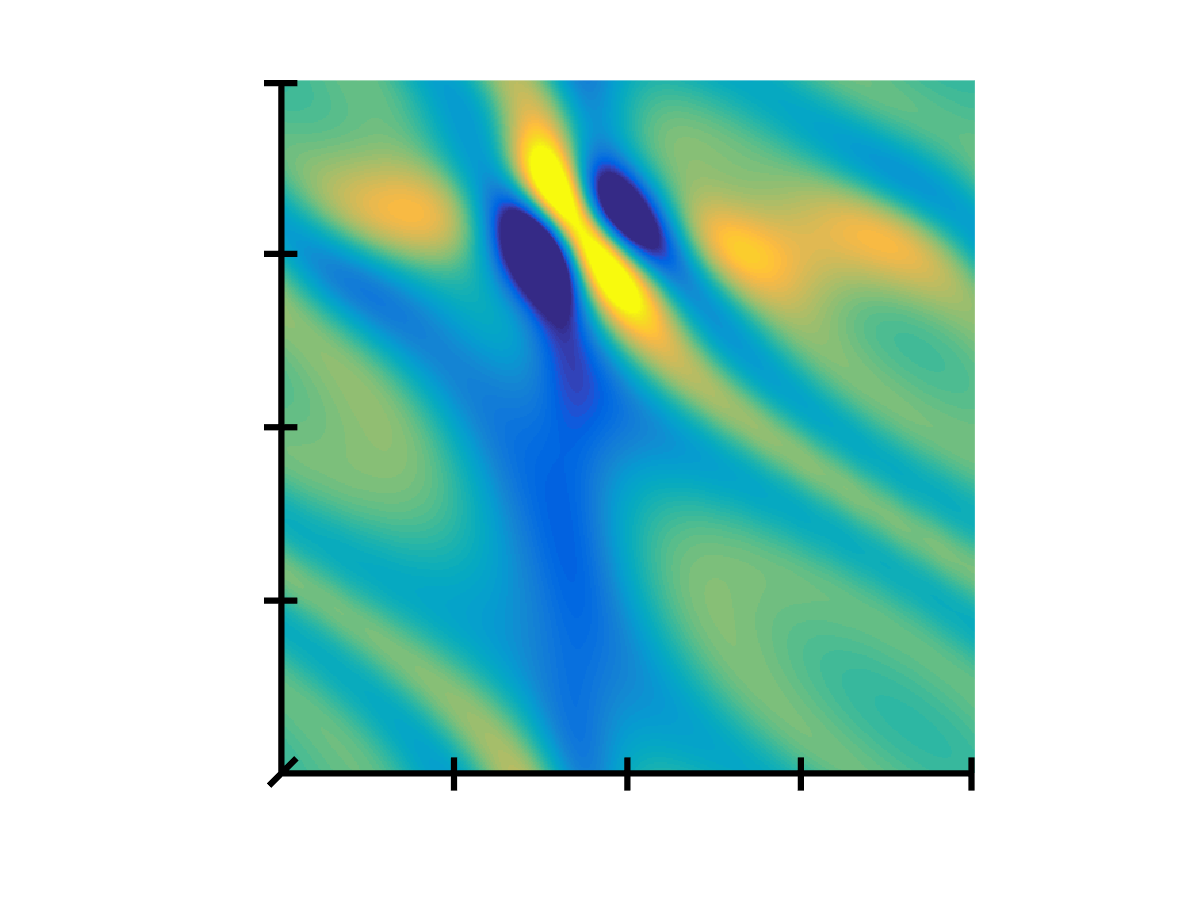}}
\put(280,0){\includegraphics[width=125 pt]{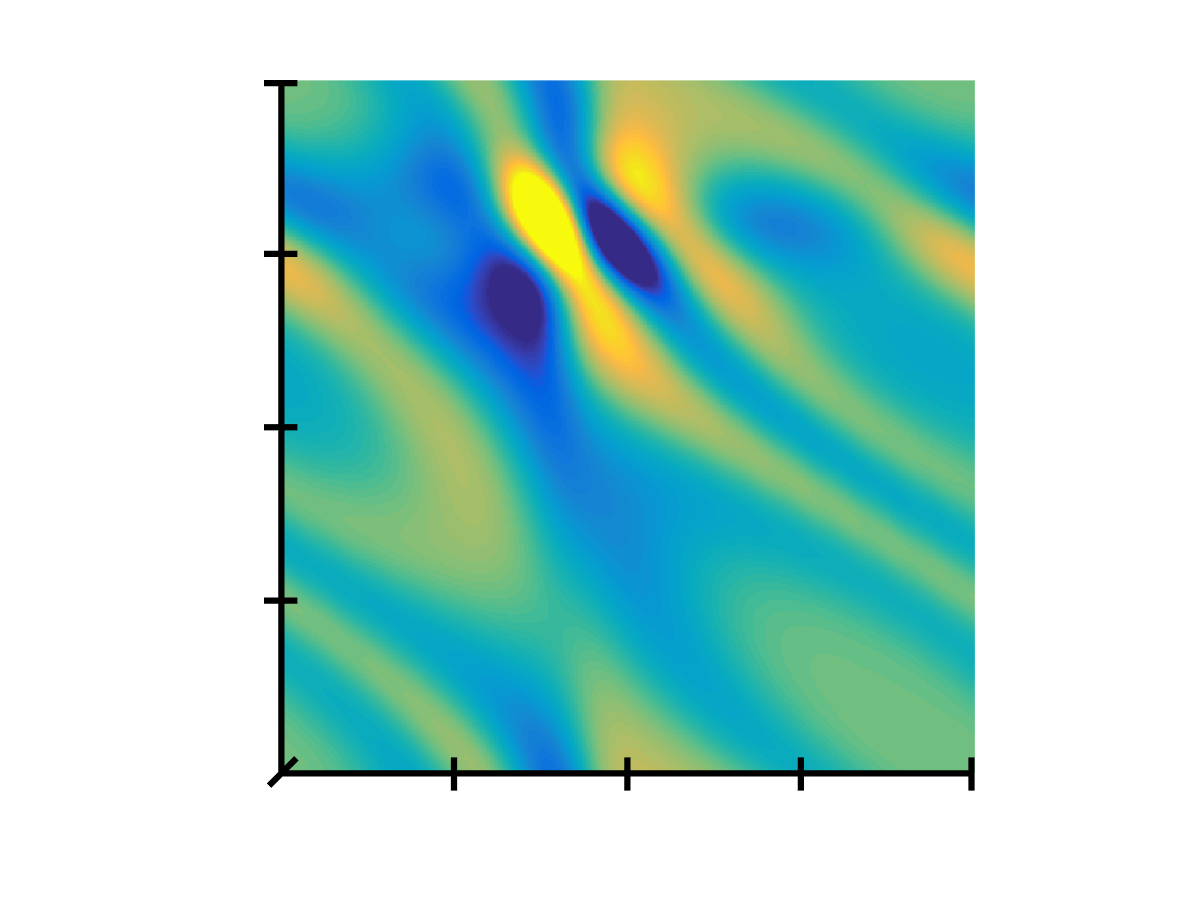}}

\put(170,100){\includegraphics[width=125 pt]{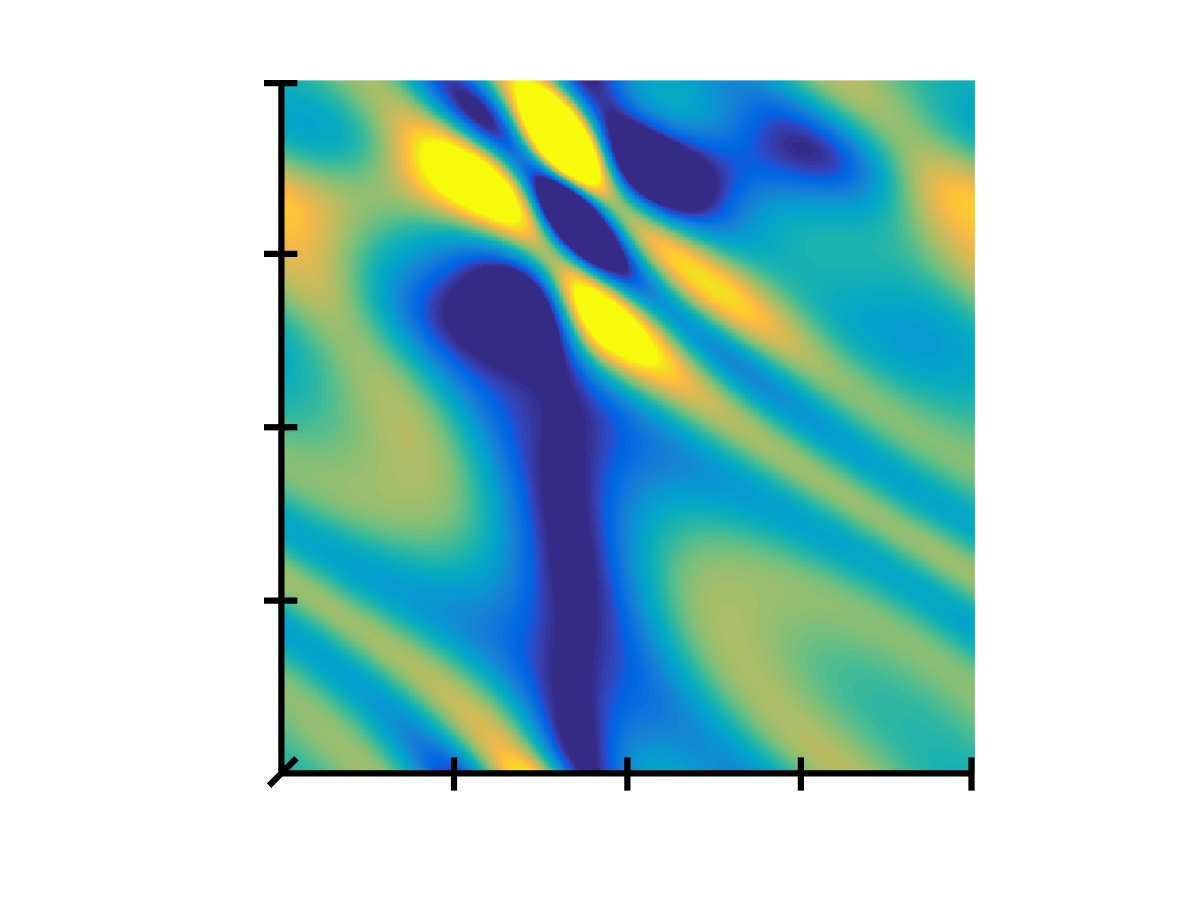}}
\put(170,0){\includegraphics[width=125 pt]{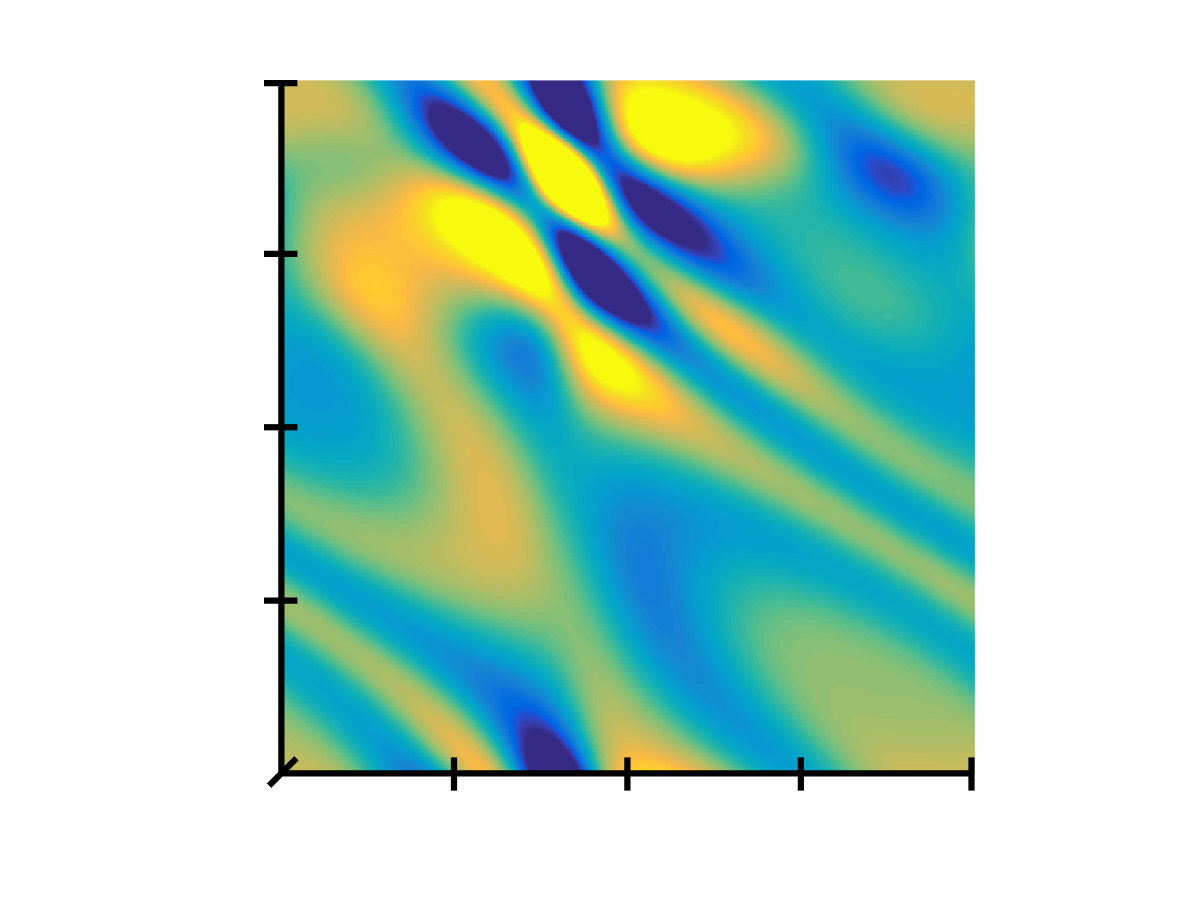}}

\put(60,100){\includegraphics[width=125 pt]{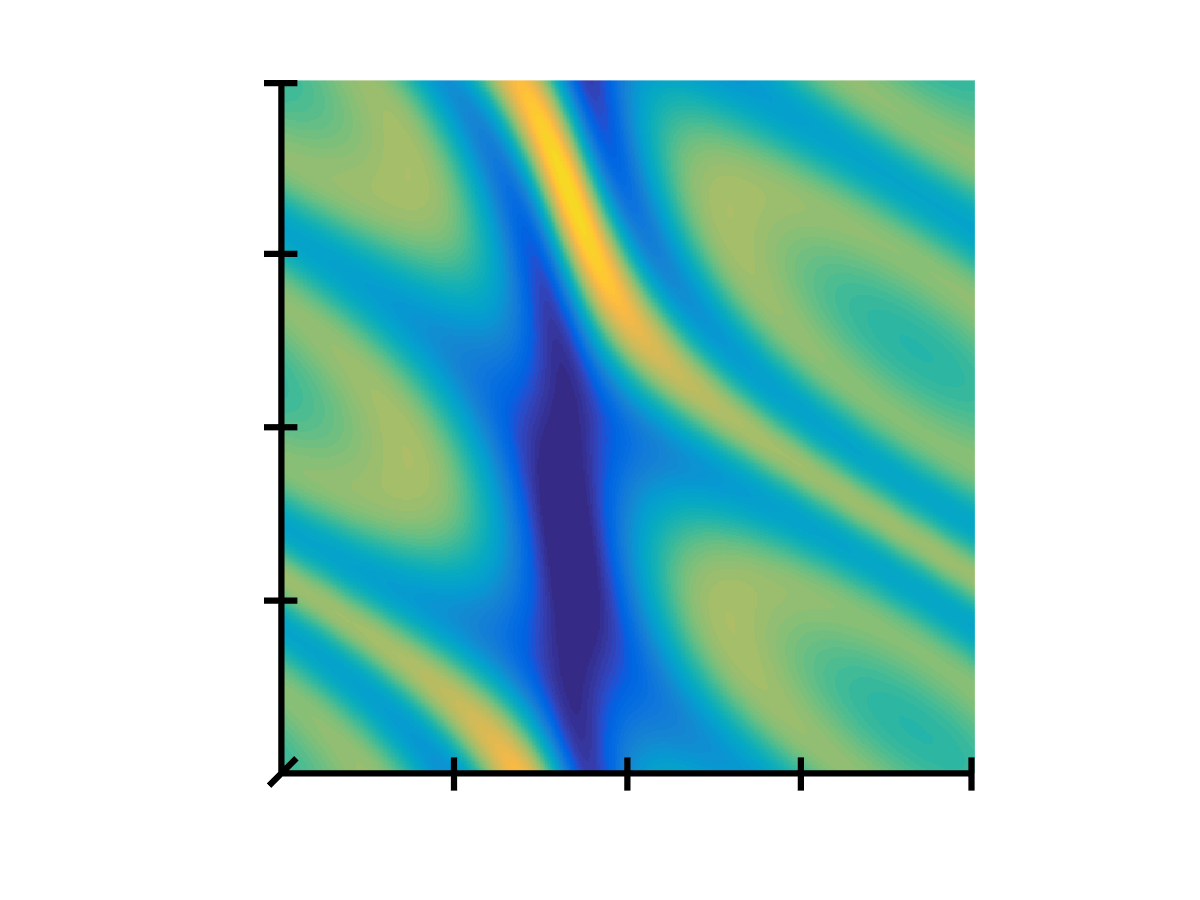}}
\put(60,0){\includegraphics[width=125 pt]{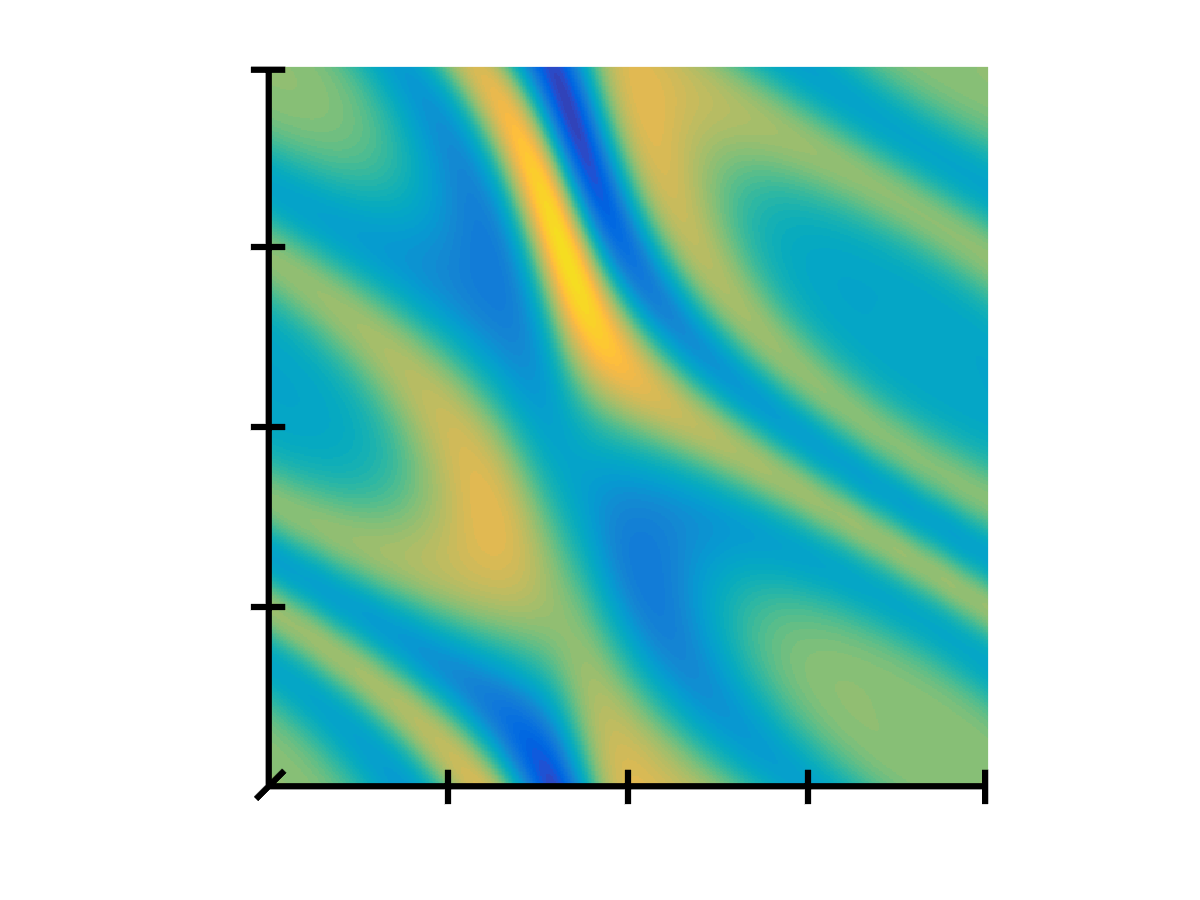}}

\put(-50,100){\includegraphics[width=125 pt]{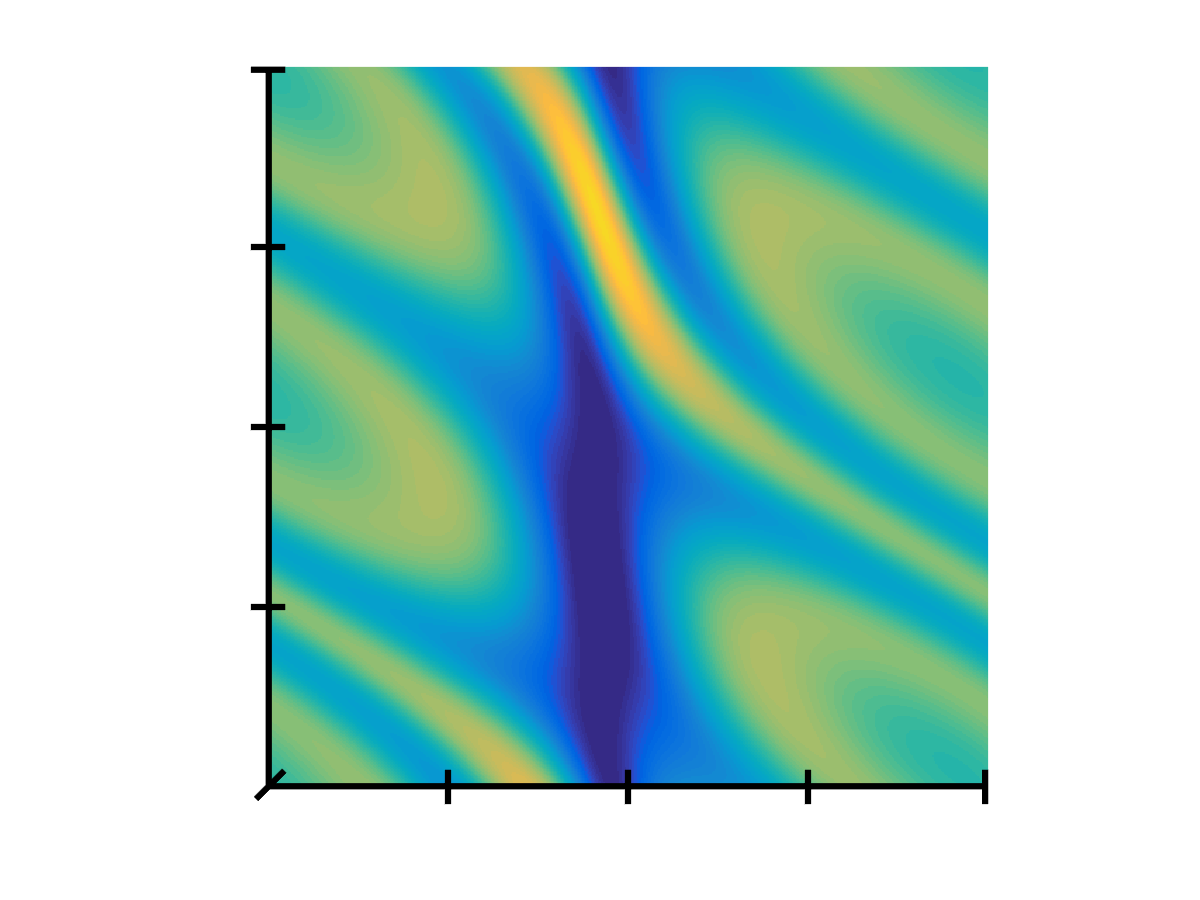}}
\put(-50,0){\includegraphics[width=125 pt]{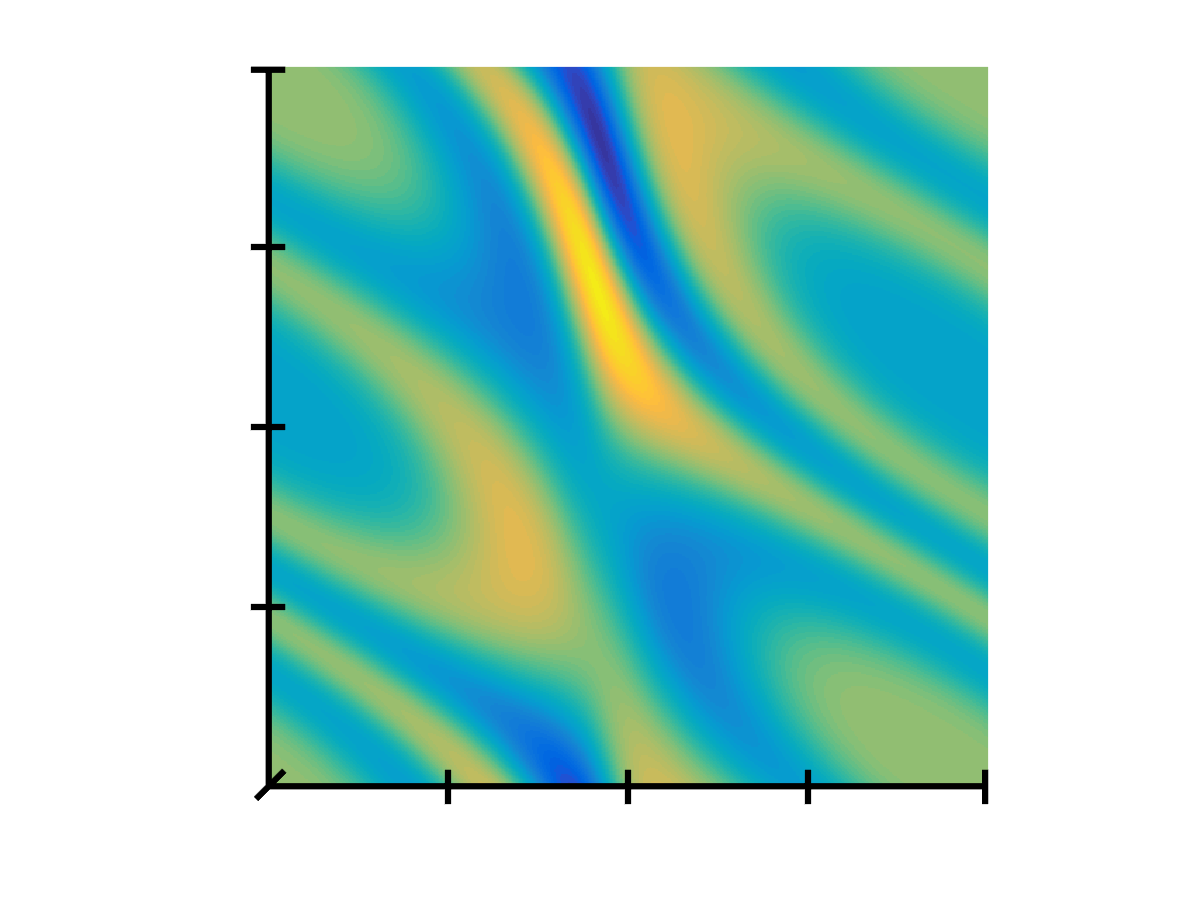}}

\put(-10,195){Reference}
\put(390,153){\rotatebox{-90}{real}}
\put(390,70){\rotatebox{-90}{imaginary}}
\put(-50,35){\large \rotatebox{90}{$\arg(k)$}}
\put(0,-10){\large {$\arg(z)$}}
\put(90,195){Full boundary}
\put(-32,3){0}
\put(48,0){$2\pi$}
\put(12,0){$\pi$}
\put(-35,84){$2\pi$}
\put(-32,47){$\pi$}
\put(207,195){75\% cut-off}
\put(315,195){75\% scaling}

\end{picture} }
\caption{\label{fig:CGOsol} Computed CGO sinogram $S_\sigma(\theta,\varphi,r)$ for $|k|=r=2$ of a circular inclusion. (Left) Reference by solving the Lippmann-Schwinger type equation for $\mu$. The solutions to the right are computed from \eqref{eq:psiExp} for full-boundary data and 75\% of the boundary available. We plotted the angular variable $\theta$ for $z$ on the x-axis and $\varphi$ for $k$ on the y-axis of each image.}
\end{figure}
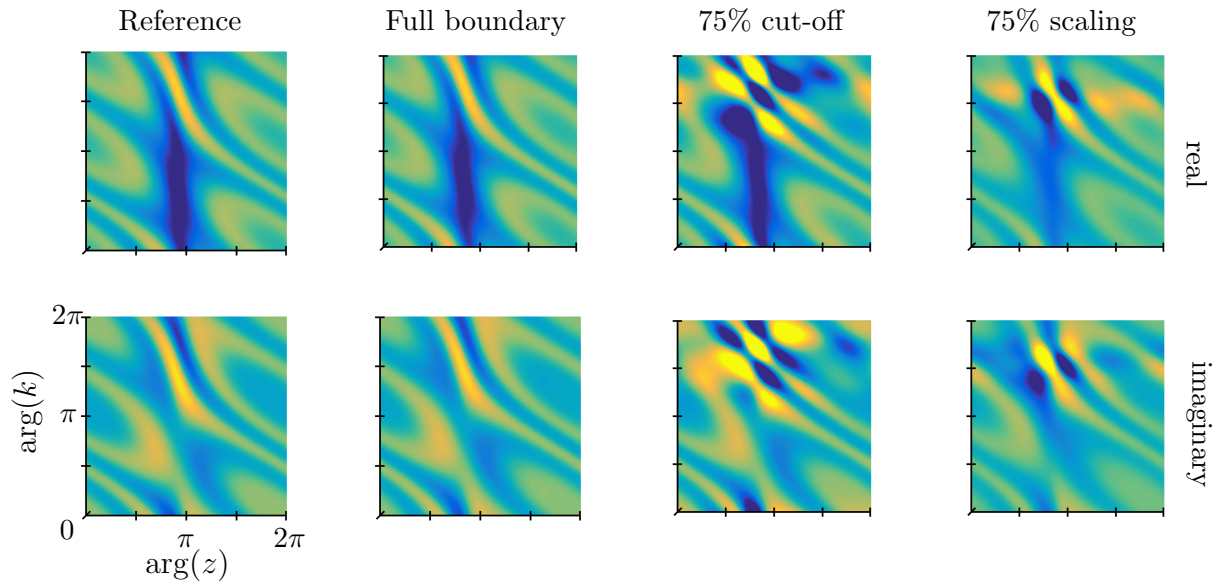

From Figure \ref{fig:CGOsol} one can see in the second column that approximating the CGO solutions by \eqref{eq:psiExp} seems to be reasonable. The characteristics are well preserved and only minimal differences can be seen. Restricting the input currents and measurements to 75\% of the boundary introduces clear instabilities in the CGO solutions. It is interesting to note that for $\varphi\in[0,3\pi/2]$ the solutions are rather stable and the instabilities occur mostly for the argument $\varphi>3\pi/2$.

\subsection{Reconstructions on the unit disk}
The first test case is a simple circular inclusion located close to the measurement boundary $\Gamma$, with radius $r=0.15$, conductivity value $2$, and smoothed at the inclusion boundary. We have computed reconstructions with decreasing percent of the boundary available. The measured voltages are extrapolated with the cubic spline approach as described in \ref{sec:extrapolation}. The partial ND matrices are formed from 32 basis functions and the reconstructions are computed by using the Born approximation approach for the scattering transform and then solving the D-bar equation to obtain the reconstructions in Figure \ref{fig:circleRecons}. The radius for computing the scattering transform and threshold values are presented in Table \ref{tab:scatVals}.

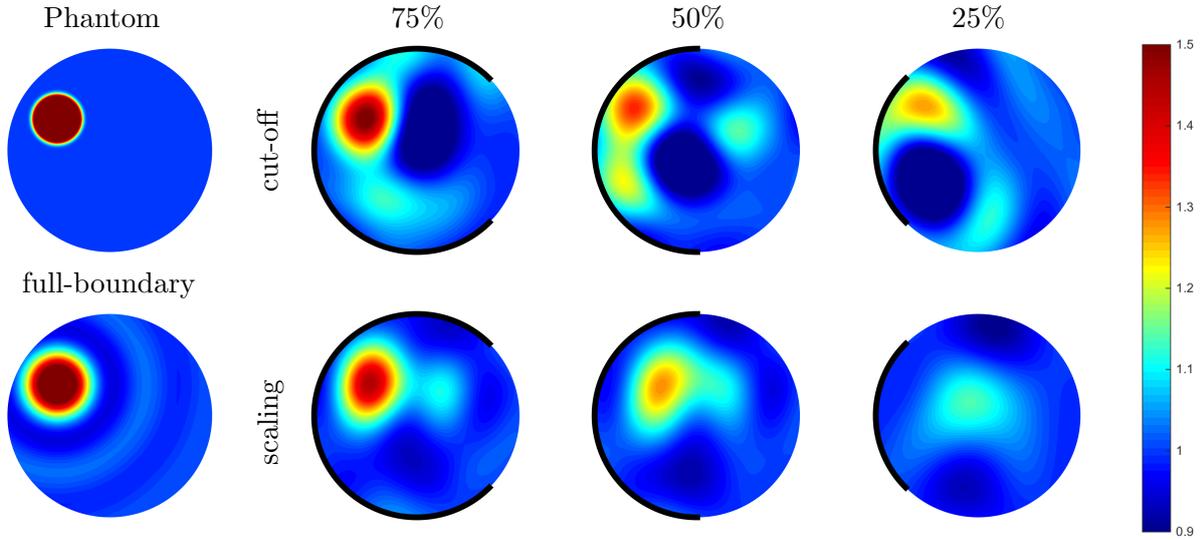
\begin{figure}[ht!]
\centering
\ {
\begin{picture}(400,200)

\put(150,-20){\includegraphics[width=300 pt]{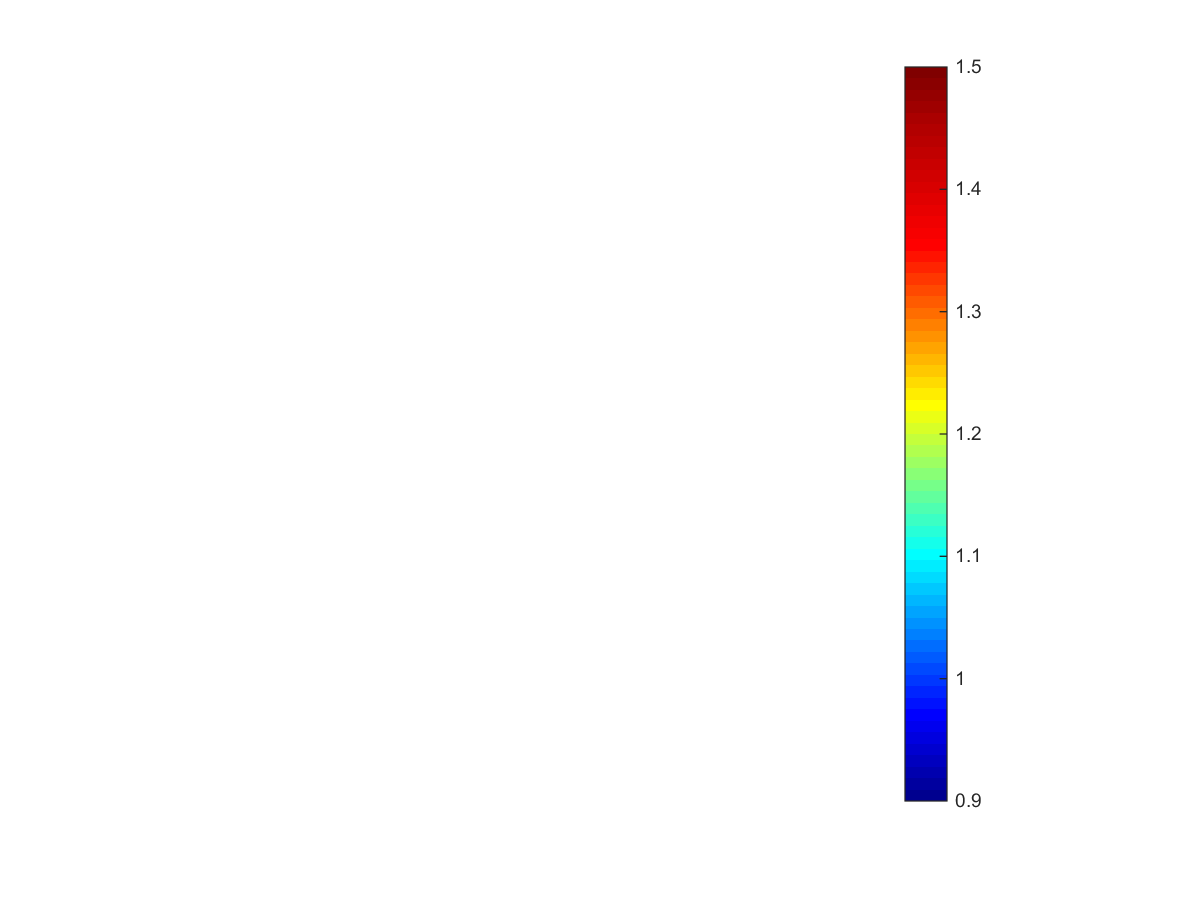}}

\put(250,100){\includegraphics[width=125 pt]{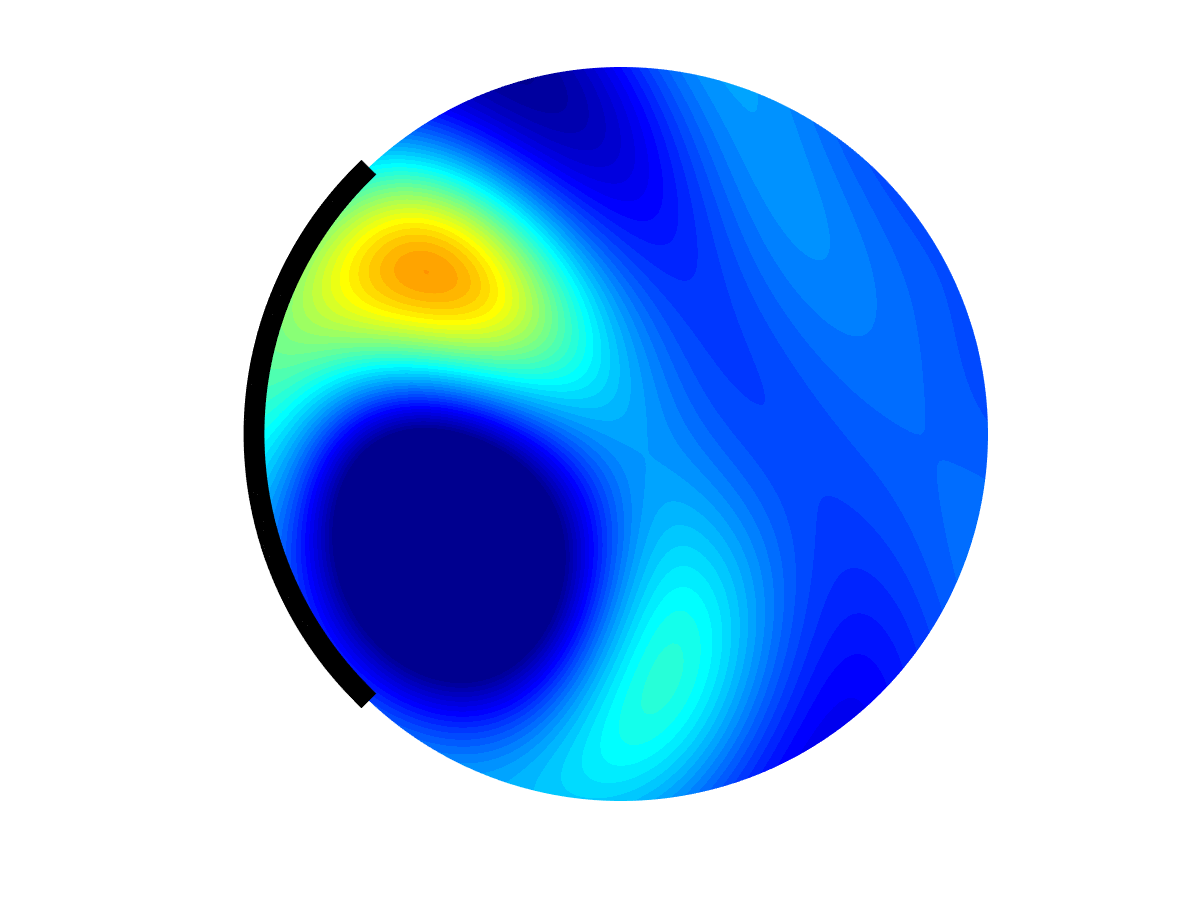}}
\put(250,0){\includegraphics[width=125 pt]{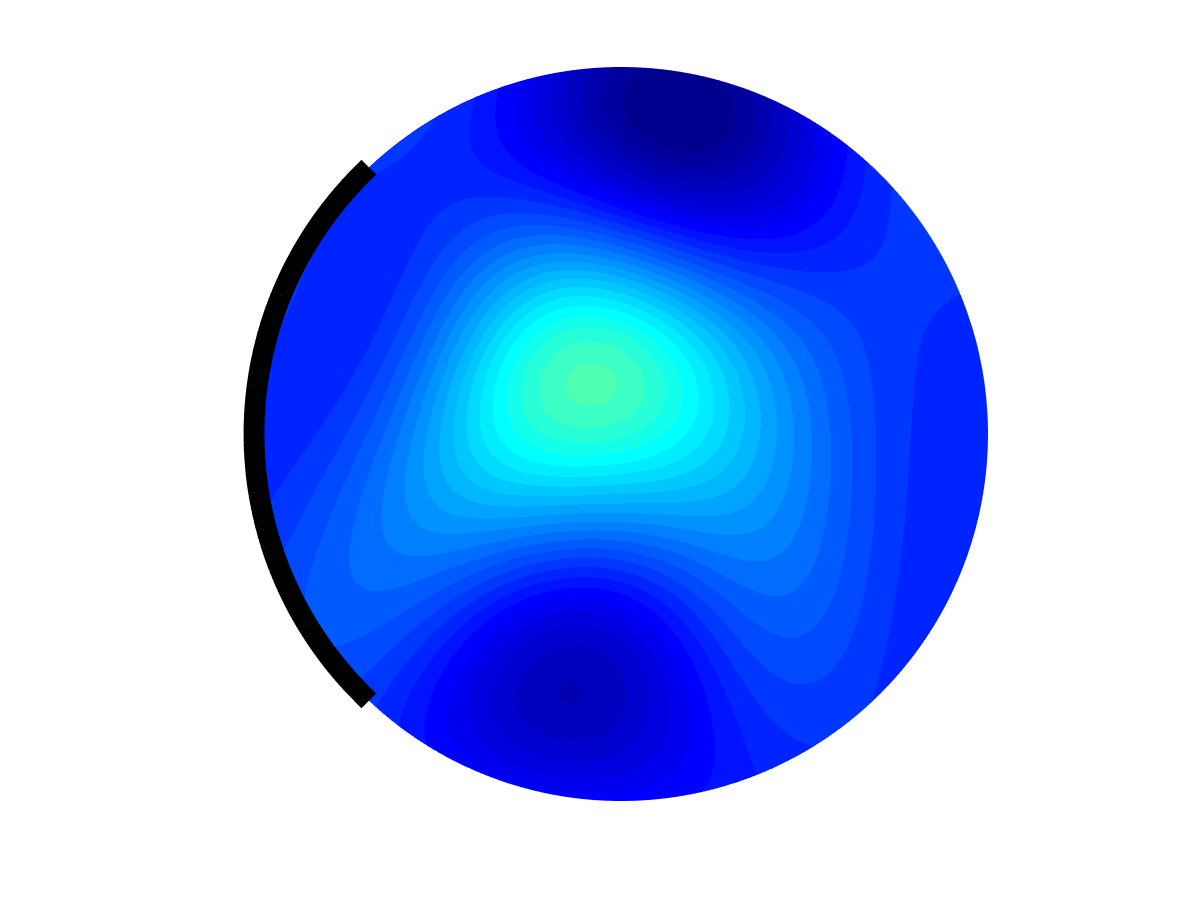}}

\put(145,100){\includegraphics[width=125 pt]{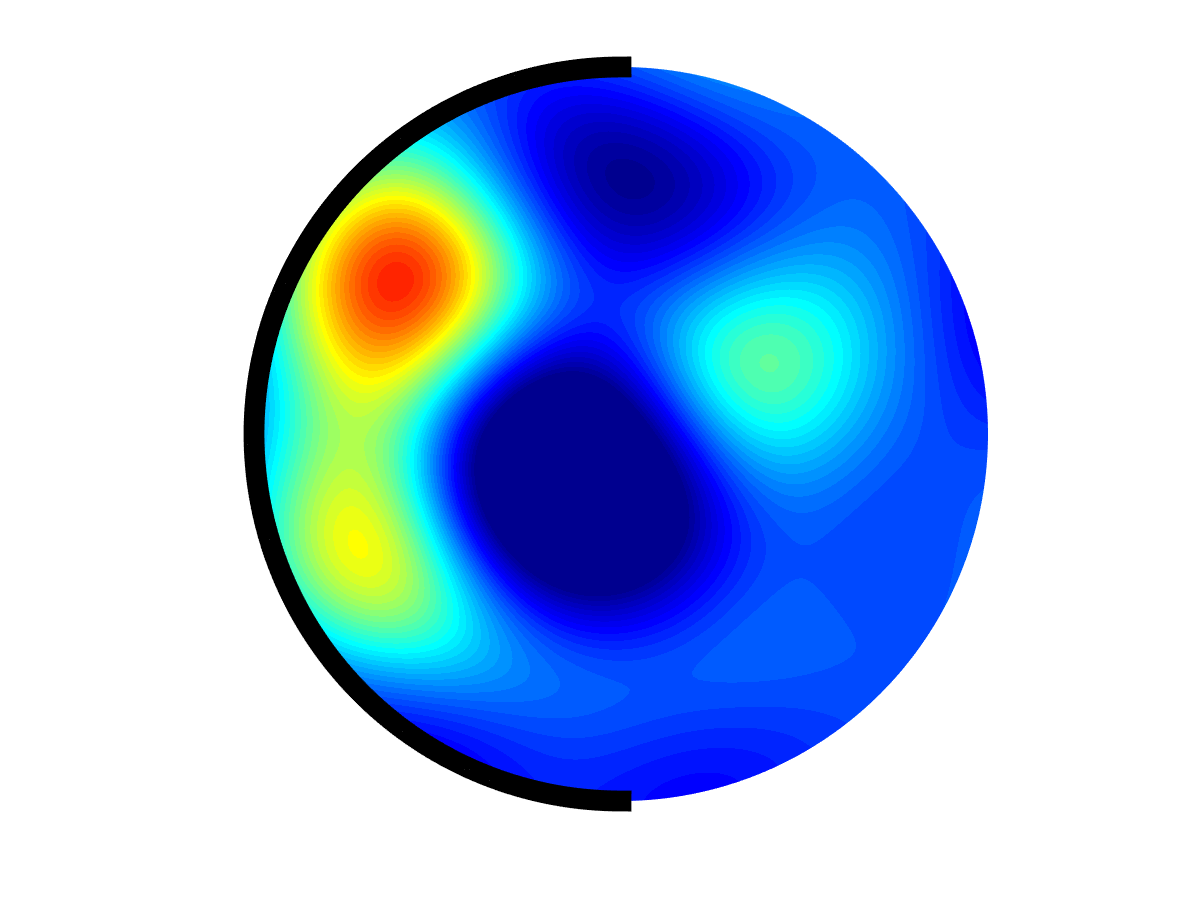}}
\put(145,0){\includegraphics[width=125 pt]{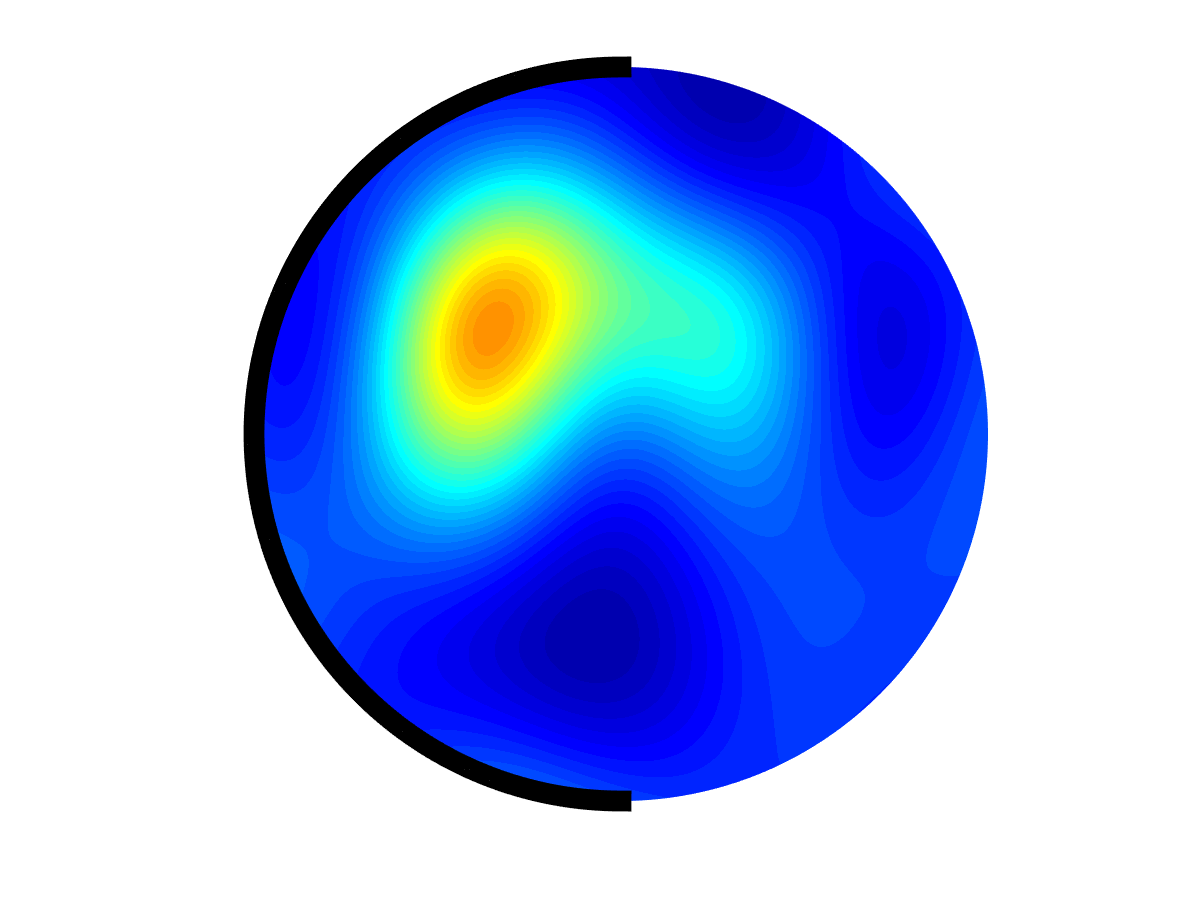}}

\put(40,100){\includegraphics[width=125 pt]{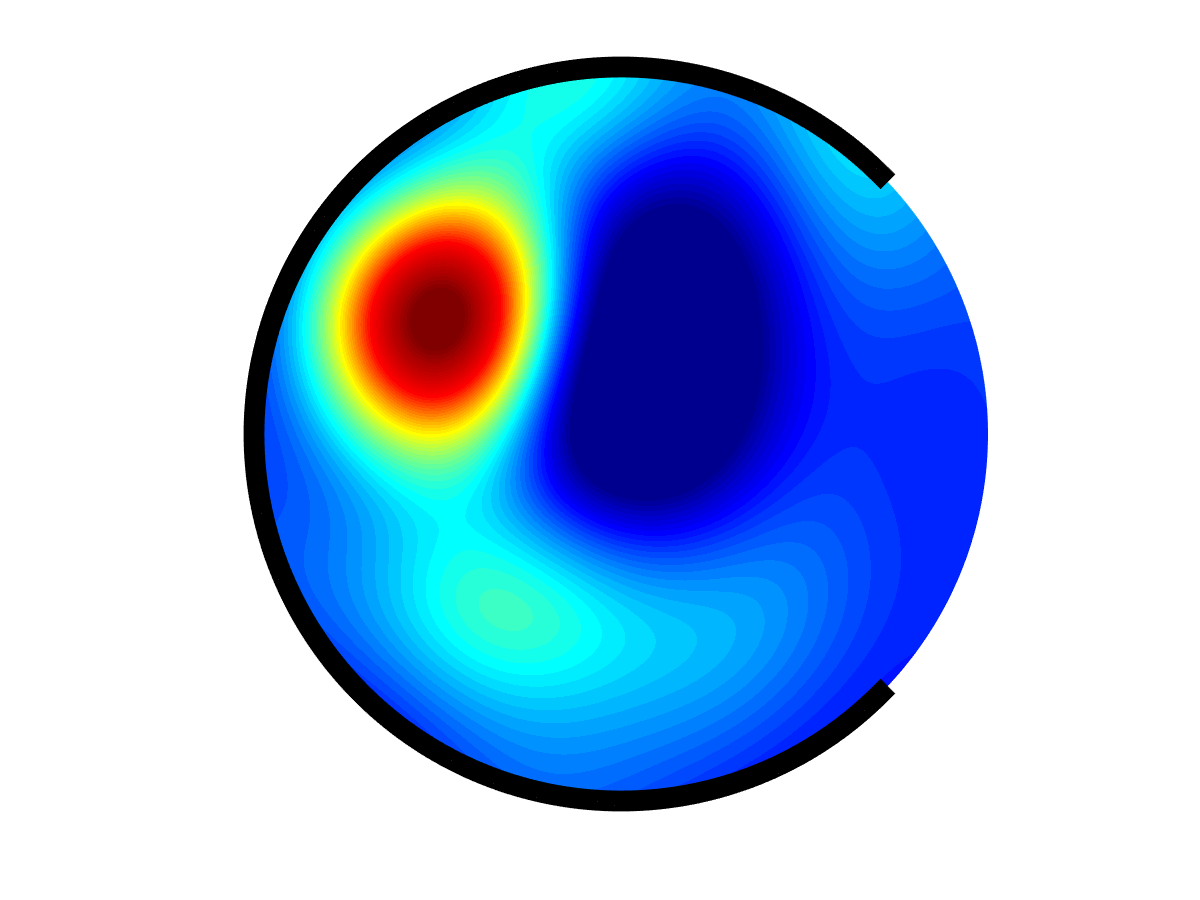}}
\put(40,0){\includegraphics[width=125 pt]{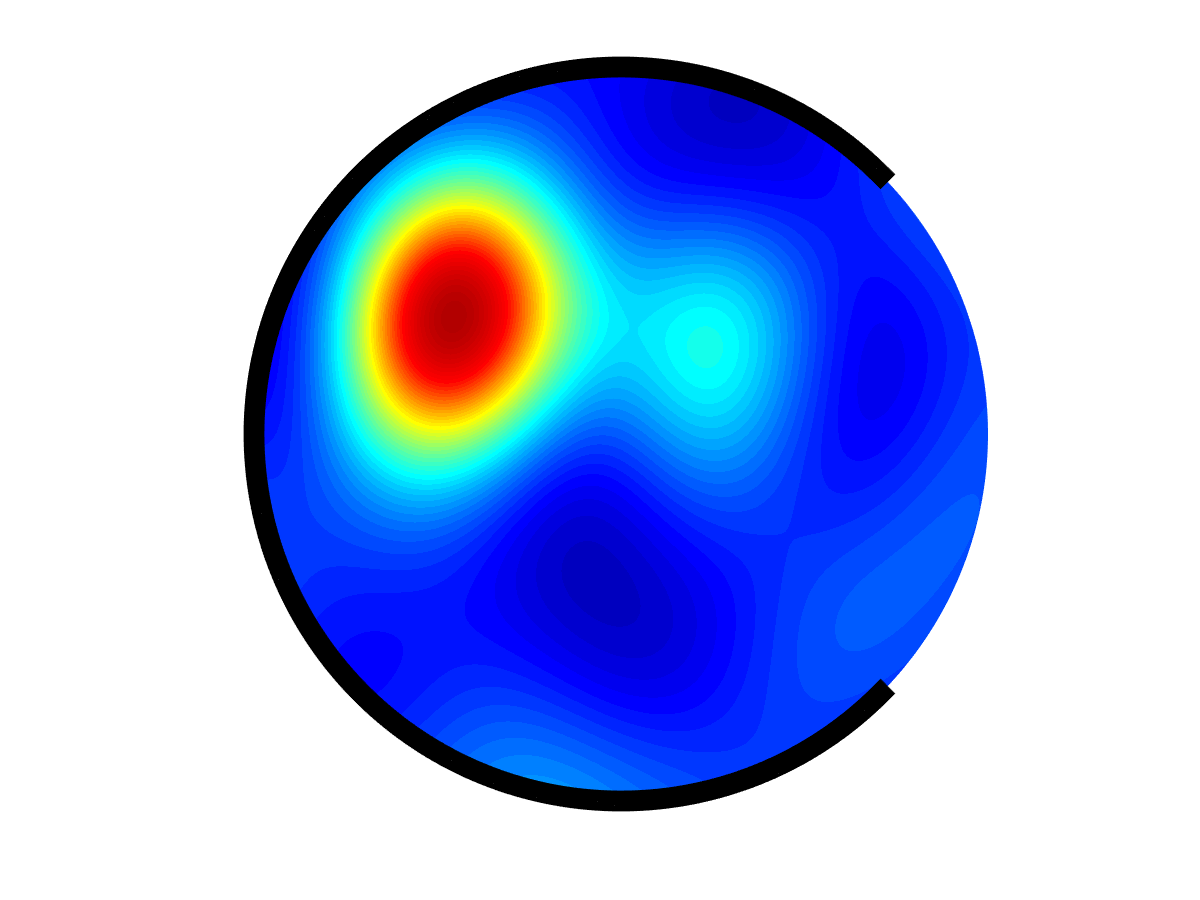}}

\put(-75,100){\includegraphics[width=125 pt]{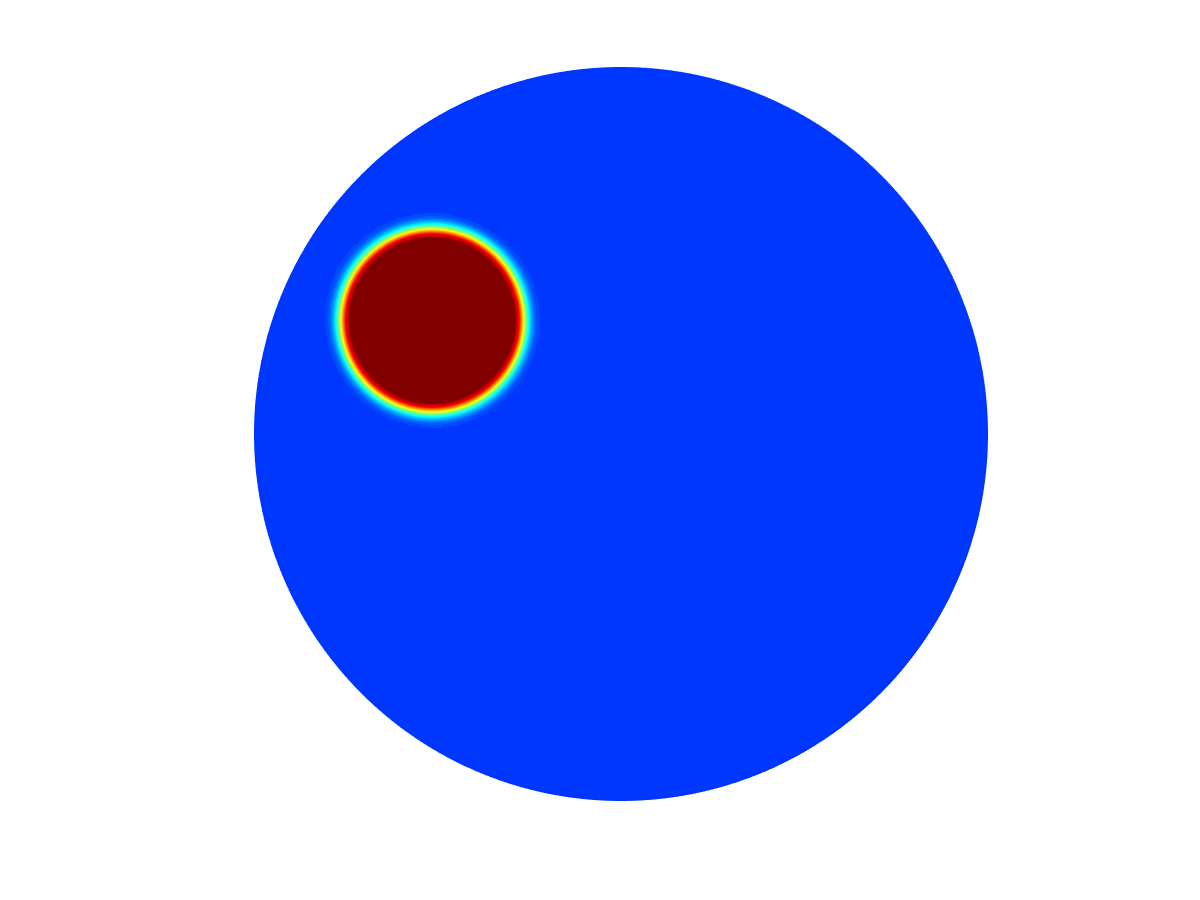}}
\put(-75,0){\includegraphics[width=125 pt]{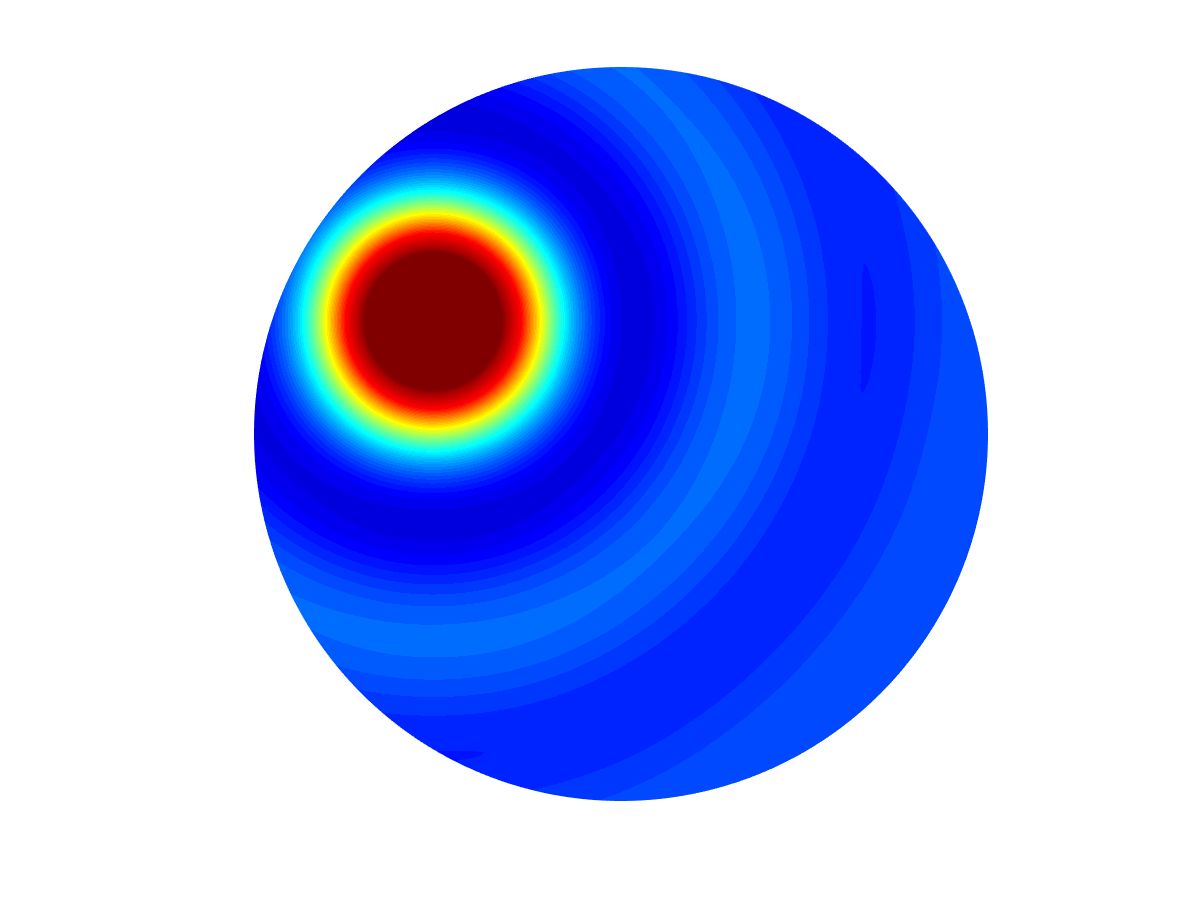}}

\put(-35,195){Phantom}
\put(-43,95){full-boundary}
\put(45,133){\rotatebox{90}{cut-off}}
\put(45,30){\rotatebox{90}{scaling}}
\put(95,195){75\%}
\put(200,195){50\%}
\put(305,195){25\%}

\end{picture} }
\caption{\label{fig:circleRecons}Reconstructions of a phantom with circular inclusion (conductivity value 2) with decreasing percentage of the boundary available. The measurement domain is centred on the left as indicated for each reconstruction by the black line. The first row shows reconstructions for the cut-off basis functions and the lower row for scaling basis functions.}
\end{figure}

The second test case is the medical motivated Heart-and-Lungs phantom on the unit disk. The conductivity is piecewise constant and hence represents more realistic measurements. The conductivity values are 0.5 for the lungs filled with air and 2 for the heart due to the blood. The measurement boundary $\Gamma$ is again centred on the left. Reconstructions can be seen in Figure \ref{fig:HnL_unitCirc} and the corresponding parameter choices for the scattering transform in Table \ref{tab:scatVals}.

\begin{figure}[ht!]
\centering
\ {
\begin{picture}(400,200)

\put(130,-20){\includegraphics[width=300 pt]{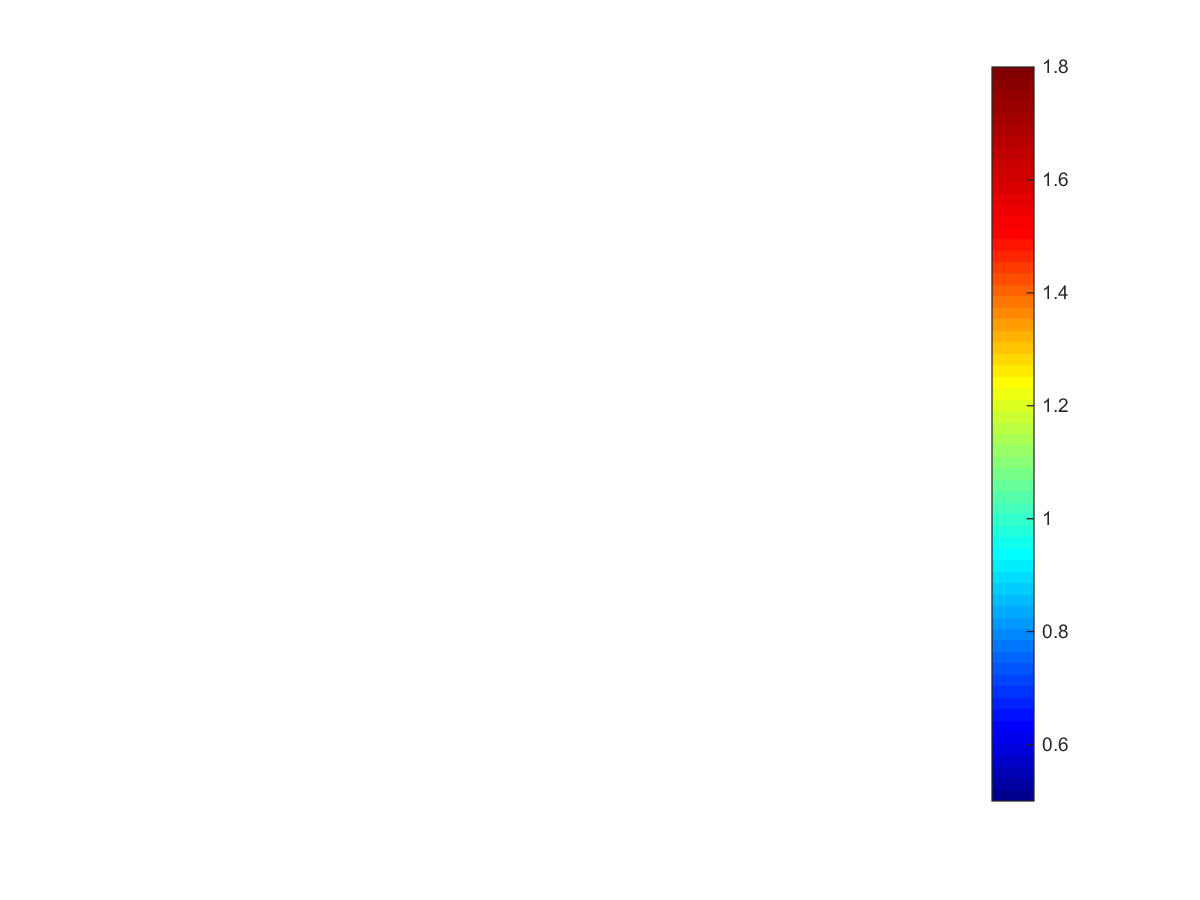}}

\put(250,100){\includegraphics[width=125 pt]{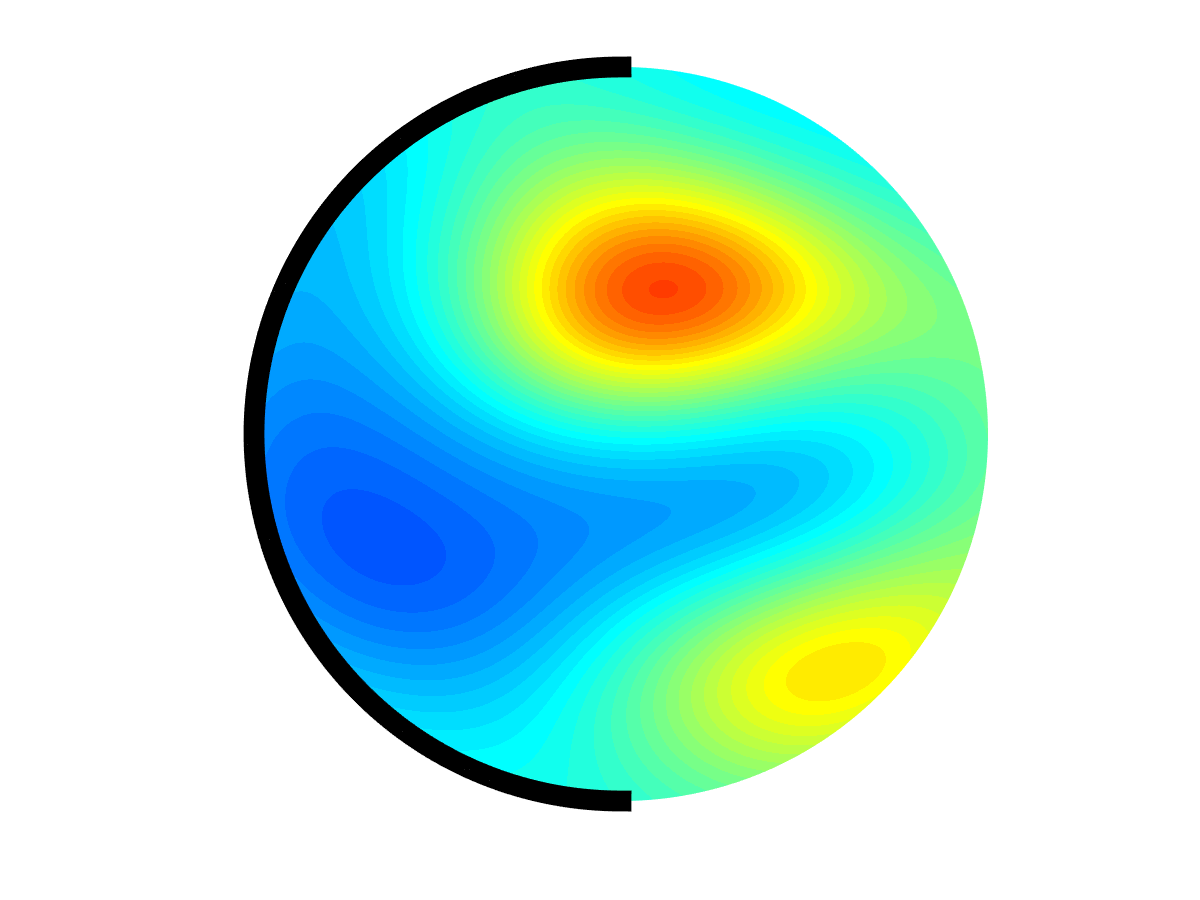}}
\put(250,0){\includegraphics[width=125 pt]{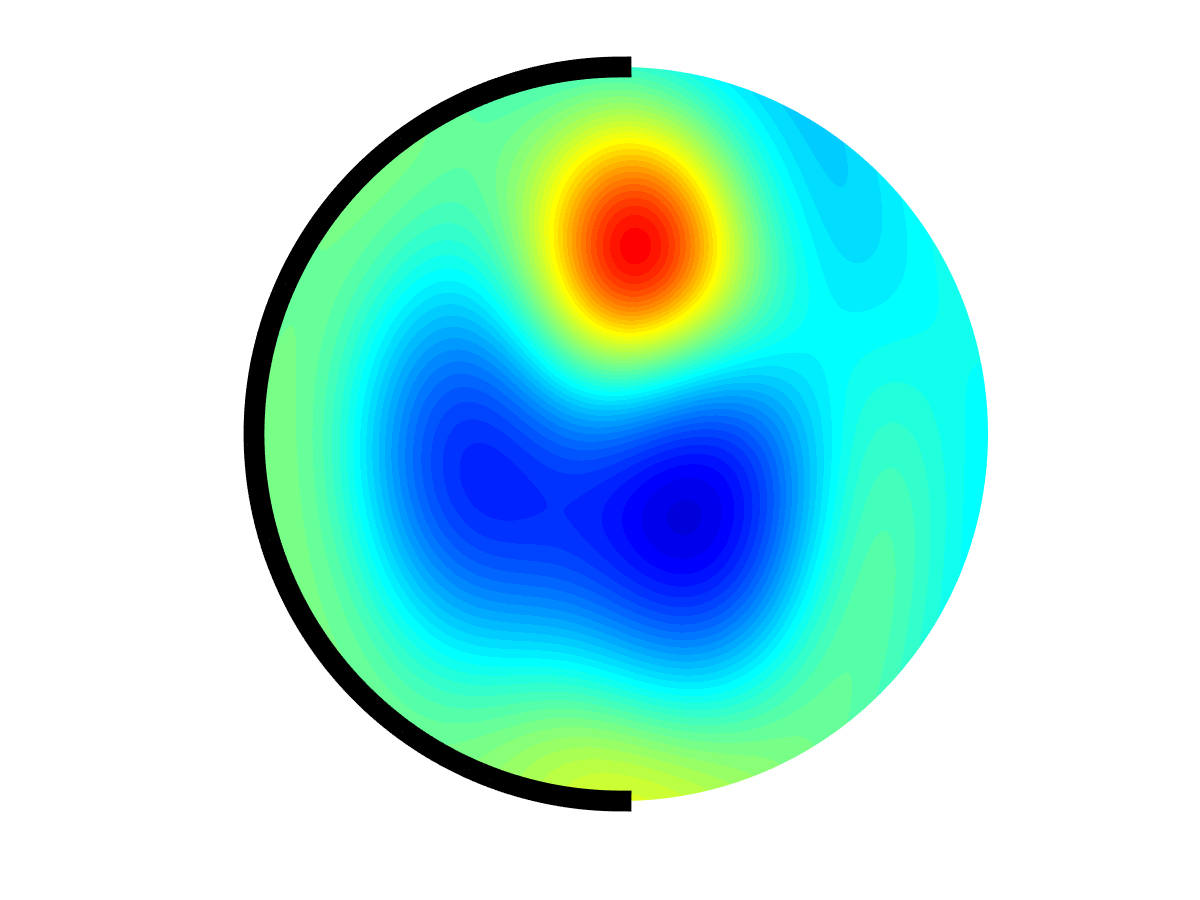}}

\put(145,100){\includegraphics[width=125 pt]{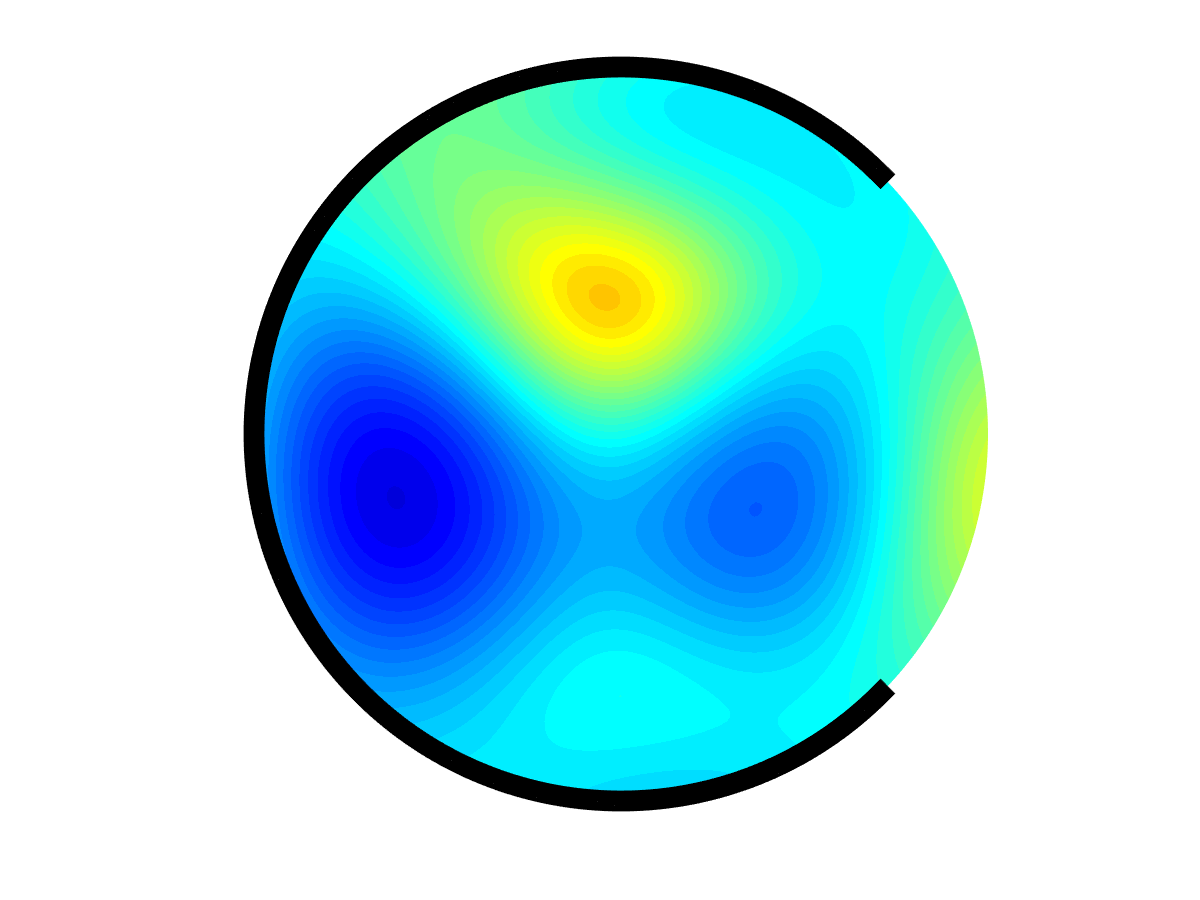}}
\put(145,0){\includegraphics[width=125 pt]{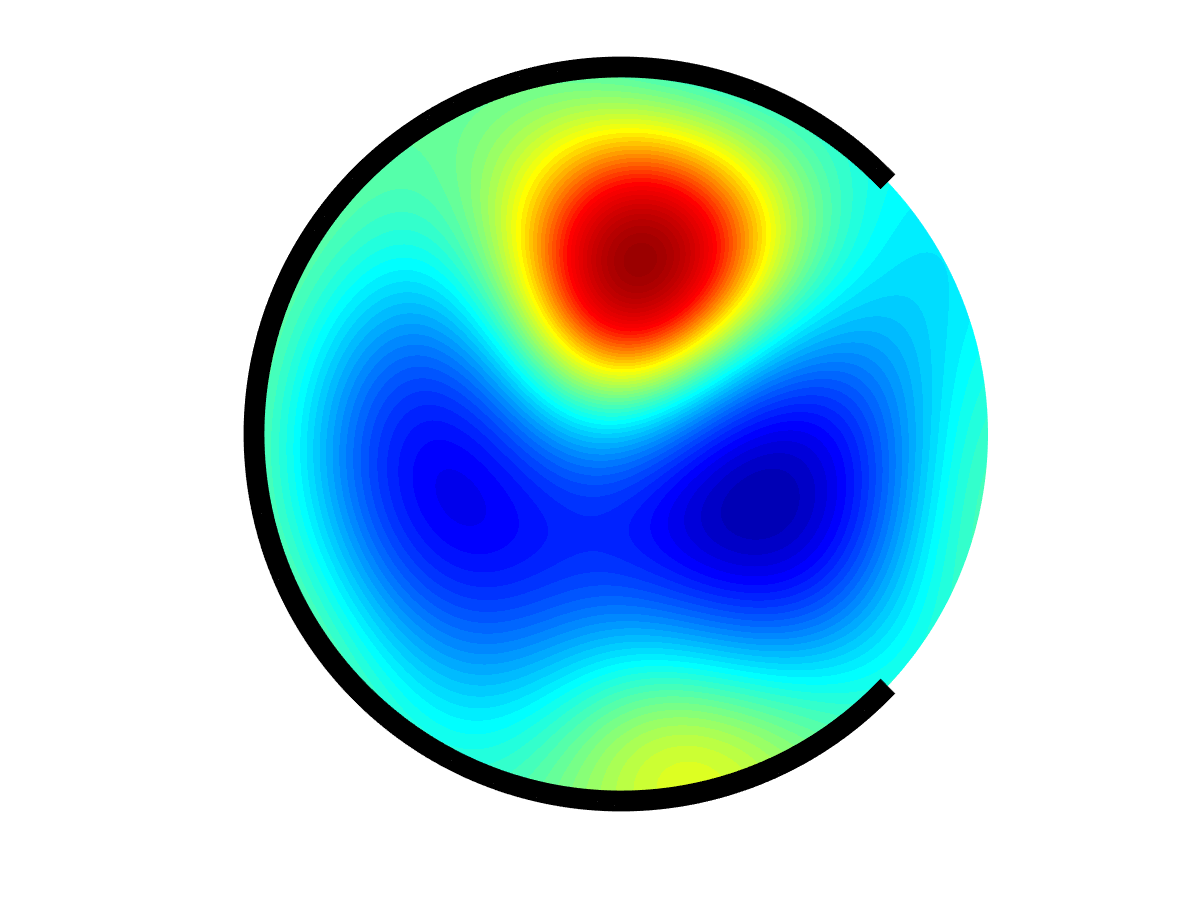}}

\put(40,100){\includegraphics[width=125 pt]{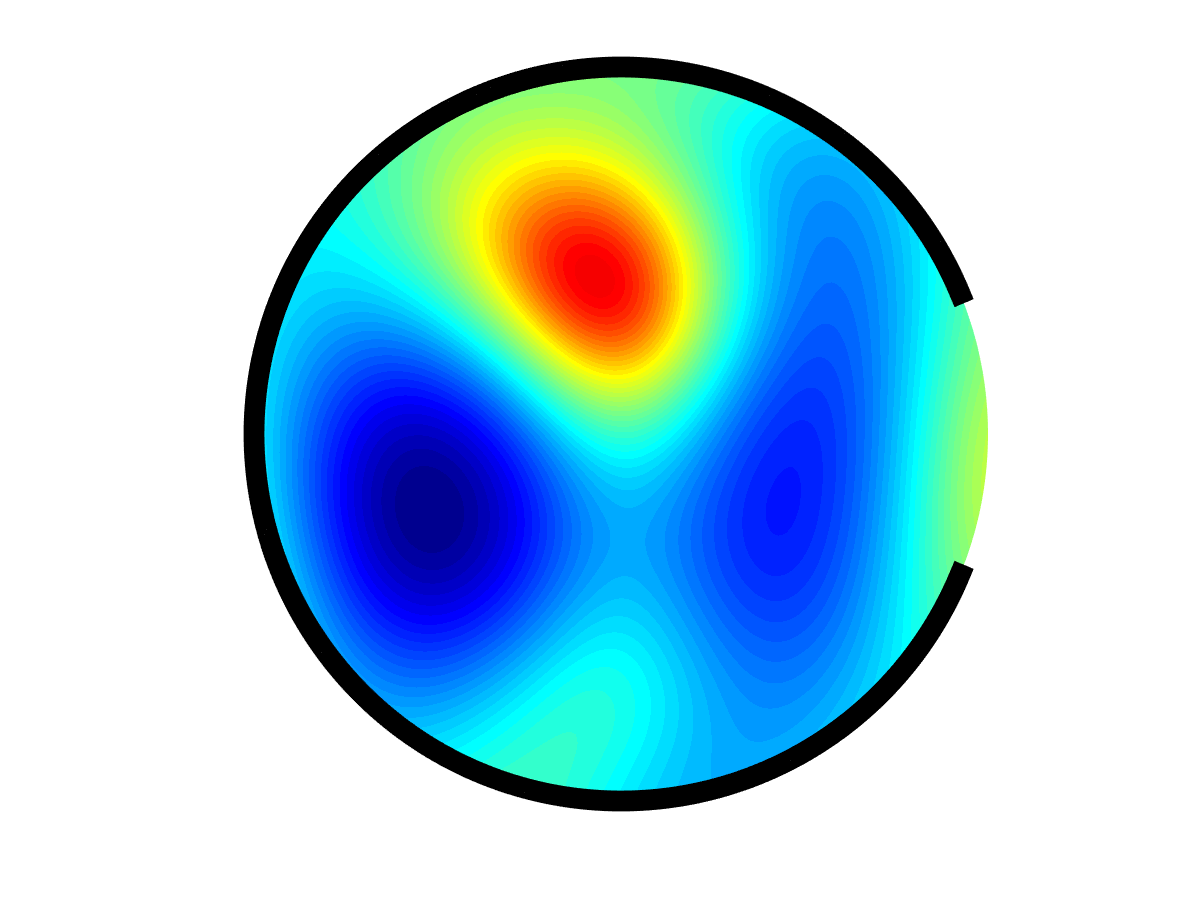}}
\put(40,0){\includegraphics[width=125 pt]{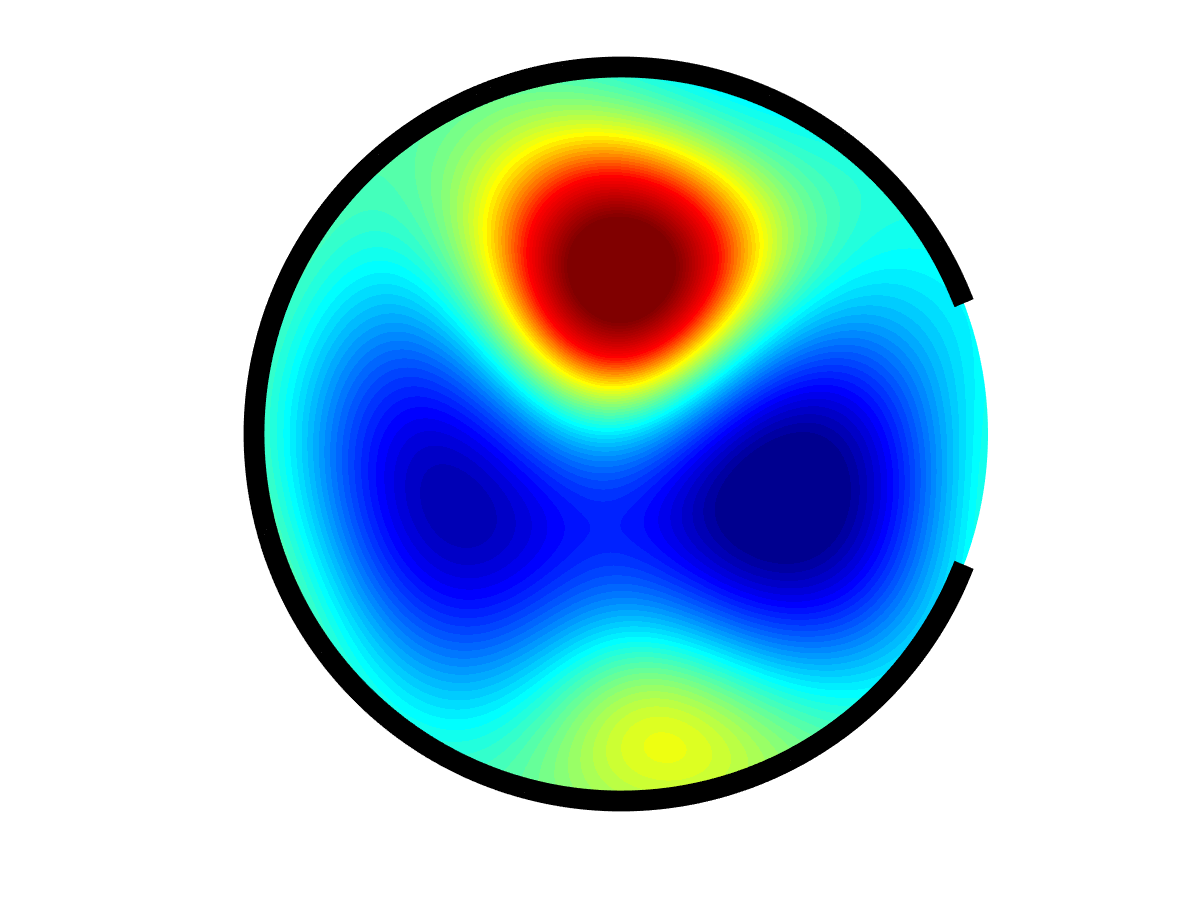}}

\put(-75,100){\includegraphics[width=125 pt]{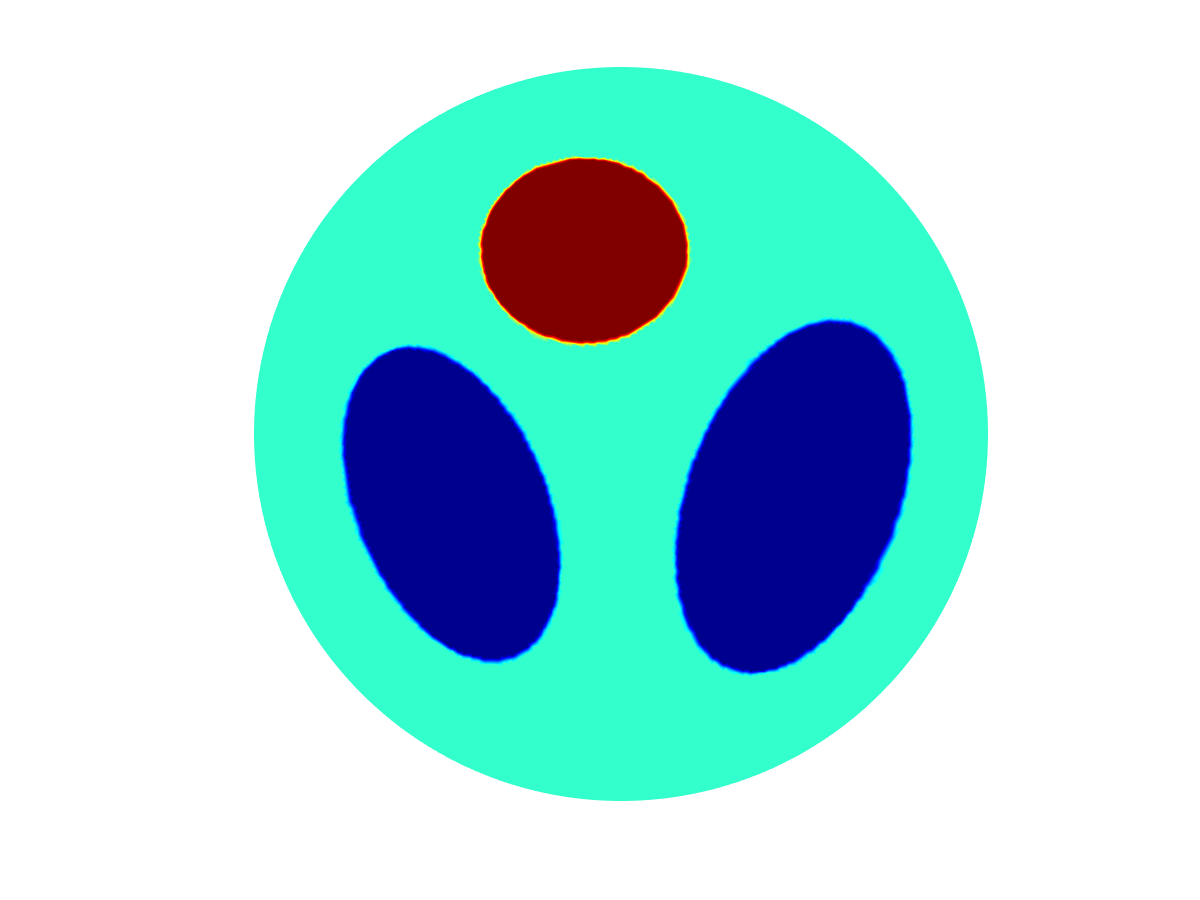}}
\put(-75,0){\includegraphics[width=125 pt]{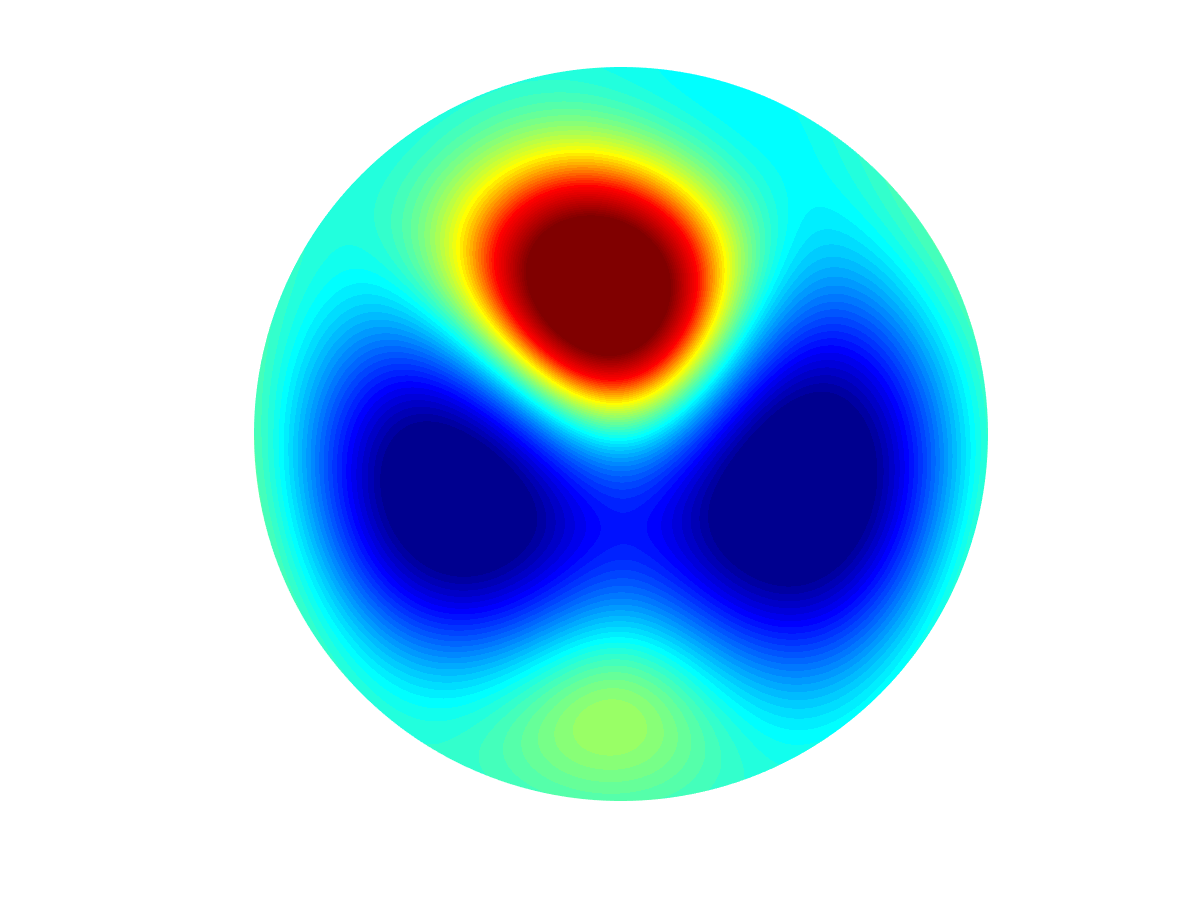}}

\put(-35,195){Phantom}
\put(-43,95){full-boundary}
\put(45,133){\rotatebox{90}{cut-off}}
\put(45,30){\rotatebox{90}{scaling}}
\put(93,195){87.5\%}
\put(200,195){75\%}
\put(305,195){50\%}

\end{picture} }
\caption{\label{fig:HnL_unitCirc}Reconstructions of the Heart-and-Lungs phantom with decreasing percentage of the boundary available. The measurement domain is centred on the left as indicated by the black lines. The conductivity values are 0.5 for the lungs and 2 for the heart.}
\end{figure}
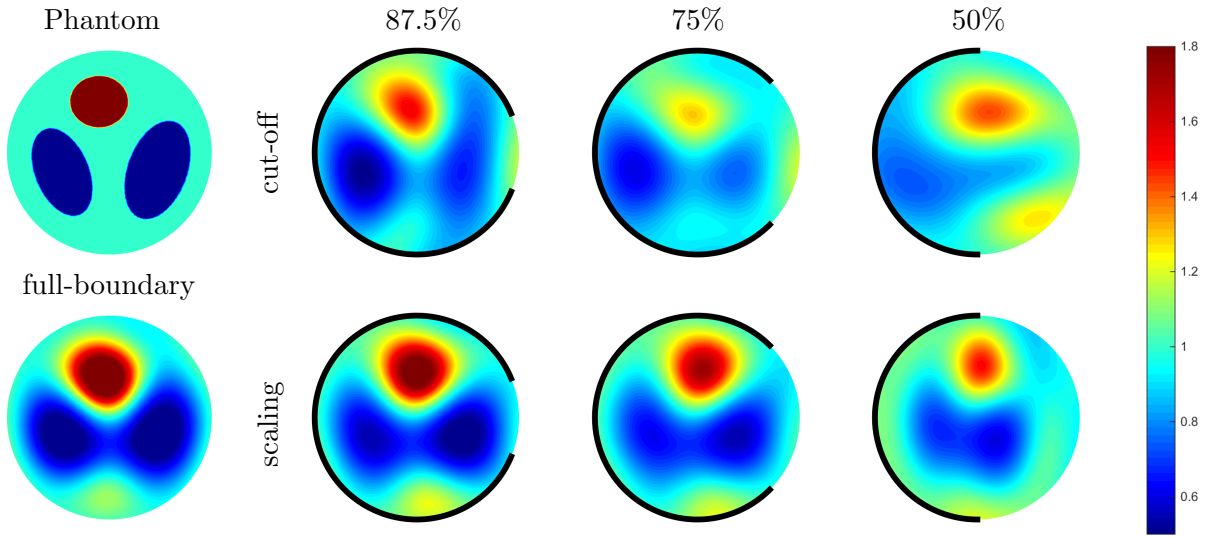 

\begin{table}[h!]
  \centering

  \begin{tabular}{r|c|c|c|c|c}
     & 87.5\% & 75\%& 50\% & 25\%  \\ \hline
Circle: cut-off & - & $R=5$ & $R=4.5$ & $R=4.5$   \\ 
				& - & $c=5$ & $c=6$ & $c=8$   \\				\hline 	
	    scaling & - & $R=5$ & $R=5$ & $R=4$ \\
	    		& - & $c=4$ & $c=4$ & $c=4$  \\ \hline \hline
	    
   HnL: cut-off &$R=4$ & $R=4$ & $R=4$ & - \\ 
				&$c=10$ & $c=8$ & $c=8$ & -  \\ \hline
		scaling &$R=4$ & $R=4$ & $R=5$ & - \\ 
				&$c=10$ & $c=10$ & $c=10$ & -  \\
  \end{tabular}

    \caption{    \label{tab:scatVals}Values chosen for computing $\widetilde{\T}^{\mathrm{ND}}_{R,c}(k)$ for each case presented. The values are chosen such that features in the image are close to the original and contrast is maximized.}

\end{table}

At last we verify that Theorem \ref{theo:ReconError} holds numerically. As stated in the theorem, we compute the reconstruction error for a fixed cut-off radius for all reconstructions. The resulting errors are plotted in Figure \ref{fig:conv_recon} for both of the phantoms in the unit disk. Theorem \ref{theo:ReconError} only covers the cut-off case, since we do not have an operator norm for the scaling basis. As mentioned before, in case we only have a finite matrix approximation of the ND map available the result extends to the scaling case as well. The predicted linear convergence rate is achieved for both basis functions with extrapolated measurements. It is interesting to note that the graphs have similarities in their shape to the error in ND matrices. This is due to the continuity of solutions of the D-bar equation, that means the behaviour of the data error directly translates to the reconstruction error.

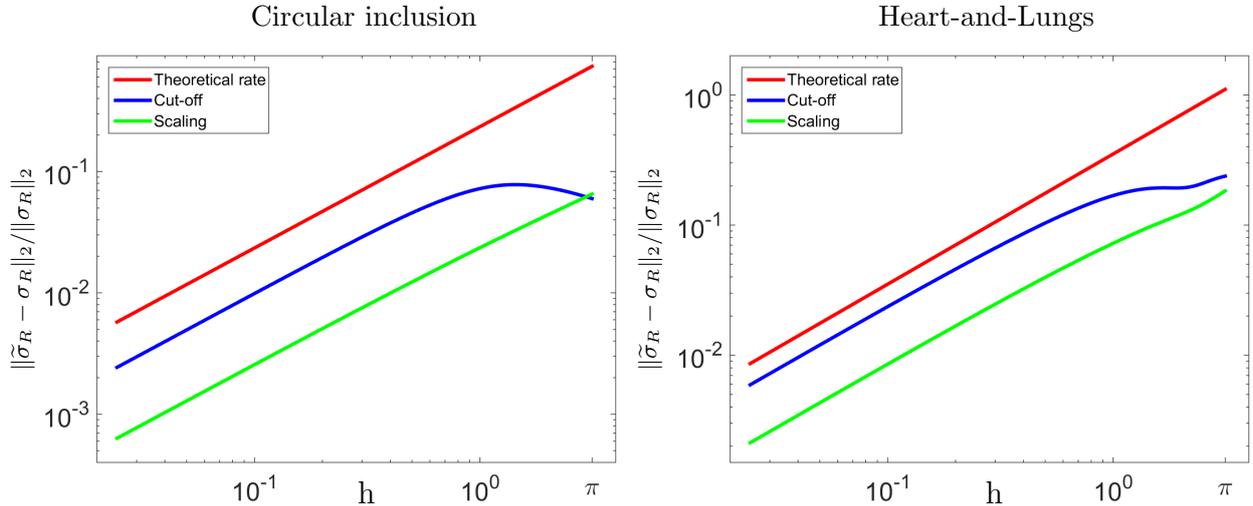
\begin{figure}[ht!]
\centering
\begin{picture}(350,180)
\put(182,-10){\includegraphics[width=250 pt]{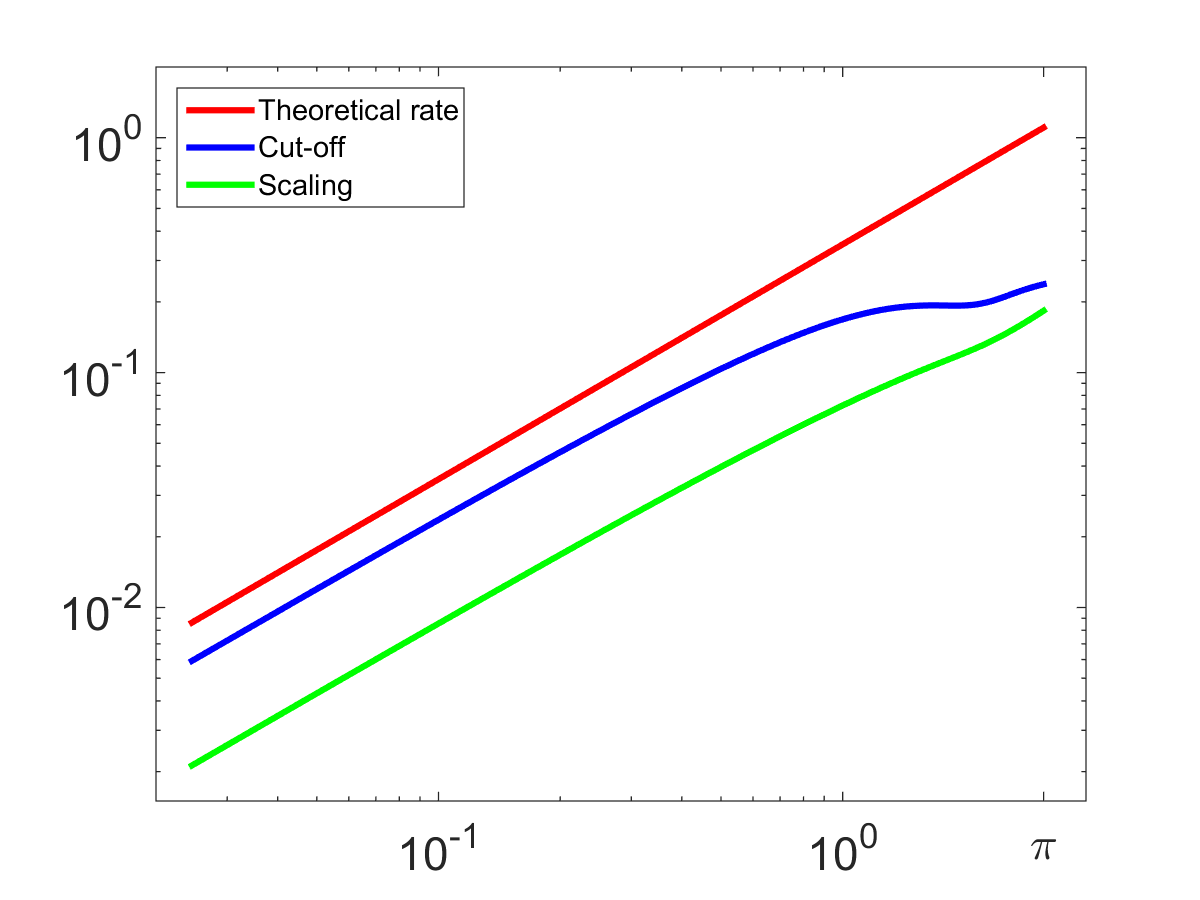}}
\put(-55,-10){\includegraphics[width=250 pt]{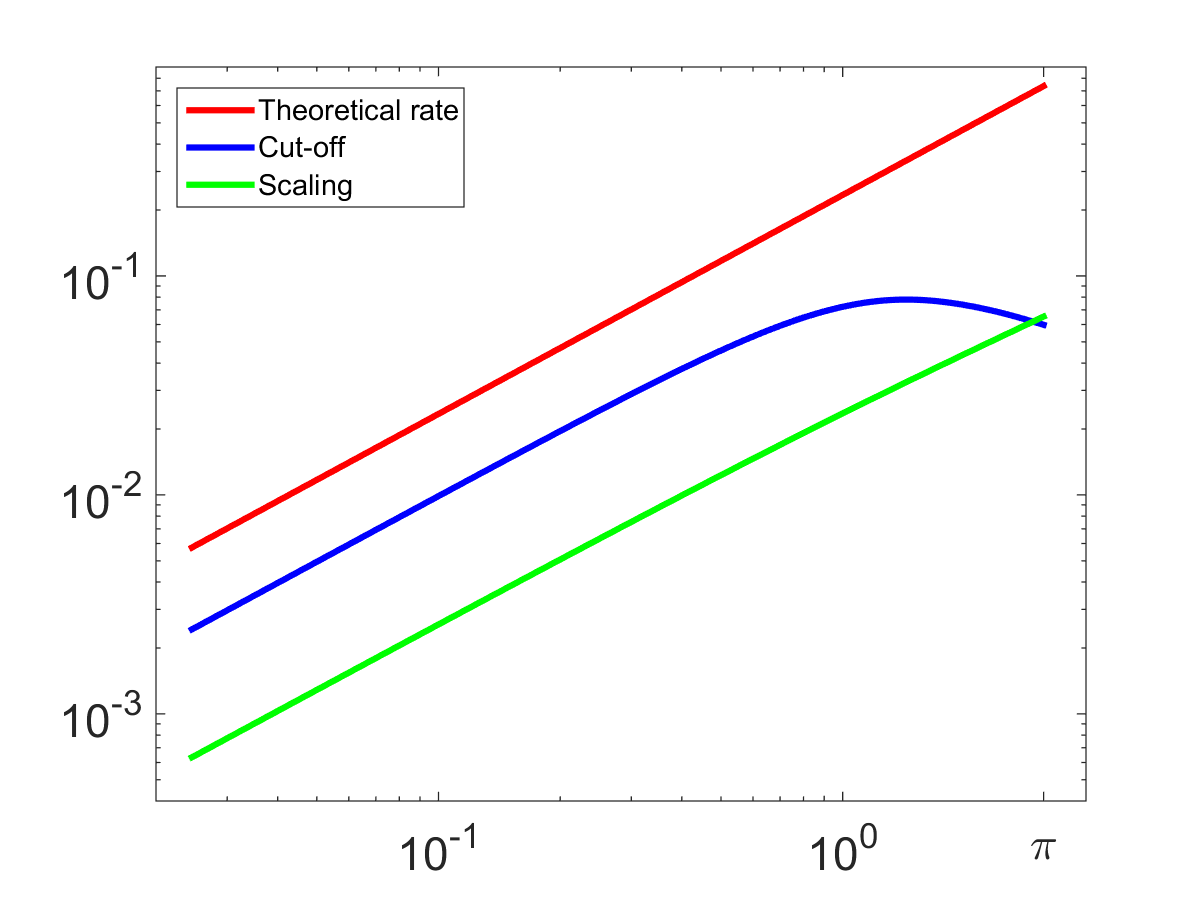}}
\put(75,-5){\large h}
\put(310,-5){\large h}
\put(-55,45){\rotatebox{90}{\footnotesize $\|\widetilde{\sigma}_{R}-\sigma_{R}\|_2/\|\sigma_{R}\|_2$ }}
\put(180,45){\rotatebox{90}{\footnotesize $\|\widetilde{\sigma}_{R}-\sigma_{R}\|_2/\|\sigma_{R}\|_2$ }}
\put(35,175){Circular inclusion}
\put(270,175){Heart-and-Lungs}
\end{picture}
\caption{\label{fig:conv_recon}Relative error of reconstructions from extrapolated measurements with fixed cut-off radius $R=3$ and no threshold. On the left for the circular inclusion and on the right for the Heart-and-Lungs phantom.}
\end{figure} 

\begin{figure}[h!]
\centering
\begin{picture}(250,120)
\put(-25,-15){\includegraphics[width=150 pt]{images/recons/HnL_phantom_clean.png}}
\put(100,-15){\includegraphics[width=150 pt]{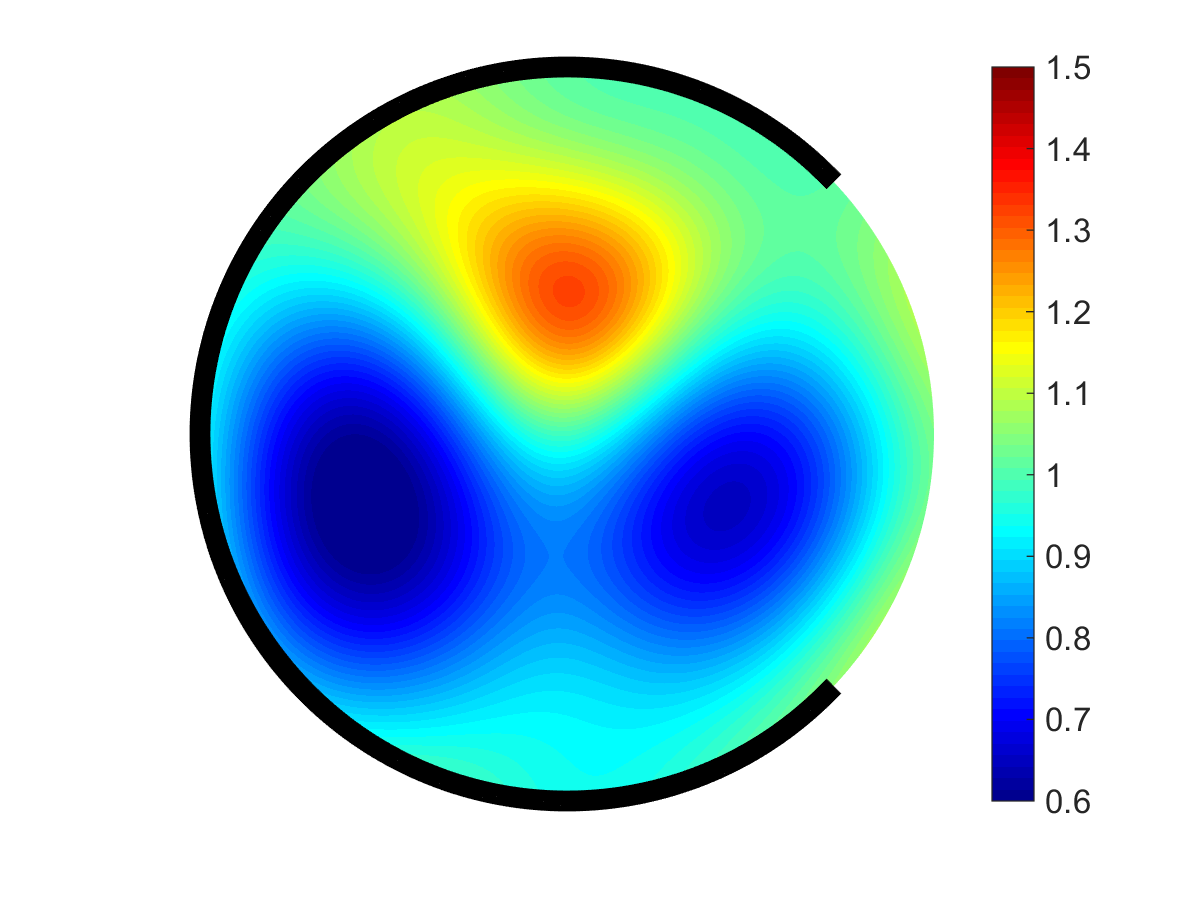}}
\put(25,100){\Large Phantom}
\put(125,108){\Large Combined data}
\put(128,95){\Large reconstruction}
\end{picture}
\caption{\label{fig:combinedData} Reconstruction of the Heart-and-Lungs phantom by combining the data acquired on 75\% of the boundary with the scaling and cut-off basis. Some contrast is lost, but features are nicely preserved and separated.}
\end{figure}

The reconstructions in Figure \ref{fig:circleRecons} and Figure \ref{fig:HnL_unitCirc} have very similar characteristics. In general the contrast of the reconstruction decreases with the measurement domain available. For the cut-off basis we see that the reconstruction is centred close to the measurement boundary, whereas for the scaling basis the reconstruction tends to move into the middle of the domain. This suggests a numerical experiment of combining the data as $\widetilde{\mathbf{R}}_{\sigma,1}=1/2(\widetilde{\mathbf{R}}^c_{\sigma,1}+\widetilde{\mathbf{R}}^s_{\sigma,1})$ and use this for computing the reconstruction, as illustrated in Figure \ref{fig:combinedData}.

\subsection{Reconstructions for a human chest phantom}
We consider a Heart-and-Lungs phantom defined on a chest-shaped mesh (see Figure \ref{fig:HnL_FEM_recons}). 
The conductivity values in this case are 0.5 for the lungs filled with air and 3 for the heart. 

\begin{figure}[h!]
\centering
\begin{picture}(300,350)
\put(38,-10){\includegraphics[width=350 pt]{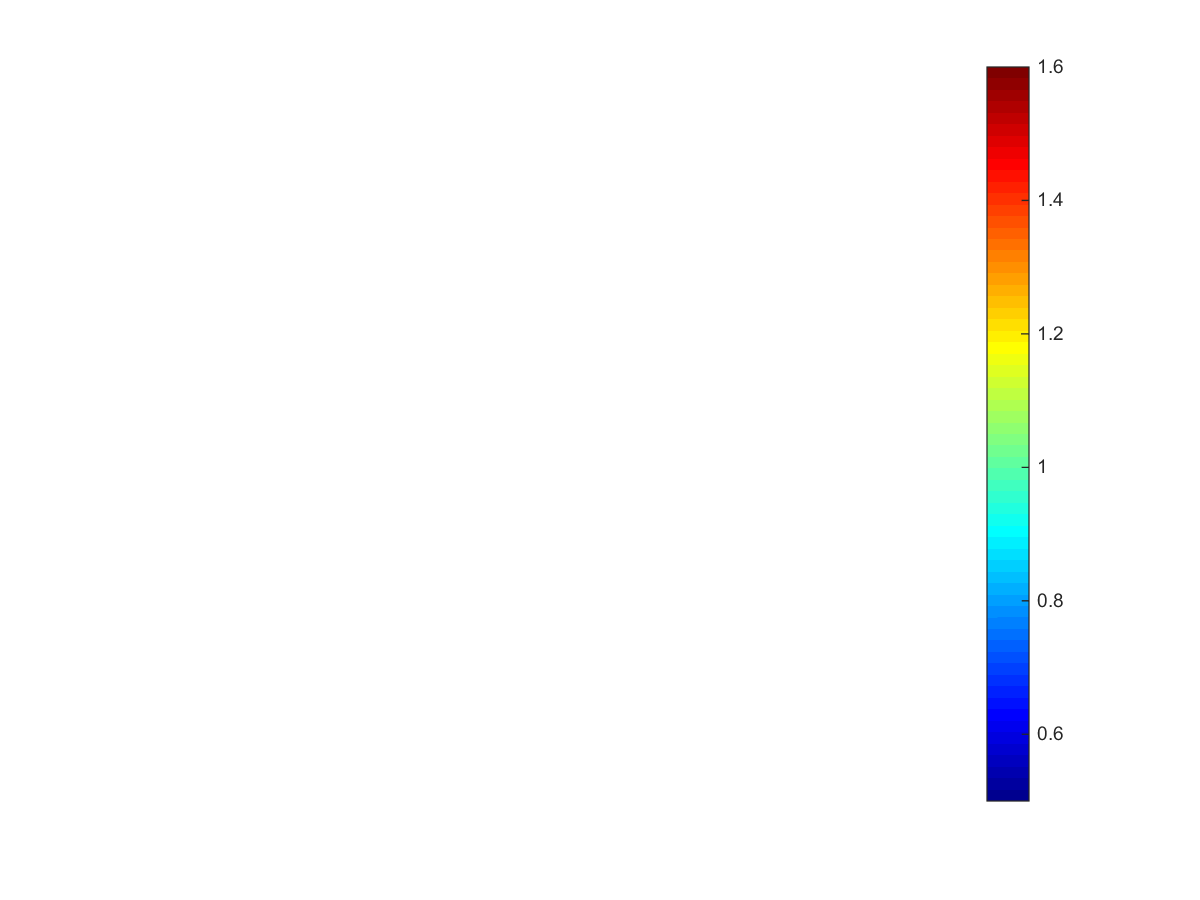}}
\put(-10,220){\includegraphics[width=175 pt]{images/recons/HnL_FEM_phantom.png}}
\put(150,220){\includegraphics[width=175 pt]{images/recons/HnL_FEM_full.png}}
\put(-10,110){\includegraphics[width=175 pt]{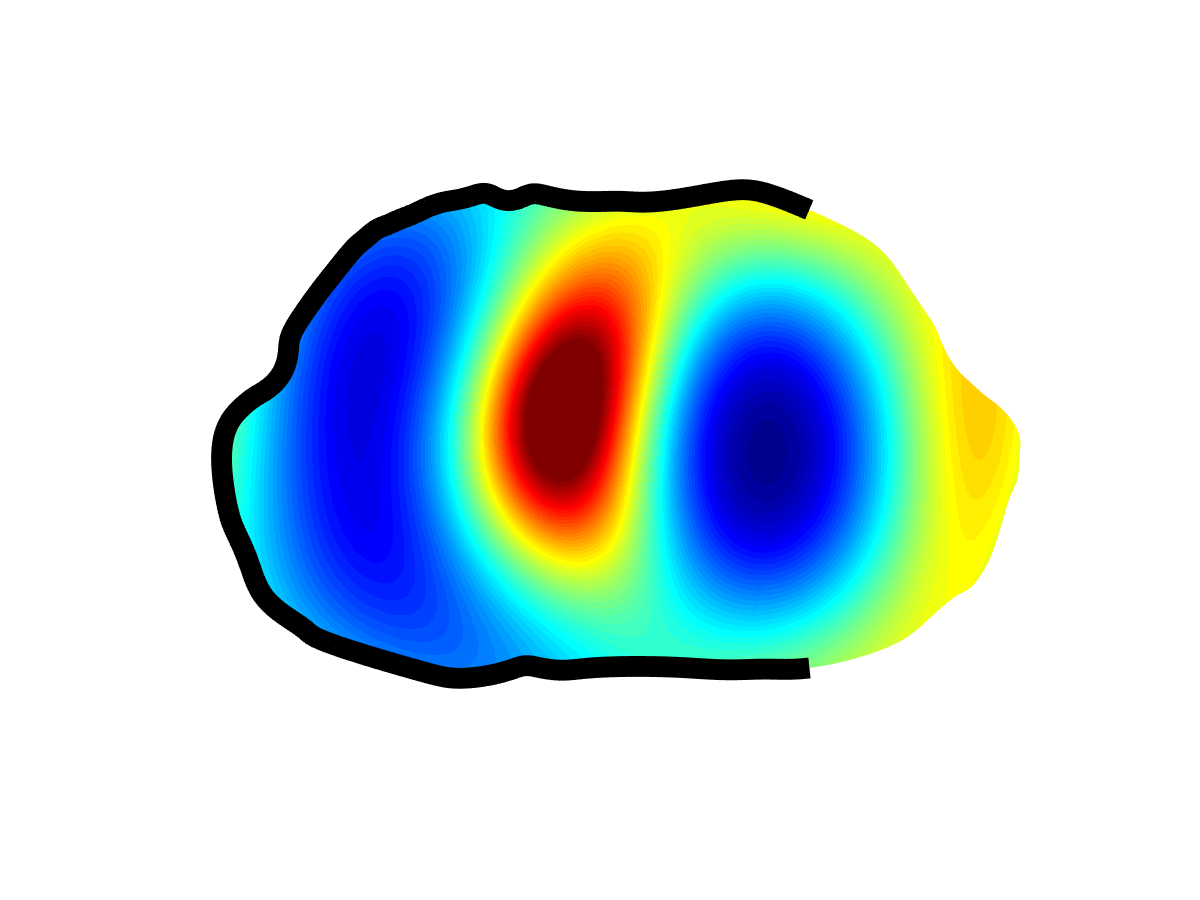}}
\put(150,110){\includegraphics[width=175 pt]{images/recons/HnL_FEM_cut_back_75_gamma.png}}
\put(-10,0){\includegraphics[width=175 pt]{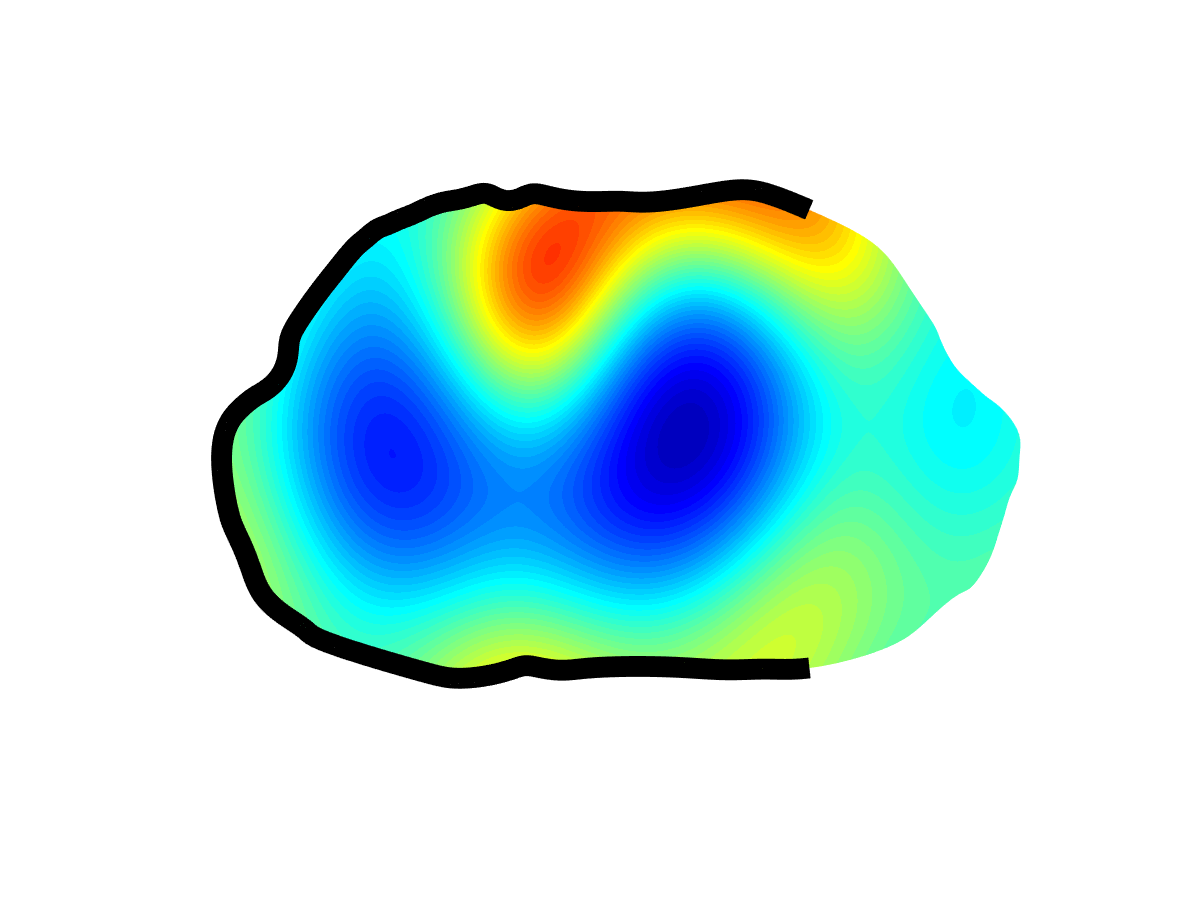}}
\put(150,0){\includegraphics[width=175 pt]{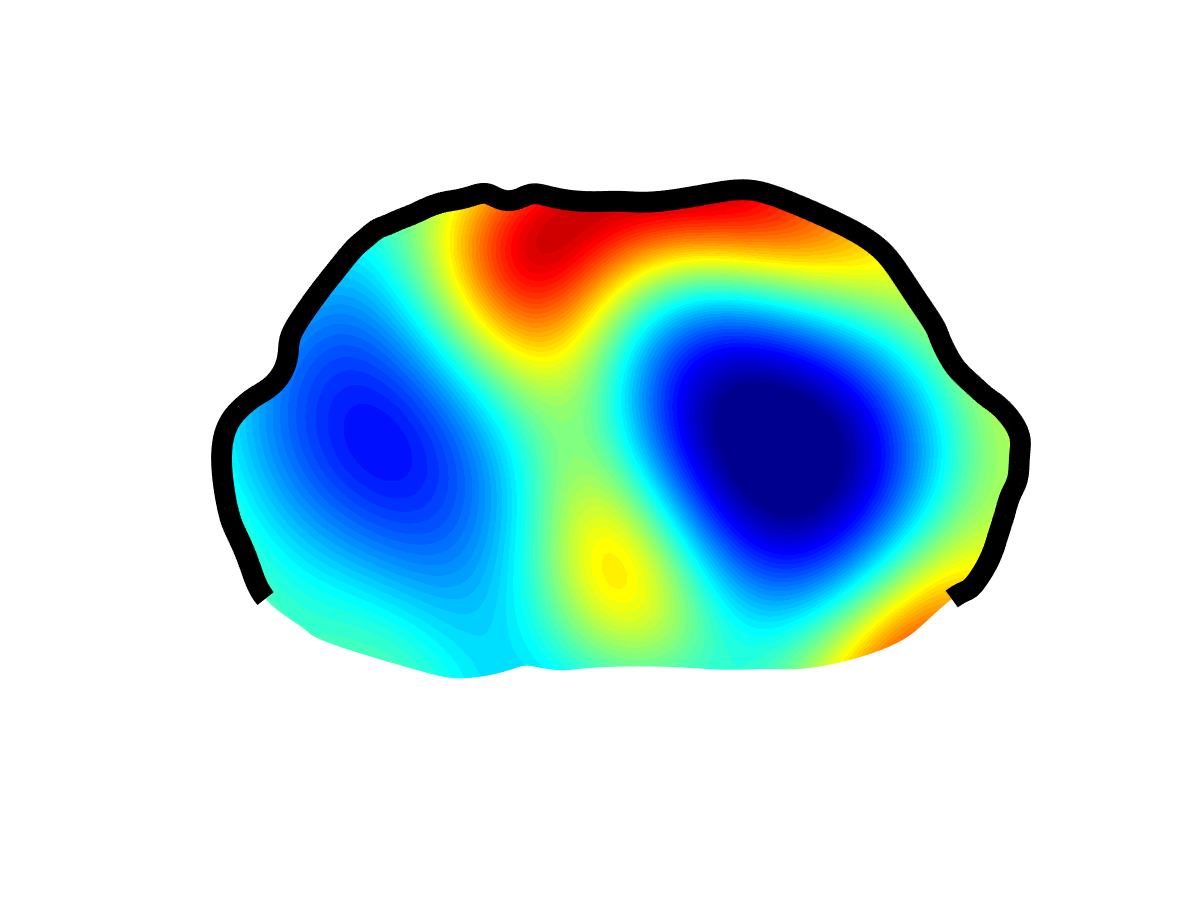}}
\put(45,220){\Large{Lateral 66\%}}
\put(205,220){\Large{Ventral 75\%}}
\put(50,330){\Large{Phantom}}
\put(198,330){\Large{Full-boundary}}
\put(-10,160){\rotatebox{90}{\Large{Cut-off}}}
\put(-10,45){\rotatebox{90}{\Large{Scaling}}}

\end{picture}
\caption{\label{fig:HnL_FEM_recons} Reconstructions of a human chest phantom with two measurement positions. (Left) A sideways lying patient with 66\% of the boundary available for measurements and (Right) a patient lying on the back with 75\% of the boundary available. The measurement domain is indicated by the black line. Results for both choices of the partial-boundary map are illustrated.}
\end{figure}

For the simulation of the ND map we first need to define an orthonormal basis on the boundary: 
Let $r : [0,2\pi] \to [0,2\pi], s\mapsto r(s)$ be the arclength parametrization of $\partial \Omega$ (where $\Omega$ is our chest region), i.e. if $ \gamma (s) = e^{i r(s)}$, then $|\gamma'(s)| =1$. The functions $\varphi_n = (2\pi)^{-1/2} e^{i n r(s)}$ then form an orthonormal basis of $L^2(\partial \Omega)$ and the ND map can be simulated as before. Thus, we compute the partial ND matrix as 
\[
(\widetilde{\mathbf{R}}_\sigma)_{n,\ell}=(\widetilde{\ND}_\sigma \varphi_n,\varphi_\ell)=(\ND_\sigma\parphi_n,\varphi_\ell)=\frac{1}{\sqrt{2\pi}}\int_{\bndry} u_n|_{\bndry}(r(s))e^{-i\ell r(s)}ds.
\]
This also applies when computing the scattering transform over $\partial\Omega$. The D-bar equation is then solved for each point in the chest mesh.

The reconstructions are all displayed in Figure \ref{fig:HnL_FEM_recons}. For the measurement domain we have chosen two positions of the patient, lying sideways and on the back. The position of the measurement domain was chosen such that it reasonably represents the accessible part of the patient. In all four cases the reconstruction quality seems to be sufficient for detecting the collapse of a lung, for example. The heart is pushed to the boundary in some reconstructions and seems to disappear. This could be overcome by reconstructing on a slightly larger mesh to catch the disappearing features.

\section{Conclusions}\label{sec:conclusion}
The inverse problem of electrical impedance tomography with the realistic assumption that only part of the boundary is accessible for measurements has been studied. In this setting it is important to consider the Neumann problem rather than the Dirichlet problem, since partially supported Dirichlet boundary conditions are nonphysical in many applications.
In this context we have introduced a partial Neumann-to-Dirichlet map: represented as a composition of the ND map and a partial-boundary map. The choice of the partial-boundary map is crucial for the error analysis and the reconstruction quality. For an electrode model involving contact impedances one could choose it as the nonorthogonal projection introduced in \cite{Hyvoenen2009}. This suggests that our approach can be extended to real measurement data.

Our main result, Proposition \ref{prop:convergence_single}, shows that the error between the partial ND map and the full ND map depends linearly on the length of the missing boundary. This result assumes knowledge of the measurement on the whole boundary. If one restricts the measurement domain to the same area as the current input, one has to extend the data to the full boundary. This is done via an extrapolation procedure on the difference data, such that the linear error estimate is preserved. Thus, the error in the data is not governed by restriction of the measurement domain, but by restricting the input current. 

Computation of the CGO solutions and the scattering transform is done directly from the ND map by utilizing a Born approximation. The approximate reconstructions are then computed from the truncated scattering transform. Further, if the choice of partial-boundary map permits to establish an operator estimate for the partial ND map, we can prove that the reconstruction error from the partial ND map (compared to full ND map) has linear dependence as well.

In the computational section we have verified that the error estimates hold asymptotically when the missing boundary is small. For finite data the reconstruction error holds also for partial-boundary maps that do not provide an operator estimate. It is of particular interest to note that the features of the reconstructions depend heavily on the choice of partial-boundary map. This suggests that more research can be done to optimize the choice of boundary currents for the partial data problem.

There are several approaches one could imagine to improve direct reconstruction from the partial ND map. One possibility is to combine data from measurements with several choices of partial boundary map as illustrated in Figure \ref{fig:combinedData}. For the chest reconstructions in Figure \ref{fig:HnL_FEM_recons} one can see that some features seem to disappear from the image. This could be overcome by reconstructing on a slightly larger mesh and outlining the original domain in the reconstruction.

\section*{Acknowledgements}
This work was supported by the Academy of Finland through the Finnish
Centre of Excellence in Inverse Problems Research 2012–2017, decision number
250215. AH and MS were partially supported by FiDiPro project of the Academy of Finland,
decision number 263235.

\bibliographystyle{siam}
\bibliography{Inverse_problems_references_2016}

\end{document}